\documentclass[preprint]{amsart}
\usepackage{amsmath, amsfonts, amssymb, amsbsy, bigstrut, graphicx, enumerate,  upref, longtable, comment, bbm, physics, bm}
\usepackage{scalerel}
\usepackage{pstricks,csquotes}
\usepackage[breaklinks]{hyperref}

\hypersetup{
colorlinks=true,%
citecolor=blue,%
filecolor=blue,%
linkcolor=red,%
urlcolor=blue
}

\usepackage[numbers, sort&compress]{natbib} 
\usepackage{tikz}
\usetikzlibrary{decorations.markings}
\usepackage[noabbrev,capitalize]{cleveref}
\usepackage{calc}
\usepackage{longtable}
\usepackage{accents}

\newcommand{\E}{\mathbb{E}}
\newcommand{\R}{\mathbb{R}}

\newcommand{\argmin}{\operatorname{argmin}}
\newcommand{\xsph}{x^{\phi}}
\newcommand{\xsut}{x^{u_t}}
\newcommand{\emn}{\mathrm{EOT}(\mu,\nu;\vep)}

\newcommand{\px}{p_{\scaleto{X}{3.5pt}}}
\newcommand{\py}{p_{\scaleto{Y}{3.5pt}}}

\newcommand{\et}{\rho_t}
\newcommand{\mcx}{\mathcal{X}}

\newcommand{\bmd}{\Delta}
\newcommand{\trc}{\mathrm{Trace}}
\newcommand{\mcI}{\mathcal{I}}
\newcommand{\mcJ}{\mathcal{J}}
\newcommand{\mcD}{\mathcal{D}}
\newcommand{\mc}{\mathcal{C}^{\vep}}

\newcommand{\secphx}{\frac{\partial\xsph}{\partial x\hfill}}
\newcommand{\secphxil}{\frac{\partial\xsph_i}{\partial x_{\ell}}}
\newcommand{\secphxim}{\frac{\partial \xsph_i}{\partial x_{m}}}

\newcommand{\ldet}{\log{\mathrm{det}}}

\newcommand{\secph}{\frac{\partial x\hfill}{\partial\xsph}}
\newcommand{\secphlj}{\frac{\partial x_{\ell}}{\partial\xsph_j}}
\newcommand{\secphjl}{\frac{\partial x_{j}}{\partial\xsph_{\ell}}}
\newcommand{\secphmk}{\frac{\partial x_{m}}{\partial\xsph_k}}
\newcommand{\secphmi}{\frac{\partial x_{m}}{\partial\xsph_i}}

\newcommand{\lmn}{\lambda_{\mathrm{min}}}
\newcommand{\lmx}{\lambda_{\mathrm{max}}}
\newcommand{\mI}{KL}
\newcommand{\lsi}[1]{C_{#1}}
\newcommand{\ptac}{\mathcal{P}_2^{\mathrm{ac}}(\R^d)}
\newcommand{\sfe}{\sigma_F^2(t)}
\newcommand{\sse}{\sigma_S^2(t)}

\newcommand{\diffcont}{\mathcal{C}}
\newcommand{\iprod}[1]{\left\langle #1 \right\rangle}
\newcommand{\probspace}{\mathcal{P}_2\left( \R^d\right)}
\newcommand{\tanspace}{\mathrm{Tan}}
\newcommand{\ltwo}{\mathbf{L}^2}
\newcommand{\wass}{\mathbb{W}}

\newcommand{\vep}{\varepsilon}
\newcommand{\gpp}[1]{\tilde{\gamma}_{#1}^{\vep}}
\newcommand{\gnp}[1]{\gamma_{#1}^{\vep}}
\newcommand{\mk}[1]{\rho_{#1}^{\vep}}
\newcommand{\nk}[1]{\eta_{#1}^{\vep}}
\newcommand{\gvp}{\gamma^{\vep}}
\newcommand{\SP}[1]{\textcolor{purple}{SP:#1}}
\newcommand{\ND}[1]{\textcolor{blue}{#1}}
\newcommand{\wt}[3]{\wass_2^{#1}({#2},{#3})}
\newcommand{\dlt}[4]{d^{#1}_{\mathrm{LOT},{#2}}({#3},{#4})}
\newcommand{\lot}[1]{\mathrm{LOT}^2_{#1}}
\newcommand{\opV}{\mathcal{V}^{\vep}}
\newcommand{\opU}{\mathcal{U}^{\vep}}
\newcommand{\opS}{\mathcal{S}^{\vep}}
\newcommand{\rv}{\rho^{\vep}}

\newcommand{\NN}{\mathbb{N}}
\newcommand{\rr}{\mathbb{R}}

\newcommand{\KL}[2]{\mathrm{KL}(#1 \parallel  #2)}
\newcommand{\mgf}[3]{\mathcal{M}_{#1}[#2:#3]}
\newcommand{\mfR}[3]{\mathrm{Rem}\big[#1\big](#2;#3)}

\newcommand{\fil}{\mathcal{F}}
\newcommand{\opP}{\bm{P}^{\vep}}
\newcommand{\opQ}{\bm{Q}^{\vep}}
\newcommand{\tpP}[1]{\bm{\tilde{P}}_{#1\vep}}
\newcommand{\tpQ}[1]{\bm{\tilde{Q}}_{#1\vep}}
\newcommand{\opR}{\bm{R}^{\vep}}
\newcommand{\fv}{\mathrm{FV}}
\newcommand{\hs}{\mathrm{HS}}

\allowdisplaybreaks

\newtheorem{thm}{Theorem}
\newtheorem{lmm}[thm]{Lemma}

\newtheorem{prop}[thm]{Proposition}
\newtheorem{defn}[thm]{Definition}
\newtheorem{assm}{Assumption}

\theoremstyle{definition}
\newtheorem{remark}[thm]{Remark}
\newtheorem{ex}[thm]{Example}

\numberwithin{thm}{section}
\numberwithin{assm}{section}
\numberwithin{equation}{section}

\begin{document}

\title[Wasserstein Mirror Gradient Flows]{Wasserstein Mirror Gradient Flow as the limit of the Sinkhorn algorithm}

\author{Nabarun Deb}
\address{Nabarun Deb\\ Department of Mathematics \\ University of British Columbia\\ Vancouver, Canada\\ {Email: ndeb@math.ubc.ca}}
\author{Young-Heon Kim}
\address{Young-Heon Kim\\ Department of Mathematics \\ University of British Columbia\\ Vancouver, Canada\\ {Email: yhkim@math.ubc.ca}}
\author{Soumik Pal}
\address{Soumik Pal\\ Department of Mathematics \\ University of Washington\\ Seattle WA 98195, USA\\ {Email: soumik@uw.edu}}
\author{Geoffrey Schiebinger}
\address{Geoffrey Schiebinger\\ Department of Mathematics \\ University of British Columbia\\ Vancouver, Canada\\ {Email: geoff@math.ubc.ca}}

\begin{abstract}
 	We prove that the sequence of marginals obtained from the iterations of the Sinkhorn algorithm or the iterative proportional fitting procedure (IPFP) on joint densities, converges to an absolutely continuous curve on the $2$-Wasserstein space, as the regularization parameter $\vep$ goes to zero and the number of iterations is scaled as $1/\vep$ (and other technical assumptions). This limit, which we call the Sinkhorn flow, is an example of a Wasserstein mirror gradient flow, a concept we introduce here inspired by the well-known Euclidean mirror gradient flows. In the case of Sinkhorn, the gradient is that of the relative entropy functional with respect to one of the marginals and the mirror is half of the squared Wasserstein distance functional from the other marginal. Interestingly, the norm of the velocity field of this flow can be interpreted as the metric derivative with respect to the linearized optimal transport (LOT) distance. An equivalent description of this flow is provided by the parabolic Monge-Amp\`{e}re PDE whose connection to the Sinkhorn algorithm was noticed by Berman (2020). We derive conditions for exponential convergence for this limiting flow. We also construct a Mckean-Vlasov diffusion whose  marginal distributions follow the Sinkhorn flow.
\end{abstract}

	

\thanks{Thanks to PIMS Kantorovich Initiative for facilitating this collaboration supported through a PIMS PRN and the NSF Infrastructure grant DMS 2133244.  Pal is supported by NSF grants DMS-2052239 and DMS-2134012.  Kim and Schiebinger are supported in part by New Frontier Research Funds (NFRF) of Canada as well as NSERC Discovery grant. Schiebinger is also supported by a MSHR Scholar Award, a CASI from the Burroughs Wellcome Fund and a CIHR Project Grant. Nabarun Deb is supported by the PIMS PRN postdoc fellowship.}

\maketitle

\section{Introduction}\label{sec:intro}

We study the scaling limit of the sequence of iterates coming from the celebrated Sinkhorn (or IPFP) algorithm \cite{sinkhorn1967diagonal} as the regularization parameter $\vep\to 0+$ and the number of iterations scales like $\lceil{1/\vep}\rceil$. Given two probability measures $\mu$ and $\nu$ on $\R^d$ with finite second moments, the Sinkhorn algorithm aims to solve the entropy regularized  optimal transport problem:
\begin{equation}\label{eq:eot}
\emn:=\min_{\pi\in \Pi(\mu,\nu)} \left[\frac{1}{2}\int \lVert x-y\rVert^2\,\,d\pi(x,y)+\vep\KL{\pi}{\mu\otimes \nu}\right],
\end{equation}
where $\Pi(\mu,\nu)$ denotes the space of probability measures on $\R^d\times\R^d$ with marginals $\mu$ and $\nu$. Here
\begin{equation}\label{eq:KLdef}
\KL{\pi}{\mu\otimes\nu}=\int \log\left(\frac{d\pi}{d(\mu\otimes\nu)}(x,y)\right)\,d\pi(x,y)
\end{equation}
denotes the standard Kullback-Leibler (KL) divergence and $\vep>0$ is called the regularization parameter. The $\emn$ problem above has a wide array of applications and hence, has attracted significant attention in probability, statistics, and machine learning; see, for example,  \cite{Guillaume2020,Chen2016,Cominetti1994,Christian2012,peyre2019computational,schiebinger2019optimal} and the references therein. The unregularized version of \eqref{eq:eot}, which theoretically corresponds to setting $\vep=0$ in $\emn$, is called the optimal transport problem~ \cite{Kantorovitch1942,monge1781memoire,Villani2003}, and it has thriving applications in large-scale problems~\cite{arjovsky2017wasserstein,rubner2000earth}. A computationally convenient way of addressing the optimal transport problem is to solve $\emn$ for small values of the regularization parameter $\vep>0$ \cite{cuturi2013sinkhorn}. 
Consequently, the recent years have witnessed a rich and growing body of literature aimed at understanding the convergence and stability properties of $\emn$ as $\vep\to 0$;  see~\cite{Bernton2022,carlier2022convergence,chiarini2022gradient,Conforti2021,eckstein2022convergence,Gigli2018,Christian2012,Mikami2004,Nutz2022,pal2019difference} and the references therein.  

\begin{figure}
    \centering
    \includegraphics[width=0.40\textwidth]{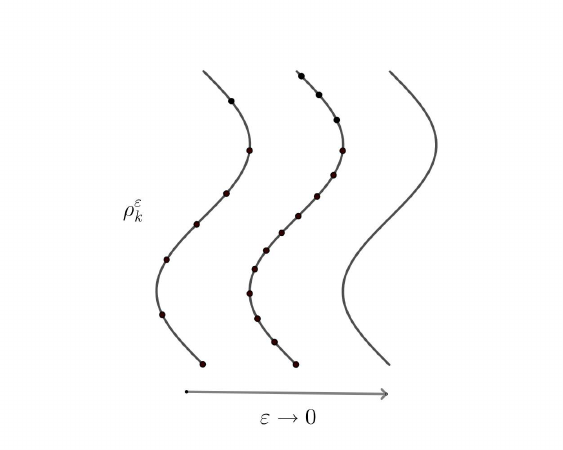}
    \caption{Convergence of the Sinkhorn iterates $\left(\mk{k},\; k\in \mathbb{N}\right)$ to an absolutely continuous curve (the \emph{Sinkhorn flow}) in the $2$-Wasserstein space, with $k=O(1/\vep)$, as $\vep\to 0$. Here, for a given $\vep>0$, we view the points on the corresponding curves as a discretization of an absolutely continuous curve with $O(1/\vep)$ points.}
    \label{fig:curve}
\end{figure}

One of the major breakthroughs that led to the popularity of the $\emn$ problem is that it can be computed efficiently via the Sinkhorn algorithm, also called IPFP or the Iterative Proportional Fitting Procedure, (see \cite{Chen2016,csiszar1975divergence,franklin1989scaling,ruschendorf1995convergence}), which can in turn be parallelized easily~\cite{cuturi2013sinkhorn,knight2008sinkhorn}. The algorithm proceeds as follows: start with a probability measure $\gnp{0}$ on $\R^d\times\R^d$, and a function $u_0:\R^d\to\R$, given by:
\[
\frac{d\gnp{0}}{d(\mu\otimes\nu)}(x,y)\propto \exp\left(\frac{1}{\vep}\langle x,y\rangle-\frac{1}{\vep}u_0(x)\right),\quad x,y\in\R^d.
\]
For a joint probability distribution $\gamma$, let $\px \gamma$ and $\py\gamma$ denote the $X$ and $Y$ marginal distributions under $\gamma$. Then, for $k\ge 1$, the Sinkhorn algorithm  proceeds iteratively by defining a sequence of pairs of joint distributions $\left(\gpp{k}, \gnp{k} \right)$ given by 
\begin{equation}\label{eq:sinkupdt}
    \frac{d\gpp{k}\hfill}{d\gnp{k-1}}(x,y)=\frac{d\mu}{d(\px\gnp{k-1})}(x), \quad \text{and} \;\frac{d\gnp{k}}{d\gpp{k}}(x,y)=\frac{d\nu}{d(\py\gpp{k})}(y).
\end{equation}

It is known and easily verifiable that each $\gnp{k}$ an $\gpp{k}$ solves the entropic problem in \eqref{eq:eot} between its own marginals. Further, it is easy to check that $\px\gpp{k}=\mu$ and $\py\gnp{k}=\nu$. Therefore, the Sinkhorn algorithm alternately attains the target marginals on the two coordinates, and is hence completely characterized by the respective sequence of opposite marginals $\mk{k}:=\px\gnp{k}$ and $\nk{k}:=\py\gpp{k}$. Our aim here is to study the limit of $\mk{k}$ as $k\to\infty$ and $\vep\to 0$ in an appropriate sense (noting that the limit of $\nk{k}$ can be studied similarly).

While it is known \cite{ruschendorf1995convergence} that the sequence $\left(\mk{k},\; k \in \NN \right)$ converges to the target marginal $\mu$ as $k\rightarrow \infty$ for fixed $\vep>0$, the corresponding limiting objects when $\vep\to 0$ and $k\to\infty$ simultaneously are not well understood. 
As $\vep \rightarrow 0+$, it is intuitive that $\mk{k}$ and $\mk{k+1}$ gets increasingly closer to each other. Assuming that all these measures have finite second moments, it makes sense to consider $\left( \mk{k},\; k \in \NN\right)$ as a sequence in the $2$-Wasserstein space, $\wass_2$, (defined below in \eqref{eq:2wass}) and ask if, as $\vep \rightarrow 0+$ and the number of iterations are properly scaled, this sequence of points in $\wass_2$ converges to an absolutely continuous (AC) curve (illustrated in Figure~\ref{fig:curve}). Every AC curve in the Wasserstein space is the solution of a continuity equation with a velocity field \cite[Chapter 8]{ambrosio2005gradient}. Therefore, characterizing the velocity field of this limiting AC curve gives approximate geometric properties of $\left(\mk{k},\; k \in \NN \right)$ as $\vep \rightarrow 0+$. This is the goal of this paper.

\begin{figure}
\centering
\begin{tikzpicture}
\begin{scope}[thick,decoration={
    markings,
    mark=at position 0.999 with {\arrow{>}}}
    ] 
\draw[thick, postaction={decorate}] (0,0) .. controls (1,3) and (3,0) .. (4,2);
\end{scope}
\begin{scope}[thick,decoration={
    markings,
    mark=at position 0.5 with {\arrow{>}}}
    ] 
    \draw[postaction={decorate}] (0.5,0.99) -- (1,3);
    \draw[postaction={decorate}] (2,1.355) -- (1,3);
    \draw[postaction={decorate}] (3.6,1.5) -- (1,3);
\end{scope}
\node at (0.9,3.1) {$e^{-g}$};
\node at (4.3,2.3) {$e^{-f}$};
\node at (0.2,0.0) {$\rho_0$};
\node at (0.6,0.8) {$\rho_1$};
\node at (2,1.2) {$\rho_2$};
\node at (3.8,1.4) {$\rho_3$};
\node at (0.3,1.7) {$\nabla u_1$};
\node at (1.4,1.8) {$\nabla u_2$};
\node at (2.5,2.5) {$\nabla u_3$};

\qquad\qquad

\begin{scope}[thick,decoration={
    markings,
    mark=at position 0.999 with {\arrow{>}}}
    ] 
\draw[thick, postaction={decorate}] (5,0) .. controls (6,3) and (8,0) .. (9,2);
\end{scope}
\begin{scope}[thick,decoration={
    markings,
    mark=at position 0.5 with {\arrow{<}}}
    ] 
    \draw[postaction={decorate}] (5.5,0.99) -- (6,3);
    \draw[postaction={decorate}] (7,1.355) -- (6,3);
    \draw[postaction={decorate}] (8.6,1.5) -- (6,3);
    \draw[postaction={decorate}]
    (7, 1.355) -- (5.5, 0.99);
\end{scope}
\node at (5.9,3.1) {$e^{-g}$};
\node at (9.3,2.3) {$e^{-f}$};
\node at (5.2,0.0) {$\rho_0$};
\node at (5.6,0.8) {$\rho_1$};
\node at (7,1.2) {$\rho_2$};
\node at (8.8,1.4) {$\rho_3$};
\node at (5.3,1.7) {$\nabla w_1$};
\node at (6.3,1.8) {$\nabla w_2$};
\node at (7.5,2.5) {$\nabla w_3$};
\node at (6.25,0.8) {$v_1$};
\end{tikzpicture}
\caption{Evolution of the Sinkhorn flow $(\rho_t,t\ge 0)$. On the left,  $\nabla u_t$ is the Brenier map transporting $\rho_t$ to $e^{-g}$. $(u_t,t\ge 0)$ satisfies the parabolic Monge-Amp\`{e}re PDE: $\frac{\partial}{\partial t} \nabla u_t= \nabla_{\mathbb{W}} \KL{\rho_t}{e^{-f}}$. See Theorem \ref{thm:infexist}. On the right, $w_t=u_t^*$ is the corresponding family of convex conjugates. $\nabla w_t=(\nabla u_t)^{-1}$ transports $e^{-g}$ to $\rho_t$. For $s < t$, $t-s \approx 0$, the difference $\nabla w_t - \nabla w_s$ is approximately $v_s \circ \nabla w_s$, where $v_s$ is the velocity of the Sinkhorn flow at time $s$. This relates the velocity field with Linearized Optimal Transport. See \cref{prop:metderlin}.} 
\label{fig:evolution}
\end{figure}
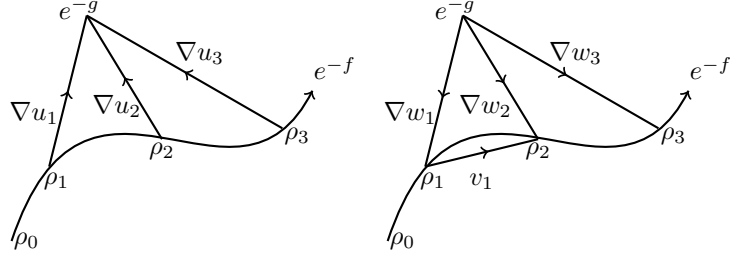

Under appropriate assumptions, in \cref{thm:inftheo}, we show that, indeed, as $\vep \rightarrow 0+$ with $k$ scaling as $O(1/\vep)$, the sequence $\left(\mk{k},\; k \in \NN \right)$ converges to an absolutely continuous curve in the Wasserstein space called the \emph{Sinkhorn flow}, characterized by a continuity equation that we call the \textit{Sinkhorn PDE}. We explicitly identify the velocity vector field for this evolution PDE.

We further observe in \cref{sec:wassmirrorflowformal} that the Sinkhorn flow is one member of a particular class of absolutely continuous curves in the Wasserstein space which we introduce and refer to as  \emph{Wasserstein mirror gradient flows}, or \emph{Wasserstein mirror flows} for short (see \cref{sec:othexamp} for examples). The concept of mirror gradient flows, although popular in the Euclidean setting \cite{nemirovskii83}, appears to be unexplored in the Wasserstein space. In fact, the Sinkhorn flow can be viewed as the Wasserstein mirror flow of the KL divergence functional. This is different from the solution to the Fokker-Planck PDE \cite[Section 8.3]{santambrogio2015optimal}, which corresponds to the usual Wasserstein gradient flow of the KL divergence functional. Now, it is widely known that the marginal distributions of the Langevin diffusion solve the Fokker-Planck PDE. In the same vein, we also obtain a new Mckean-Vlasov \cite{mckean75} diffusion process (see \cref{sec:diffmirr} below), which we call the \emph{Sinkhorn diffusion}, whose marginal distribution exactly tracks the evolution of the Sinkhorn PDE mentioned above. 

A lot of the existing literature on the evolution of the Sinkhorn algorithm \eqref{eq:sinkupdt} focuses on obtaining rates of convergence for a fixed $\vep$ as a function of $k$. We refer the reader to \cite{Ghosal2022,Nutz2023,ruschendorf1995convergence} for some qualitative results and \cite{deligiannidis2021quantitative,EcksteinNutz2022,ghosal2022convergence} for quantitative ones. However, the constants appearing in these bounds blow up as $\vep \downarrow 0$. In contrast, we avoid this pitfall by scaling our iteration as $O(1/\vep)$, then taking a scaling limit and then analyzing the rate of convergence of the scaling limit using geometric ideas.
See \cref{lem:expcon} for conditions under which there is exponentially fast convergence. This knowledge can now be transferred to small but finite $\vep$ using our \emph{quantitative rates}, in the squared Wasserstein distance, for the convergence of $\mk{k}$ to the marginals of the limiting flow (see \cref{thm:convergence}).



Two other lines of work on the Sinkhorn algorithm are more closely related to ours. In \cite{leger2021gradient}, the author shows that the Sinkhorn algorithm \eqref{eq:sinkupdt} can be viewed as a mirror descent algorithm on the space of probability measures for fixed $k$ and $\vep$. In particular, \cite[Corollary 1]{leger2021gradient} leverages this connection to show that $\mk{k}$ converges in KL to $\mu$ at a $O(k\vep)$ rate. However, this result cannot identify our limiting dynamics as $k\vep\to t$ and $t\ge 0$. In a different direction,  \cite{berman2020} studies the Sinkhorn algorithm \eqref{eq:sinkupdt} in the same scaling limit as we do, and proves that the evolution of the underlying Sinkhorn potentials (defined below in \eqref{eq:twostepit}) converge to the solutions of a parabolic Monge-Amp\`ere (PMA) PDE (see \eqref{eq:pma} below). However, while Berman's picture focuses on potentials, our focus is on  the convergence of the marginal measures. The proof of convergence of the evolution of the measures $\mk{k}$s to the Sinkhorn flow requires a higher order asymptotic analysis than Berman's proof in \cite{berman2020}. The main difficulty is that the convergence of potentials, as shown by Berman, does not imply the convergence of the flow. Rather, one needs to argue that discrete time difference of the potentials converge to the time derivative in the continuum. 
 Although this required us to develop  different techniques than both Berman and L\'{e}ger, our ideas are inspired from their work.

As we argue in Section \ref{sec:wassmirrorflowformal}, every Wasserstein mirror gradient flow can be described by a pair of equivalent PDEs, one describing the time evolution of the corresponding potentials and the other describing the time evolution of the marginal measures as a continuity equation. In our context, Berman's PMA tracks the evolution of the Sinkhorn potentials, and equivalently, the Sinkhorn flow which we introduce in this paper, tracks the evolution of the marginal distributions. Thus the Sinkhorn flow (see \cref{thm:inftheo} below) also yields a geometric description of the PMA which may be of independent interest.   

\subsection{Notation}\label{sec:nota}

For convenience of the reader, we list some generic notation in the form of a table at the end of the paper; see \cref{tab:table3}. We also refer the reader to \cite[Chapter 10]{ambrosio2005gradient} for a background on Otto calculus which will be used throughout the paper. In this section let us focus on the more specialized notation.

For a function $u:\R^d\to\R \cup \{\infty\}$, its Legendre transform is given by
\[
u^*(y):=\sup_{x\in\R^d}(\iprod{x,y} -u(x)),\quad\quad y\in\R^d.
\]
Now let $u$ be a differentiable convex function such that $\nabla u$ is a diffeomorphism on $\R^d$. That is $\nabla u: \R^d \rightarrow \R^d$ is a differentiable map such that its inverse $\left( \nabla u\right)^{-1}= \nabla u^*$ is also differentiable. This allows us to think of $\R^d$ as a manifold with two global coordinate charts $x\mapsto x$ and $x\mapsto \nabla u(x)$.

\begin{defn}[Mirror coordinates]\label{def:pdcor}
Given $x\in \mcx$, define $x^u:=\nabla u(x)$. We will call $x^u$ the mirror coordinate for $x$ with respect to $u$.  
\end{defn}

We will refer to the Hessian matrix $\nabla^2 u(x)$ by the notation $\frac{\partial x^u}{\partial x\hfill}$. Note that $\left( x^u \right)^{u^*}=x$ and $\left(y^{u^*} \right)^u=y$. Thus (see, for example, \cite[Lemma 2.3]{berman2020}), if $y=x^u$,
\begin{equation}\label{eq:conjrel}
\frac{\partial x\hfill}{\partial x^u}:=\frac{\partial y^{u^*}}{\partial y\hfill}=\nabla^2u^*\left( x^u\right)= \left(\frac{\partial x^u}{\partial x\hfill} \right)^{-1}. 
\end{equation}
We will utilize this convenient notation throughout. 

\begin{remark}
    As a matter of convention, $\frac{\partial f\hfill}{\partial x^{u}} (x^{u})$ will denote the derivative of $f$ evaluated at $x^{u}$. This is different from taking the derivative of the composition function $f(x^{u})$ which we denote by $\frac{\partial}{\partial x}(f(x^{u}))$. In particular, 
    \begin{align}\label{eq:notcal}
    \frac{\partial}{\partial x}(f(x^{u}))=\left\langle  \frac{\partial x^{u}}{\partial x\hfill},\frac{\partial f\hfill}{\partial x^{u}}(x^{u})\right\rangle.
    \end{align}
\end{remark}


Let $\ptac$ denote the space of probability measures on $\R^d$ with positive densities and finite second moments. We will need three notions of distances between two probability measures, say $\mu_1$ and $\mu_2$, which we write as $\KL{\mu_1}{\mu_2}$, $\wt{}{\mu_1}{\mu_2}$, and $\dlt{}{\sigma}{\mu_1}{\mu_2}$ that denote the KL divergence defined in \eqref{eq:KLdef}, the $2$-Wasserstein distance (see \cite[Equation 4]{Villani2003}, \cite[Equation 7.1.1]{ambrosio2005gradient}), and the linearized optimal transport distance (see \cite[Equations 2 and 3]{Wang2013}, \cite{cai2020linearized,moosmuller2023linear}) with respect to a reference probability distribution $\sigma\in \ptac$. The latter two are defined below.

 The squared $2$-Wasserstein distance is defined as
 \begin{equation}\label{eq:2wass}
 \wass_2^2(\mu_1,\mu_2):=\min_{\gamma\in \Pi(\mu_1,\mu_2)} \int \lVert x-y\rVert^2\,d\gamma(x,y),
 \end{equation}
 where, as mentioned before, $\Pi(\mu_1,\mu_2)$ is the space of probability measures on $\R^d\times\R^d$ which preserves the marginals at $\mu_1$ and $\mu_2$. The metric space $\left( \ptac, \wass_2\right)$ will be often denoted (by an abuse of notation) also as $\wass_2$ space. The squared linear optimal transport distance is given by:

 \begin{equation}\label{eq:2linot}
 \dlt{2}{\sigma}{\mu_1}{\mu_2}:=\int \lVert T_{\mu_1}(y)-T_{\mu_2}(y)\rVert^2 \,d\sigma(y),
 \end{equation}
 where $(T_{\mu_1})_{\#}\sigma=\mu_1$ and $(T_{\mu_2})_{\#}\sigma=\mu_2$ are the two optimal transport maps.


\subsection{Informal description of main results}\label{sec:formalres} 
Let us outline our main results which make the connection between the Sinkhorn algorithm and the notion of mirror descent more explicit. 
Recall that our state space is $\R^d$ for some $d\ge 1$.
Let $\mu$ and $\nu$ be two Borel probability measures in $\ptac$ (absolutely continuous with finite second moments). Let their densities be given by $e^{-f(x)}$ and $e^{-g(y)}$ respectively. For a full list of assumptions, see \cref{asn:solcon}. Recall the definition of absolutely continuous curves on the $2$-Wasserstein space and its correspondence with solutions of the continuity equations (see \cite[Chapter 8]{ambrosio2005gradient}).

\begin{thm}[Existence of Sinkhorn flow (Informal)]\label{thm:infexist}
  There exists an absolutely continuous curve on the $\wass_2$ space which satisfies the continuity equation $\frac{\partial}{\partial t}\rho_t+\div(\rho_t v_t)=0$ with 
  \begin{equation}\label{eq:velocity}
 v_t(x)= -\frac{\partial\hfill}{\partial x^{u_t}}((f+\log{\rho_t})(x)) =
 -\left( \nabla^2 u_t(x)\right)^{-1}\nabla_{\mathbb{W}}\KL{\rho_t}{e^{-f}}(x).
\end{equation}
Here,  $\nabla u_t$ is the Brenier map transporting $\rho_t$ to $\nu$, i.e., $u_t$ is convex and $(\nabla u_t)_{\#}\rho_t=e^{-g}$ and $\left( \nabla^2 u_t(x)\right)^{-1}$ is the inverse of the Hessian matrix of $u_t$ at $x$. Here $\nabla_{\mathbb{W}}\KL{\rho_t}{e^{-f}}(x)$ refers to the usual Wasserstein gradient of the functional $\rho \mapsto \KL{\rho}{e^{-f}}$ evaluated at $x$ (see \cite[Chapter 10]{ambrosio2005gradient} for details on differential calculus on the $\wass_2$ space). We will call this curve the Sinkhorn flow. 

The Sinkhorn flow can be viewed as the Wasserstein mirror gradient flow of the KL functional $\KL{\cdot}{e^{-f}}$ with the mirror map given by half of the squared Wasserstein distance $\wass_2^2(\cdot, e^{-g})$. In particular, the family of Brenier maps satisfies the PDE 
\begin{align}\label{eq:newPDE}
\frac{\partial}{\partial t} \nabla u_t(x)= \nabla_{\mathbb{W}}\KL{\rho_t}{e^{-f}}(x). 
\end{align}
This is nothing but the parabolic Monge-Amp\`{e}re PDE, see  \eqref{eq:pma}, written slightly differently.
\end{thm}
\noindent For a pictorial description of the Sinkhorn flow, we refer the reader to Figure \ref{fig:evolution}. A precise version of the above statement can be found in \cref{thm:existlin}. Also, the notion of Wasserstein mirror gradient flows is introduced in \cref{sec:wassmirrorflowformal}. Very roughly, a ``mirror'' associates a dual pair of ``coordinate systems'' to each point on a manifold. Then, given a function $F$, the mirror gradient flow of $F$ is a curve on the manifold that admits two equivalent descriptions. They both describe the time derivative of one coordinate to be equal to the negative gradient of $F$ with respect to its mirror. When the mirror is the identity, we recover our usual gradient flow. We extend this Euclidean notion to the Wasserstein space by an analogous formalism that views the $2$-Wasserstein space as an infinite-dimensional Riemannian manifold as in \cite{Otto_2001}.

We note that $v_t$ as in \eqref{eq:velocity} is not a gradient with respect to the canonical coordinate system (as is the case with the usual Wasserstein gradient flows), but instead, it is a gradient with respect to the mirror coordinate system (see \cref{def:pdcor}). For the usual  Wasserstein gradient flows, 
it is well-known (see \cite[Proposition 8.4.6]{ambrosio2005gradient}) that there exists a velocity field $(v_t)_{t\ge 0}$ (gradient of a function with respect to the canonical coordinate system) of an absolutely continuous curve $(\rho_t)_{t\ge 0}$ which satisfies:
$$\limsup_{\delta\to 0} \delta^{-1}\wass_2\big(\rho_{t+\delta},\rho_t\big)= \lVert v_t\rVert_{L^2(\rho_t)}.$$
In other words, for Wasserstein gradient flows, $\lVert v_t\rVert_{L^2(\rho_t)}$ can be interpreted as the metric derivative of the curve $(\rho_t)_{t\ge 0}$ at time $t$, with respect to the $\wass_2$ distance. This interpretation may not hold for the Sinkhorn flow as in \cref{thm:inftheo} because $v_t$ is not a usual gradient. In \cref{prop:metderlin}, we therefore provide an alternate interpretation of $v_t$ as in \eqref{eq:velocity}. We show that $\lVert v_t\rVert_{L^2(\rho_t)}$ is the metric derivative of the curve $\rho_t$ but with respect to the linearized optimal transport distance (see~\eqref{eq:2linot}) instead of the usual $\wass_2$ distance (see \eqref{eq:2wass}), i.e.,  
$$\lim_{\delta\to 0} \delta^{-1}\dlt{}{e^{-g}}{\rho_{t+\delta}}{\rho_t}=\lVert v_t\rVert_{L^2(\rho_t)},$$
where $v_t$ and $\rho_t$ are defined as in \cref{thm:inftheo}. This provides an interesting connection between the Sinkhorn flow and the LOT literature \cite{Wang2013}. We refer the reader to Figure \ref{fig:evolution} for an illustration of this connection.

Given a flow in the space of probability measures, it is natural to ask if there is a stochastic process, preferably Markov, that generates the same family of marginals. In \cref{sec:diffmirr}, we answer this question in the affirmative.
\begin{thm}[Sinkhorn diffusion (Informal)]\label{thm:diffexist}
    Let $h_t:=-\log\rho_t$ where $\rho_t$ is a solution of the Sinkhorn PDE/flow in \cref{thm:infexist}. Let $B$ denote a standard $d$-dimensional Brownian motion and consider the following SDE given $X_0\sim \rho_0$:
    $$ dX_t=\left(-\frac{\partial f\hfill}{\partial \xsut}(X_t)-\frac{\partial g\hfill}{\partial\xsut}\left(X_t^{u_t}\right)+\frac{\partial h_t\hfill}{\partial \xsut}(X_t)\right)\,dt+\sqrt{2\frac{\partial X_t\hfill}{\partial X_t^{u_t}}}dB_t.$$
    Here, as before, $\nabla u_t$ is the convex Brenier map transporting $\rho_t$ to $e^{-g}$.
    Then, the marginals of the process $(X_t)_{t\ge 0}$ evolve according to the Sinkhorn flow in \cref{thm:infexist}. We will call this SDE the \emph{Sinkhorn diffusion}.
\end{thm}

A formal version of this result is provided in \cref{thm:existpropX}. The Sinkhorn diffusion is motivated from a natural Markov chain embedded in the Sinkhorn algorithm \eqref{eq:sinkupdt} defined in \cref{prop:mchn}. Note that when $u_t$ is free of $t$, i.e., $u_t=u$, then $h_t=g$ for all $t$. Therefore, the Sinkhorn diffusion reduces to the mirror Langevin diffusion, first introduced in \cite{zhang2020wasserstein}, and further analyzed in \cite{ahn2021efficient,chewi2020exponential}. In the further special case where $f=g$, then $x^{u_t}=x$ for all $t$ and the Sinkhorn diffusion reduces to the standard Langevin diffusion (see \cite{durmus2019analysis,jordan1998variational}). Let us also point out that our definition of Wasserstein mirror flow is different from the similarly named ``mirrored Wasserstein gradient flow" introduced in \cite{sharrock2023learning}. 

Finally, we connect the Sinkhorn flow with the Sinkhorn algorithm by proving that the $X$-marginals $\mk{k}$s from the algorithm converge to the marginals of the Sinkhorn flow, under an appropriate scaling.

\begin{thm}[Convergence as $\vep\downarrow 0$ (Informal)]\label{thm:inftheo}
Under appropriate assumptions on $\mu,\nu$ and $u_0$, given any $t>0$, we have $\wass_2^2(\rho_{\lfloor t/\vep\rfloor}^{\vep},\rho_t)=O(\vep)$.
\end{thm}

A formal version of this result is presented in \cref{thm:convergence}. 

\begin{remark}[Comparison with Fokker-Planck PDE]\label{rem:fplank}
 Recall from \cref{thm:diffexist} that $h_t=-\log{\rho_t}$. With this notation, the Sinkhorn PDE simplifies to 
 \begin{equation*}
 \frac{d}{dt}h_t(x)+\left\langle \frac{\partial g\hfill}{\partial x^{u_t}}(x^{u_t}),\frac{\partial}{\partial x}(h_t-f)(x)\right\rangle+\trc\left(\left(\frac{\partial x\hfill}{\partial x^{u_t}}\right)\nabla^2 (f-h_t)(x)\right)=0,
 \end{equation*}
 which has a similar form to the usual Fokker-Planck PDE \cite[Chapter 8.3]{santambrogio2015optimal}, given by:
 $$\frac{\partial}{\partial t} h_t(x)+\left\langle \frac{\partial}{\partial x} h_t(x),\frac{\partial}{\partial x} (h_t-f)(x) \right\rangle + \bmd{(f-h_t)(x)}=0.$$
 In particular, the velocity field for usual Fokker-Planck PDE is given by 
 \begin{equation}\label{eq:velfp}
 v_t(x)=-\frac{\partial}{\partial x}(f+\log{\rho_t})(x),
 \end{equation}
 which is the same as in~\eqref{eq:velocity} with $\frac{\partial\hfill}{\partial \xsut}$ replaced by $\frac{\partial}{\partial x}$.  
 \end{remark}

 Based on the above remark, it is natural to ask: (a) Can the Sinkhorn PDE be the same as the Fokker-Planck PDE ? (b) When they are different, can the Sinkhorn PDE converge faster as $t\to\infty$ ?
 Consider two simple examples below.
  \begin{ex}[Gaussian location]\label{ex:loc}
     Suppose $\mu= N(0,1)$ and $\nu= N(\theta,1)$ for $\theta\in\R\setminus\{0\}$. Also let $\rho_0\sim N(\theta,1)$. Elementary computations show that both the Fokker-Planck PDE and the Sinkhorn PDE lead to the same flow  
     $\rho_t= N(\theta \exp(-t),1)$.
  \end{ex}
 \begin{ex}[Gaussian scale]\label{ex:scale}
   Suppose $\mu= N(0,1)$ and $\nu= N(0,\eta^2)$ for $\eta\in (0,1)$. Also let $\rho_0= N(0,\eta^2)$. Then both the usual Fokker-Planck PDE and the Sinkhorn PDE are Gaussians with mean $0$ and variances $\sfe$ and $\sse$ respectively, given by:
     $$\sfe=1-(1-\eta^2)\exp(-2t),\quad \sse=\left(1-2\frac{1-\eta}{\exp\left(\frac{2t}{\eta}\right)(1+\eta)+(1-\eta)}\right)^2.$$
     While both $\sfe$ and $\sse$ converge to $1$ exponentially, it is easy to check that 
     $$\frac{1-\sfe}{1-\sse}\ge \frac{(1+\eta)^2}{4}\exp\left(2t\left(\frac{1}{\eta}-1\right)\right)\to\infty,\quad \mbox{as}\ t\to\infty,$$
     as $\eta<1$. This provides an example where the convergence of the Sinkhorn PDE is much faster than the usual Fokker-Planck.
\end{ex}

As these examples demonstrate, the rate of convergence, as $t\rightarrow \infty$, of the Sinkhorn flow is subtle. We show in Section \ref{sec:mirror} that the convergence is determined by a Hessian geometry (see \eqref{eq:hessian-metric}) that is distinct from the usual Wasserstein geometry. This Hessian geometry can slow down or speed up the mirror gradient flow at any $\rho \in \ptac$ depending on the Hessian of the Brenier map transporting $\rho$ to $\nu$. This phenomenon is already well understood in Euclidean mirror flows \cite[Section B.2.4]{Wilson18}. However, as shown in Theorem \ref{lem:expcon}, under appropriately nice assumptions, one can have exponential convergence of the Sinkhorn flow.  

\begin{thm}[Convergence as $t\rightarrow \infty$ (Informal)]\label{thm:infrate}
Suppose that the marginal density $e^{-f}$ satisfies the logarithmic Sobolev inequality (see \cref{def:isid} below). Also, if $(u_t)$ is the solution of the corresponding PMA, assume that the eigenvalues of the inverse Hessian matrices $(\nabla^2 u_t)^{-1}(x)$ are uniformly bounded away from zero in $(t,x)$. Then, for some $c_0 >0$,  
\[
\KL{\rho_t}{e^{-f}}\le \KL{\rho_0}{e^{-f}}\exp(-c_0 t).
\]
\end{thm}

Sufficient conditions under which our assumptions are valid have already been explored in Berman \cite{berman2020}. We note that the rate of convergence results in Theorems~\ref{thm:inftheo} and \ref{thm:infrate} recover the correct scaling in terms of $\vep$ which allows us to work under the scaling limit regime where the number of iterations scales like $1/\vep$ with $\vep\to 0$.  On the other hand, in existing rate of convergence results in  \cite{leger2021gradient,ghosal2022convergence,Conforti2021} cited above, the rates involve terms depending on $\vep$ which blow up in the aforementioned scaling limit regime as $\vep\to 0$. Having said that, our results do require stronger regularity conditions on $f$, $g$ and solutions of the parabolic Monge-Amp\`{e}re equation described in \cref{thm:infexist} (also see \eqref{eq:pma}).

\subsection{Torus setting}\label{sec:torus}
Our main results, as stated above, require certain regularity assumptions on the solutions of the PDE \eqref{eq:newPDE}. These assumptions have been established in the PDE literature in the specific case where the probability measures $\mu$ and $\nu$ are supported on the flat torus $\mathbb{T}^d:=\R^d/\mathbb{Z}^d$ (see \cref{rem:bermanver} for more details). Recall that $\mathbb{T}^d$ consists of equivalence classes $\{x+k:\, k\in\mathbb{Z}^d\}$ for $x\in [0,1)^D$. Let us identify the equivalence class $\{x+k:\, k\in\mathbb{Z}^d\}$ with $x$ for simplicity of notation. Therefore the standard $L^2$ metric on $\mathbb{T}^d$ is given by 
$$c_{\mathbb{T}^d}(x,y):=\inf\{\lVert x-y+k\rVert:\, k\in\mathbb{Z}^d\}, \qquad x,y\in\mathbb{T}^d.$$
The solution of the PDE \eqref{eq:newPDE} on the torus is then $\mathbb{Z}^d$-periodic, i.e., $u_t(x)=u_t(x=k)$ for $x\in [0,1)^d$ and $k\in\mathbb{Z}^d$. In the same way, the probability measures $\rho_k^{\vep}$s and $\rho_t$s, in Theorems \ref{thm:infexist} and \ref{thm:inftheo}, are also $\mathbb{Z}^d$-periodic. Similarly the Brownian motion which occurs in the context of \cref{thm:diffexist} can also be defined on the flat torus (see \cite[Examples 8 and 9, pages 229--230]{bhattacharya2023continuous}
for specific definitions).
Moreover, convexity of a function, the Legendre transform $u^*$ of $u: \mathbb{R}^d\to \mathbb{R}$ can be extended to the Riemannian setting, called $c_{\mathbb{T}^d}$-convex functions and the $c_{\mathbb{T}^d}$-Legendre transform; see, e.g., \cite{McCannPolar}. 
With these conventions, all the aforementioned results continue to hold using the same proofs as we have currently presented in the paper.

\subsection{Organization} The rest of the paper is organized as follows. In \cref{sec:mirror}, we introduce Wasserstein mirror gradient flows and argue that the Sinkhorn flow is a particular example of this larger class. We also zoom in on the Sinkhorn PDE and prove its existence, rate of convergence, and connections to the linear optimal transport distance. In \cref{sec:diffmirr}, we obtain a Mckean-Vlasov SDE whose marginal distributions follow the Sinkhorn PDE. In \cref{sec:mcconst}, we move our attention to the Sinkhorn algorithm and prove that its $X$ marginals converge to the marginals of the Sinkhorn PDE, with quantitative error bounds in terms of the regularization parameter $\vep$. Finally, in Sections \ref{sec:pfres} and \ref{sec:mainresultlems}, we prove all the technical lemmas.

\section{Mirror gradient flows: Euclidean and Wasserstein}\label{sec:mirror}

\subsection{Mirror descent and mirror gradient flow in Euclidean spaces}

Mirror descent was originally proposed by Nemirovsky and Yudin in \cite{nemirovskii83} as a generalization of gradient descent. For a modern account and its connections to other gradient based optimization schemes, see \cite[Sections 2.1.3 and B.2]{Wilson18}.




Assume that $u:\R^d\rightarrow \R$ is a differentiable convex function such that $\nabla u: \R^d \rightarrow \R^d$ is a diffeomorphism. This diffeomorphism induces two global coordinate charts on $\R^d$, namely $x\leftrightarrow x^u:= \nabla u(x)$ (also see \cref{def:pdcor}). This map is called the \textit{mirror map} in the literature. Given a differentiable function $F: \rr^d \rightarrow \rr$, its Euclidean mirror gradient flow can be described by the ODE (with a given initial condition)
\[
\frac{d}{dt} x^u_t = - \frac{\partial}{\partial x} F(x_t),\; t\ge 0.
\]
Applying chain rule to the LHS, we get the equivalent alternative ODE
\begin{equation}\label{eq:mirrorflowEuclid}
\frac{d}{dt} x_t = -\frac{\partial x\hfill}{\partial x^u}(x^u_t) \frac{\partial F}{\partial x}(x_t)= - \frac{\partial F\hfill}{\partial x^u}(x_t). 
\end{equation}
In the last equality we are utilizing the idea that $x \mapsto x^u$ is a diffeomorphism on $\R^d$ and gradients of functions can be taken with respect to either coordinate system. When $u(x)=\frac{1}{2} \norm{x}^2$, both ODEs coincide and we recover the usual Euclidean gradient flow. 

The Euclidean mirror gradient flow can be interpreted as a ``true'' gradient flow if we change the manifold structure of $\R^d$ to make it a so-called \textit{Hessian Riemannian manifold}. See \cite{Hessian97} and \cite[Section B.2.4]{Wilson18}. Roughly, at a point $x\in \R^d$, consider the positive definite Hessian matrix $\nabla^{2} u(x)$. For any tangent vectors $y$ at $x$, change the Riemannian metric to $y^T \nabla^{2} u(x) y$. This induces a new Riemannian manifold on $\R^d$. The mirror gradient flow \eqref{eq:mirrorflowEuclid} is the gradient flow of $F$ on this Hessian Riemannian manifold. 

\subsubsection{Examples.} Consider the following examples where we take $d=1$, $F(x)=\frac{1}{2}x^2$, and different choices of the mirror function $u$. Notice that the function $F$ admits a unique minimizer at zero. For each of the following flows, the initial condition is $x_0=1$.

\begin{enumerate}[(i)]
    \item $u(x)=\frac{1}{2}x^2$. In this case $x^u=x$. Thus, we get back our usual gradient flow equation $\dot{x}_t=-\nabla_x F(x_t)=-x_t$. For our given initial condition, the solution is $x_t=\exp(-t)$, $t\ge 0$. It converges to zero exponentially fast. 
    \item $u(x)=x^4$. Thus $x^u=4x^3$ and $\frac{\partial x^u}{\partial x}=12x^2$. Thus, 
    \[
    \dot{x}_t= -\nabla_{x^u} F(x_t)= -\frac{1}{12 x_t^2} x_t= -\frac{1}{12 x_t}.
    \]
    The solution of this ODE, with $x_0=1$, is $x_t=\sqrt{1-t/6}$, $0\le t \le 6$. The solution does not extend beyond $[0,6]$ due to a singularity at $t=6$ when $x_6=0$, the minimizer of $F$. The flow converges to the minimizer in finite time, unlike the previous example. 
    \item $u(x)=1/x$ for $x>0$. Thus $x^u=-\frac{1}{x^2}$ and $\frac{\partial x^u}{\partial x\hfill}=\frac{2}{x^3}$. The mirror flow equation is given by $\dot{x}_t= -\frac{1}{2} x_t^4$. The unique solution with $x_0=1$ is $x_t=\left(1 + 3t/2 \right)^{-1/3}$. The solution is well defined for all $t\ge 0$ and converges to the minimizer of $F$ just polynomially. 
\end{enumerate}




In all these examples it is clear that the behavior of the Hessian of the mirror function $u$ in a neighborhood of the minimizer of $F$ plays a very important role. If this Hessian is very close to zero, the mirror flow speeds up significantly more than the gradient flow of $F$. This is an intuition that we expect to carry over to the Wasserstein set-up as well.

\subsection{An informal description of Wasserstein mirror gradient flows}\label{sec:wassmirrorflowformal}

Recall that $\ptac$ denote the set of Lebesgue absolutely continuous Borel probability measures on $\R^d$ with finite second moments. Equip this space with the Wasserstein $2$ metric. 
We take the point of view (originally due to Otto \cite{Otto_2001}) that this Wasserstein space can be thought of as an ``infinite dimensional Riemannian manifold'' in the following sense \cite[Chapter 8]{ambrosio2005gradient}. At any $\rho \in \ptac$ define the tangent space by $\tanspace_\rho = \overline{\left\{ \nabla g,\; g \in \diffcont_c^\infty \right\}}$,
where the closure is taken in $L^2(\rho)$ \cite[Section 8.4]{ambrosio2005gradient}. The metric tensor is given by the $L^2$  norm. 


For suitable functions $F:\ptac\rightarrow \R \cup \{\infty\}$, the Wasserstein gradient at an absolutely continuous probability measure with a $\mathcal{C}^1$ density is the element $\nabla_{\wass}F(\rho)\in \tanspace_\rho$, characterized by the following identity \cite[Lemma 10.4.1]{ambrosio2005gradient} that holds for all $v\in \tanspace_\rho$:
\[
\iprod{\nabla_{\wass}F(\rho), v }_{\ltwo(\rho)}= \int_{\R^d} \frac{\delta F}{\delta \rho}(x) \dot{\rho}(x) dx,
\]
where $\dot{\rho}= - \div(v\rho)$. The Wasserstein gradient flow is given by the continuity equation \cite[Section 11.1]{ambrosio2005gradient} $\frac{\partial}{\partial t} \rho_t = \div{\left(\nabla_{\wass}F(\rho_t) \rho_t\right)}$.


Let 
\begin{equation}\label{eq:mirror}
U(\rho)=\frac{1}{2}\wass_2^2\left(\rho, e^{-g}\right),
\end{equation}
which is known to be convex over generalized geodesics with base $e^{-g}$ (\cite[Lemma 9.2.1 and Definition 9.2.2]{ambrosio2005gradient}). 
We will use $U$ as a ``mirror'' to generate a ``mirror potential'' given by 
\[
\rho \mapsto \rho^U:=\nabla_{\wass}U(\rho).
\]
Notice that while $\rho$ is a measure, its mirror potential $\rho^U$ is a function in $\ltwo(\rho)$. For the special case of $U(\rho)=\frac{1}{2}\wass_2^2\left(\rho, e^{-g}\right)$, $\rho^U$ is the  Kantorovich map (gradient of the Kantorovich potential) transporting $\rho$ to $e^{-g}$ \cite[Theorem 10.4.12]{ambrosio2005gradient}. 

By analogy with the Euclidean space, for a suitable function $F: \probspace\rightarrow \R\cup \{\infty\}$, one may define the Wasserstein mirror gradient flow by two equivalent PDEs. The first one takes time derivative in the mirror potential.
\begin{equation}\label{eq:mirrorgradflow}
\frac{\partial}{\partial t} \rho^U_t = - \nabla_{\wass} F(\rho_t),
\end{equation}
which represents an equality of two elements in $\tanspace_{\rho_t}$ viewed as a subspace of $\ltwo(\rho_t)$.

Let $\nabla u$ denote the Brenier map transporting $\rho$ to $e^{-g}$ for a convex function $u$. Then 
\begin{equation}\label{eq:KPtoBP}
\nabla u(x) =  x - \rho^U(x).
\end{equation}
In terms of the convex Brenier potentials, \eqref{eq:mirrorgradflow} can be equivalently written as 
\begin{equation}\label{eq:mirrorgradflow2}
\frac{\partial}{\partial t} \nabla u_t = \nabla_{\wass} F(\rho_t).
\end{equation}
When $\nabla_{\wass} F(\rho_t)= \nabla \frac{\delta F}{\delta \rho_t}$, the gradient of the first variation of $F$ at $\rho_t$ \cite[Lemma 10.4.1]{ambrosio2005gradient}, by removing the gradients on both sides of the above, we get the PDE.
\begin{equation}\label{eq:mirrorgradflow3}
\frac{\partial u_t}{\partial t}(x)= \frac{\delta F}{\delta \rho_t}(x), \quad \text{where $u_t$ is convex and} \; (\nabla u_t)_{\#} \rho_t=e^{-g}.
\end{equation}
As shown in the following Section \ref{sec:pmasec}, this family of PDEs includes the parabolic Monge-Amp\`ere \eqref{eq:pma2}.

The second equivalent way of describing the Wasserstein mirror gradient flow is to take the time derivative of the curve in the space of measures, i.e., specify a continuity equation for the curve on Wasserstein space.  

Fix an absolutely continuous measure $\rho$. The Hessian matrix $\nabla^2 u$ is defined $\rho$-a.s. in the Alexandrov sense and is positive semidefinite. Define a new metric tensor on $\tanspace_{\rho}$ by 
\begin{equation}\label{eq:hessian-metric}
\iprod{v_1, v_2}_U= \iprod{v_1, \left(\nabla^2 u\right) v_2}_{\ltwo(\rho)}, \quad v_1, v_2 \in \tanspace_\rho. 
\end{equation}
This defines a new Riemannian gradient on the Wasserstein space $\nabla_{\wass}^U$ and a new gradient flow equation
\begin{equation}\label{eq:newgradflow}
\frac{\partial}{\partial t} \rho_t = \div\left( \nabla_{\wass}^U F(\rho_t) \rho_t \right).
\end{equation}
When it exists, we call this flow the Wasserstein mirror gradient flow of $F$ with respect to $U(\rho)=\frac{1}{2}\wass_2^2\left(\rho, e^{-g}\right)$. While $\nabla_{\wass} F(\rho)$ is given by the gradient of the first variation of $F$, $\nabla_x \left(\frac{\delta F}{\delta \rho}\right)$, for $\nabla_{\wass}^U F(\rho_t)$ we get
\begin{equation}\label{eq:expgrad}
\nabla_{\wass}^U F(\rho_t)= \nabla_{x^{u_t}}\left(\frac{\delta F}{\delta \rho_t}\right), \quad \text{where $u_t$ is convex and}\; (\nabla u_t)_{\#} \rho_t= e^{-g}.
\end{equation}
Since $\frac{\partial x\hfill}{\partial x^u} = \left(\nabla^2 u \right)^{-1}$, \eqref{eq:expgrad} can also be written as \[
\nabla_{\wass}^U F(\rho_t)=  \left(\nabla^2 u \right)^{-1} \nabla_{\wass} F(\rho)\]
which is consistent with the metric \eqref{eq:hessian-metric}.
Thus \eqref{eq:newgradflow} and \eqref{eq:expgrad} describe the Wasserstein mirror gradient flow as a continuity equation.

The connection between \eqref{eq:newgradflow} and \eqref{eq:mirrorgradflow3} is that if $(u_t,\; t\ge 0)$ is a solution of the latter, then $\rho_t:=\left( \nabla u_t\right)^{-1}_{\#} e^{-g}$ 
 satisfies the former. We show this below assuming that the solution of \eqref{eq:mirrorgradflow3} satisfies, for each $t$, $\nabla^2 u_t(x)$ is invertible and the inverse is continuous in $x$, $\rho_t$ a.s.. 

 Let $T_t=\nabla u_t$ denote the transport map. Then, by the chain rule,  
\begin{align}\label{eq:inverse_derivative}
    \left[\frac{\partial}{\partial t} (T_t)^{-1} \right]( T_t(x)) = - [\nabla T_t(x)]^{-1}   \left[\frac{\partial}{\partial t} T_t \right](x).
\end{align}
Notice that $[\nabla T_t(x)]^{-1} =\left(\nabla^2u_t(x)\right)^{-1}$.
Thus, from \eqref{eq:mirrorgradflow2}, 
\[
\left[\frac{\partial}{\partial t} (T_t)^{-1} \right]( T_t(x)) = - \left(\nabla^2u_t(x)\right)^{-1}\nabla \frac{\delta F}{\delta \rho_t}=-\nabla_{x^{u_t}}\frac{\delta F}{\delta \rho_t}=- \nabla_{\wass}^UF(\rho_t). 
\]

Let $\xi$ be a smooth, compactly supported test function. Then 
\[
\begin{split}
\frac{d}{dt} &\int \xi(x) \rho_t(x)dx = \frac{d}{dt} \int \xi\left(T_t^{-1}(y)\right)e^{-g(y)}dy\\
&=\int \nabla \xi\left(T_t^{-1}(y)\right)\cdot  \left[\frac{\partial}{\partial t} (T_t)^{-1} \right](y) e^{-g(y)}dy  \\
&=\int \nabla \xi(x) \left[\frac{\partial}{\partial t} (T_t)^{-1} \right](T_t(x)) \rho_t(x)dx=- \int \iprod{\nabla \xi(x), \nabla_{\wass}^UF(\rho_t)} \rho_t(x)dx.
\end{split}
\]
This proves that $(\rho_t,\; t\ge 0)$ is a weak solution of the continuity equation \eqref{eq:newgradflow}. 

Although, to the best of our knowledge, this particular mirror gradient flow on the Wasserstein space is new in the literature, some related ideas that involve modifications to the usual Wasserstein geometry and/or considering gradient flows on them can be found in other recent papers such as \cite{SVGD17,RankinWong23, WassersteinNewton}.  

\begin{remark}\label{rmk:hori-ver-corr}
    The paragraph around  \eqref{eq:inverse_derivative} gives a  geometric way to understand the connection between \eqref{eq:mirrorgradflow3} and 
    \eqref{eq:newgradflow}. Namely, if we consider the graph of $\nabla u_t$ (the mirror map) in the product space $\mathbb{R}^d \times \mathbb{R}^d$, the vector field $\nabla_{\wass} F(\rho_t)$ gives the vertical variation of the graph, while the vector field $-\nabla_{\wass}^U F(\rho_t)=-  \left(\nabla^2 u \right)^{-1} \nabla_{\wass} F(\rho)$ gives its horizontal counterpart, which gives the same variation of the graph. The equation \eqref{eq:inverse_derivative} gives the precise relation between the horizontal and vertical variation.
    Note that the distribution $\rho_t$ is the result of the drift $\nabla_{\wass}^U F(\rho_t)$ via the continuity equation. The difference between the mirror flow and the usual Wasserstein gradient flow, is that the vector field $\nabla_{\wass} F(\rho_t)$ provides variation on $\rho_t$ not directly in the continuity equation but indirectly through the aforementioned vertical-horizontal correspondence, which in turn is provided by the mirror map. 
\end{remark}

\subsection{Parabolic Monge-Amp\`{e}re and the mirror gradient flow of Kullback-Leibler divergence}\label{sec:pmasec}

The parabolic Monge-Amp\`{e}re (PMA) is the following partial differential equation (PDE): 
\begin{equation}\label{eq:pma}
    \frac{\partial u_t}{\partial t\hfill} (x)=f(x)-g(x^{u_t})+\ldet\left( \frac{\partial x^{u_t}}{\partial x\hfill} \right),
\end{equation}
for some initial function $u_0:\R^d\to\R$. The idea is that, as $t\rightarrow \infty$, the LHS of \eqref{eq:pma} should converge to zero, whereby $u_t$ should converge to the solution of the Monge-Amp\`{e}re equation for transporting $e^{-f}$ to $e^{-g}$. Although the PMA has been around in the literature, there has been some recent appearances in the context of optimal transport. For example,  \cite{kim2012parabolic} studied it for general cost functions on closed manifolds, while 
\cite{kitagawa2012parabolic} studies it for bounded domains. On the other hand, \cite{deb2025no} recently used it for generative modeling and variational inference.

There is a more convenient form for us that we will utilize frequently. This needs the well-known change of measure lemma (see \cite[Theorem 1.6.9]{durrettprob}) which we note here for easy reference.
\begin{lmm}\label{lem:jacobian}(Change of measure)
    Let $e^{-a}$ be a probability density function on $\R^d$ for some function $a:\R^d\rightarrow \R$. For a $\diffcont^2$ strictly convex function $\phi:\R^d\rightarrow \R$, let $e^{-b}$ be the pushforward $(\nabla \phi)_{\#} e^{-a}$ Then,
    \begin{equation}\label{eq:jacobian}
    b(x^\phi) = a(x) + \ldet \frac{\partial x^\phi}{\partial x\hfill}.
    \end{equation}
\end{lmm}

Now, let $h_t:\R^d\rightarrow \R$ be such that $e^{-h_t}$ is the pushforward of the measure $e^{-g}$ by the inverse of the map $x\mapsto x^{u_t}$. Then, by Lemma \ref{lem:jacobian}, it follows that \eqref{eq:pma} can be alternatively written as 
\begin{equation}\label{eq:pma2}
\frac{\partial u_t}{\partial t}(x)= f(x) - h_t(x) = \log \left( \frac{e^{-h_t(x)}}{e^{-f}}. \right).
\end{equation}

Consider the function on the Wasserstein space 
\begin{equation}\label{eq:choiceofF}
F(\rho):=\begin{cases}
\KL{\rho}{e^{-f}},& \text{when $\rho$ is dominated by {\color{red}$e^{-f}$}}, \\
+\infty,& \text{otherwise}.  
\end{cases}
\end{equation}
When $f$ is convex, it is well-known that $F$ is geodesically (and generalized geodesically) convex and lower semicontinuous. 

The first variation of $F$ (see \cite[Lemma 10.4.1]{ambrosio2005gradient}) at $\rho_t=e^{-h_t}$ is given by $\frac{\delta F}{\delta \rho_t}= f(x) - h_t(x)$, 
(ignoring an additive constant which does not affect its gradient $\nabla_{\wass} F = \nabla_x  \frac{\delta F}{\delta \rho}$).
Thus, \eqref{eq:pma2} may be written as $\dot{u}_t=\frac{\delta F}{\delta \rho_t}$ which is exactly \eqref{eq:mirrorgradflow3} and, by taking a gradient on both sides, we recover  \eqref{eq:mirrorgradflow2}. In our language, \textit{the PMA is the evolution of the mirror potential for the mirror gradient flow of relative entropy with respect to $e^{-f}$, where the mirror is generated by the function in \eqref{eq:mirror}}.

By \eqref{eq:newgradflow} and \eqref{eq:expgrad}, the mirror gradient flow itself satisfies the continuity equation 
 \begin{equation*}
 \partial_t \et+\div{(\et v_t)}=0, \quad v_t(x)=-\frac{\partial\hfill}{\partial x^{u_t}} (f-h_t)(x),
 \end{equation*}
 where $(u_t,\; t\ge 0)$ solves the PMA \eqref{eq:pma}. This is the same equation we introduced in \eqref{eq:velocity}. We call $(\rho_t,\; t\ge 0)$ the Sinkhorn flow, and we will show that this is the limit of iterates of the Sinkhorn algorithm. Note that, alternatively, from \eqref{eq:pma2}, $u_t$ may also be expressed as
 \[
 u_t(x)= u_0(x) + \int_0^t (f(x) + \log \rho_s(x))ds, \quad t\ge 0.
 \]

\noindent Let us now formalize the existence of a solution to \eqref{eq:velocity}. To state this, we need some assumptions on the solution of the PMA \eqref{eq:pma}. 
\begin{assm}\label{asn:solcon}
Assume that $f,g,u_0$ are such that the PMA \eqref{eq:pma} admits a solution $\left( u_t,\; t\ge 0 \right)$. Additionally,
\begin{enumerate}[(i)]
\item Given any $T>0$, there exists constants $A_T>0$ and $B_T>0$ such that 
\begin{align}\label{eq:curvbd}
\inf_x \inf_{t\in [0,T]}\lmn\left(\frac{\partial x^{u_t}}{\partial x\hfill}\right)\geq A_T,\quad \sup_x \sup_{t\in [0,T]} \lmx\left(\frac{\partial x^{u_t}}{\partial x\hfill}\right)\le B_T.
\end{align}
\item The map $(t,x)\mapsto u_t(x)$  lies in $\diffcont^{1,2}([0,\infty)\times \R^d)$\footnote{The reader is warned that the  $\diffcont^{k,\ell}$ notation is not to be confused with the H\"{o}lder class of functions}. Further, for every $T>0$, all the corresponding mixed partial derivatives (in space and time) are bounded on $[0,T]\times \R^d$.
\item The first two derivatives of $f(\cdot)$ and $g(\cdot)$ are bounded and uniformly continuous. Further, the first four spatial derivatives of $\{u_t\}_{t\in [0,T]}$ and $\{u_t^*\}_{t\in [0,T]}$ are bounded and uniformly continuous, for all $T>0$.
 \end{enumerate}
\end{assm}

Given the solution of the PMA $(u_t,\; t\ge 0)$ (as in \eqref{eq:pma}), note that $\left(w_t=u_t^*,\; t\ge 0 \right)$ denotes the corresponding process of convex conjugates. By \cref{asn:solcon}, part (i), both $u_t$ and $w_t$'s are all strictly convex $\diffcont^2$ functions. Therefore, $\nabla u_t(\cdot)$ and $\nabla w_t(\cdot)$ are both diffeomorphisms on $\R^d$. In view of~\cref{def:pdcor}, we then have a system of dual coordinates on $\R^d$ given by $x\mapsto \xsut$ (or equivalently $y\mapsto y^{w_t}$), one for each $t\geq 0$. We show in Lemma \ref{lem:dualPMA} that the family $(w_t)_{t\ge 0}$ is also a solution of a PMA that we call the \text{dual PMA}. The following observation is an immediate consequence of \cref{asn:solcon} which we note as a remark below. 

\begin{remark}\label{rem:dualasn}
 If \cref{asn:solcon} holds for the solution of the PMA $(u_t)_{t\ge 0}$, then  the assumptions (i)--(iii)  also hold for $(w_t=u_t^*)_{t\ge 0}$, as $\nabla u_t = (\nabla w_t)^{-1}$, and vice versa. This is relevant for \cref{lem:dualPMA} below.
\end{remark}

\begin{remark}\label{rem:bermanver}
Sufficient conditions for \cref{asn:solcon} have been studied in the literature under different assumptions on the supports of the probability measures $e^{-f}$ and $e^{-g}$, and the initializer $u_0$ for \eqref{eq:pma}. For example, in \cite[Proposition 4.5]{berman2020}, that if $u_0$ is four times continuously  differentiable, $f$ and $g$ are twice continuously differentiable, and $e^{-f}$, $e^{-g}$ are supported on the torus, then \cref{asn:solcon} is satisfied.  Similar results have also been established for compact Riemannian manifolds (see \cite[Theorem 1.1]{kim2012parabolic}), and general bounded convex domains (see \cite[Theorem 3.1]{kitagawa2012parabolic}, \cite[Theorem A]{Tang2013}, and \cite[Theorem 1.1]{Abedin2020}. A simple case when \cref{asn:solcon} holds for probability measures supported on $\R^d$ is where both $\mu$ and $\nu$ are Gaussian probability densities (see Examples \ref{ex:loc} and \ref{ex:scale}).

The most critical part of \cref{asn:solcon} is the curvature bound in \eqref{eq:curvbd}. When the probability measures are supported on $\R^d$, in a recent paper \cite[Lemma 2.2]{chiarini2024semiconcavity}, the authors show that the Sinkhorn potentials (see \eqref{eq:twostepit}) satisfy the curvature condition \eqref{eq:curvbd}, uniformly for all small enough $\vep$, provided that the Hessians of $f$, $g$ are bounded above and below by positive constants times identity, and both the initial potentials are strongly convex. Given the connections between the Sinkhorn potentials and the PMA \eqref{eq:pma} established in \cite[Lemma 4.4]{berman2020}, this suggests that the PMA \eqref{eq:pma} should also satisfy \eqref{eq:curvbd} under the additional log-concavity assumptions. A rigorous proof of this is left for future research.

\end{remark}

We are now in position to state our first main result which features the existence of a strong solution to \eqref{eq:velocity}. \begin{thm}\label{thm:existlin}
Fix $T>0$ and suppose that \cref{asn:solcon} holds. Recall $w_t=u_t^*$ from Remark \ref{rem:dualasn}. Consider the push-forward $\rho_t=(\nabla w_t)_{\#}e^{-g}$. Then $(\rho_t)_{t\in [0,T]}$ is an absolutely continuous curve in the $2$-Wasserstein space and it is a strong solution to \eqref{eq:velocity}. 
\end{thm}

The proof of \cref{thm:existlin} depends on the following change of coordinate lemma that will be used multiple times in this paper.

\begin{lmm}\label{lem:tensorel}
    Suppose that $\phi:\R^d \to\R$ is strictly convex such that the function $\nabla \phi: \R^d \rightarrow \R^d$ is a $\diffcont^2$ diffeomorphism. That is both  $\nabla \phi$ and its inverse $\nabla \phi^*$ are twice continuously differentiable. Then, 
    \begin{equation}\label{eq:tensorelpf}
    \frac{\partial}{\partial \xsph_j}\left(\log\det\left(\frac{\partial x^\phi}{\partial x\hfill}\right)\right)=-\sum_{\ell=1}^d \frac{\partial^2 x_j}{\partial x_{\ell}\partial \xsph_{\ell}},
    \end{equation}
    for all $j\in [d]$ and $x\in\R^d$.
\end{lmm}

The proof of \cref{lem:tensorel} is deferred to~\cref{sec:pfres}. 

\begin{proof}[Proof of \cref{thm:existlin}]
By invoking \cref{lem:jacobian} with $\phi=w_t$, $a=g$, we observe that $\rho_t$ is strictly positive and 
 $h_t:=-\log \rho_t$
 is well-defined. 
 Observe that 
 \begin{align}\label{eq:com1}
 h_t(x)=g(x^{u_t})-\ldet\left(\frac{\partial x^{u_t}}{\partial x\hfill}\right)=f-\frac{\partial u_t}{\partial t\hfill}
 \end{align}
 where the first equality is by \eqref{eq:jacobian} and the second equality is by \eqref{eq:pma}. 
 
Now starting with the first equality above along with the chain rule, we get:
\begin{align*}
\frac{\partial h_t}{\partial t\hfill}(x)\nonumber &=\frac{\partial}{\partial t}\left(g(x^{u_t})-\ldet\left(\frac{\partial x^{u_t}}{\partial x\hfill}\right)\right)\nonumber \\ &=\bigg\langle \frac{\partial x^{u_t}}{\partial t\hfill},\frac{\partial g\hfill}{\partial x^{u_t}} (x^{u_t})\bigg\rangle - \frac{\partial}{\partial t}\left(\ldet\left(\frac{\partial x^{u_t}}{\partial x\hfill}\right)\right)
\end{align*}
For the second term, we will use \cite[Section A.4.1]{boyd2004convex} to note the fact that the derivative of $\ldet(A)$ with respect to the entries of $A$ is given by $A^{-1}$. By taking $A=\frac{\partial x^{u_t}}{\partial x\hfill}$ (therefore $A^{-1}=\frac{\partial x\hfill}{\partial x^{u_t}}$) and applying the chain rule again, we get
\begin{align}\label{eq:rhot1}
    -\frac{1}{\rho_t(x)}\frac{\partial \rho_t}{\partial t}(x)=\frac{\partial h_t}{\partial t\hfill}(x)=\bigg\langle \frac{\partial x^{u_t}}{\partial t\hfill},\frac{\partial g\hfill}{\partial x^{u_t}} (x^{u_t})\bigg\rangle-\bigg\langle \frac{\partial x\hfill}{\partial x^{u_t}}, \frac{\partial}{\partial t}\left(\frac{\partial x^{u_t}}{\partial x\hfill}\right)\bigg\rangle.\end{align}

Take 
\begin{align}\label{eq:vtalt}
v_t=\frac{\partial\hfill}{\partial x^{u_t}}((h_t-f)(x))=-\frac{\partial\hfill}{\partial x^{u_t}}\left(\frac{\partial u_t}{\partial t\hfill}(x)\right).
\end{align}
The first equality is the definition of $v_t$ as in  \eqref{eq:velocity}, and the second equality is the PMA \eqref{eq:pma}. Then by the product rule we have:
\begin{align}\label{eq:rhot2}
&\frac{\div(v_t\rho_t)}{\rho_t}(x)\nonumber =\langle v_t(x),\nabla \log{\rho_t(x)}\rangle + \div v_t(x)\nonumber \\ 
&=\bigg\langle -\frac{\partial}{\partial x^{u_t}}\left(\frac{\partial u_t}{\partial t\hfill}(x)\right),-\frac{\partial h_t}{\partial x\hfill}(x)\bigg\rangle - \div{\left(\frac{\partial}{\partial x^{u_t}}\left(\frac{\partial u_t}{\partial t\hfill}(x)\right)\right)},
\end{align}
where the last line uses \eqref{eq:vtalt}. We will now simplify each of the terms above. For the first term, observe that:

\begin{align*}
    \bigg\langle -\frac{\partial}{\partial x^{u_t}}\left(\frac{\partial u_t}{\partial t\hfill}(x)\right),-\frac{\partial h_t}{\partial x\hfill}(x)\bigg\rangle &=\bigg\langle \frac{\partial}{\partial x}\frac{\partial u_t}{\partial t\hfill}(x)\left(\frac{\partial x}{\partial x^{u_t}\hfill}\right),\frac{\partial h_t}{\partial x}(x)\bigg\rangle\\ &=\bigg\langle \frac{\partial x^{u_t}}{\partial t\hfill}, \left(\frac{\partial x}{\partial x^{u_t}\hfill}\right)\frac{\partial h_t}{\partial x}(x)\bigg\rangle =\bigg\langle \frac{\partial x^{u_t}}{\partial t\hfill},\frac{\partial h_t\hfill}{\partial x^{u_t}}(x)\bigg\rangle.
\end{align*}
The first and third equalities above use the chain rule (see \eqref{eq:notcal} for clarity of notation). For the second equality, we have interchanged the time and the space derivatives in the term $\frac{\partial}{\partial x}\frac{\partial u_t}{\partial t\hfill}(x)$ using \cref{asn:solcon}, part (ii). We now move on to the second term of \eqref{eq:rhot2}. Note that
\begin{align*}
    \div{\left(\frac{\partial}{\partial x^{u_t}}\left(\frac{\partial u_t}{\partial t\hfill}(x)\right)\right)}&=\div{\left(\frac{\partial x^{u_t}}{\partial t}\left(\frac{\partial x}{\partial x^{u_t}}\right)\right)}\\ &=\sum_{i,j=1}^d \frac{\partial}{\partial x_i}\left(\left(\frac{\partial x^{u_t}}{\partial t\hfill}\right)_j\left(\frac{\partial x\hfill}{\partial x^{u_t}}\right)_{j,i}\right).
\end{align*}
In the first equality, we have used the chain rule as before and the second display uses the definition of divergence. We will now  use the product rule to obtain 
\begin{align*}
    &\;\;\;\;\sum_{i,j=1}^d \frac{\partial}{\partial x_i}\left(\left(\frac{\partial x^{u_t}}{\partial t\hfill}\right)_j\left(\frac{\partial x\hfill}{\partial x^{u_t}}\right)_{j,i}\right)\\ &=\bigg\langle \frac{\partial}{\partial t}\left(\frac{\partial x^{u_t}}{\partial x\hfill}\right),\frac{\partial x\hfill}{\partial x^{u_t}}\bigg\rangle+\sum_{j=1}^d \left(\frac{\partial x^{u_t}}{\partial t\hfill}\right)_j\left(\sum_{i=1}^d \frac{\partial^2 x_i}{\partial x_i\partial x_i^{u_t}}\right)\\ &=\bigg\langle \frac{\partial x\hfill}{\partial x^{u_t}},\frac{\partial}{\partial t}\left(\frac{\partial x^{u_t}}{\partial x\hfill}\right)\bigg\rangle-\bigg\langle \left(\frac{\partial x^{u_t}}{\partial t\hfill}\right),\frac{\partial}{\partial x^{u_t}}\left(\ldet\left(\frac{\partial x^{u_t}}{\partial x\hfill}\right)\right)\bigg\rangle.
\end{align*}
In the second inequality above, we have used \eqref{eq:tensorelpf} on the second term. We now combine our observations on the two terms of \eqref{eq:rhot2} to get: 
\begin{small}
\begin{align*}
    \frac{\div(v_t\rho_t)}{\rho_t}(x)&=\bigg\langle \left(\frac{\partial x^{u_t}}{\partial t\hfill}\right),\frac{\partial}{\partial x^{u_t}}\left(h_t(x)+\ldet\left(\frac{\partial x^{u_t}}{\partial x\hfill}\right)\right)\bigg\rangle-\bigg\langle \frac{\partial x\hfill}{\partial x^{u_t}},\frac{\partial}{\partial t}\left(\frac{\partial x^{u_t}}{\partial x\hfill}\right)\bigg\rangle.
\end{align*}
\end{small}
Observe that, by applying $\frac{\partial}{\partial x_j^{u_t}}$ on both sides of \eqref{eq:com1} we get 
\begin{align*}
    \frac{\partial}{\partial x^{u_t}}\left(h_t(x)+\ldet\left(\frac{\partial x^{u_t}}{\partial x\hfill}\right)\right)=\frac{\partial g\hfill}{\partial x^{u_t}}(x^{u_t}).
\end{align*}
Therefore, 
\begin{align}\label{eq:rhot3}
\frac{\div(v_t\rho_t)}{\rho_t}(x)=\bigg\langle \left(\frac{\partial x^{u_t}}{\partial t\hfill}\right),\frac{\partial g\hfill}{\partial x^{u_t}}(x^{u_t})\bigg\rangle-\bigg\langle \frac{\partial x\hfill}{\partial x^{u_t}},\frac{\partial}{\partial t}\left(\frac{\partial x^{u_t}}{\partial x\hfill}\right)\bigg\rangle.
\end{align}

By adding \eqref{eq:rhot1} and \eqref{eq:rhot3}, it follows that $\frac{\partial \rho_t}{\partial t\hfill}+\div{(\rho_t v_t)}=0$.

\noindent Next, we establish absolute continuity of $(\rho_t)_{t\in [0,T]}$. By invoking \cite[Theorem 8.3.1]{ambrosio2005gradient}, it suffices to show that $\sup_{x,t\in [0,T]} \lVert v_t(x)\rVert<\infty$. By the representation of $v_t$ in \eqref{eq:vtalt} and the PMA \eqref{eq:pma}, we observe that for $t\in [0,T]$,
\begin{align*}
\lVert v_t(x)\rVert&=\left\lVert \left(\frac{\partial x}{\partial x^{u_t}}\right)\frac{\partial\hfill}{\partial x}\left(f(x)-g(x^{u_t})+\ldet\left(\frac{\partial x^{u_t}}{\partial x\hfill}\right)\right)\right\rVert\\ &\le B_T \left\lVert \frac{\partial\hfill}{\partial x}\left(f(x)-g(x^{u_t})+\ldet\left(\frac{\partial x^{u_t}}{\partial x\hfill}\right)\right) \right\rVert.
\end{align*}
Here in the first equality, we have again used the chain rule and in the following inequality, we have used the upper bound in \cref{asn:solcon}, part (i). Next, note that by \cref{asn:solcon}, parts (i) and (iii), we have:
$$\sup_x \ \left\lVert \frac{\partial f}{\partial x}(x)\right\rVert<\infty, \quad \sup_x \left\lVert \frac{\partial\hfill}{\partial x}(g(x^{u_t}))\right\rVert\le \sup_x \left\lVert \frac{\partial g\hfill}{\partial x^{u_t}}(x^{u_t})\right\rVert \ \sup_x \left\lVert \frac{\partial x^{u_t}}{\partial x\hfill}\right\rVert<\infty.$$
Therefore, to prove $\sup_{x,t\in [0,T]} \lVert v_t(x)\rVert$ is bounded, it suffices to control the norm of $\frac{\partial}{\partial x}\left(\ldet\left(\frac{\partial x^{u_t}}{\partial x\hfill}\right)\right)$. To wit, we will once again use \cite[Section A.4.1]{boyd2004convex} (as we did earlier while obtaining \eqref{eq:rhot1} above). For $i\in [d]$, we then get
\begin{align*}
    \frac{\partial}{\partial x}\left(\ldet\left(\frac{\partial x^{u_t}}{\partial x\hfill}\right)\right)=\bigg\langle \frac{\partial x\hfill}{\partial x^{u_t}},\frac{\partial\hfill}{\partial x_i}\left(\frac{\partial x^{u_t}}{\partial x\hfill}\right)\bigg\rangle.
\end{align*}
Now observe that the entries of $\frac{\partial x\hfill}{\partial x^{u_t}}$ are uniformly bounded (in $x$ and $t\in [0,T]$) in terms of $A_T$ by using \cref{asn:solcon}, part (i). Also the entries of $\frac{\partial}{\partial x_i}\left(\frac{\partial x^{u_t}}{\partial x\hfill}\right)$ are uniformly bounded by \cref{asn:solcon}, part (iii), where we have assumed boundedness of the third derivative tensor of $u_t$'s. Therefore, 
$$\sup_{x,t\in [0,T]} \left\lVert \frac{\partial}{\partial x}\left(\ldet\left(\frac{\partial x^{u_t}}{\partial x\hfill}\right)\right)\right\rVert<\infty.$$
This readily implies $\sup_{x,t\in [0,T]} \lVert v_t(x)\rVert<\infty$ and completes the proof.
\end{proof}

The velocity field $(v_t)$ appearing in the Sinkhorn flow \eqref{eq:velocity} is not a gradient in $x$. Thus, it may not lie in the $2$-Wasserstein tangent space at $\rho_t$, and the metric derivative (as defined in \cite[Theorem 1.1.2]{ambrosio2005gradient}) of the curve at time $t$ may not be $\norm{v_t}_{L^2(\rho_t)}$, which is the case for usual gradient flows in an appropriate sense (see \cite[Theorem 8.3.1]{ambrosio2005gradient}). 

However, as we show below, if we replace the $2$-Wasserstein distance (see \eqref{eq:2wass}) with the $\dlt{}{e^{-g}}{\cdot}{\cdot}$ metric (see \eqref{eq:2linot}), the metric derivative of the curve is indeed given by $\norm{v_t}_{L^2(\rho_t)}$. 

\begin{thm}\label{prop:metderlin}
Suppose \cref{asn:solcon} holds. Then with $\rho_t$ as in \cref{thm:existlin}, for any $t\ge 0$, we have: 
\begin{align}\label{eq:linot}
\lim_{\delta\to 0} \frac{1}{\delta}\dlt{}{e^{-g}}{\rho_{t+\delta}}{\rho_t}=\lVert v_t\rVert_{L^2(\rho_t)},
\end{align}
and 
\begin{align}\label{eq:linotsec}
    \dlt{}{e^{-g}}{\rho_{t+\delta}}{(\nabla w_t+\delta v_t(\nabla w_t))_{\#} e^{-g}}=o(\delta).
\end{align}
\end{thm}

The conclusion in \eqref{eq:linotsec} shows that if we appropriately perturb $\nabla w_t$ in the direction of the velocity $v_t$, then the corresponding push-forward measure $(\nabla w_t+\delta v_t(\nabla w_t))_{\#} e^{-g}$ yields a first order approximation to $\rho_{t+\delta}$. Such results are very popular in the literature on continuity equations and usual gradient flows.  Similar results for Euclidean gradient flows imply that, following the tangent vector of a smooth curve $\{x_t\}$ at a point $t$ over a small timestep $\delta$ is ``close" up to the first order (in the Euclidean metric) to $x_{t+\delta}$. Alternatively, it means that the trajectory of a particle moving along a smooth curve can be approximated by piecewise constant velocity curves. This is also true for usual Wasserstein gradient flows (see \cite[Proposition 8.4.6]{ambrosio2005gradient}). There the authors show that if $(\rho_t)_{t\ge 0}$ satisfies a continuity equation with velocity $v_t$ which is a gradient in $x$ (not the case in our setting), then $\wass_2(\rho_{t+\delta},(\mathrm{Id}+\delta v_t)_{\#}\rho_t)=o(\delta)$. \cref{prop:metderlin} establishes a similar result for Wasserstein mirror flows with the Hilbertian linear optimal transport distance (see \eqref{eq:2linot}). Note that the Wasserstein distance is \emph{smaller than} the linear optimal transport distance. In this sense, the approximation of $\rho_{t+\delta}$ in \eqref{eq:linotsec} is stronger than that in \cite[Proposition 8.4.6]{ambrosio2005gradient} (albeit under stronger regularity conditions).

In order to prove \cref{prop:metderlin} we need some additional results. First is an important conjugacy relationship for the PMA. 
Remark \ref{rmk:hori-ver-corr} suggests that there is a dual flow because one can flip the choice of the horizontal and vertical directions. Indeed, recall that the potential $u_t$ in \eqref{eq:KPtoBP} gives the flow $\rho_t= (\nabla u_t)^{-1}_\# e^{-g}$. We can consider the family of convex conjugates $w_t=u_t^*$, whereby $\nabla w_t  = (\nabla u_t)^{-1}$. Viewed as vector fields the time derivative $\partial_t \nabla u_t$
    is equivalent to the vertical variation $\nabla_{\wass} F(\rho_t)$ while 
    $\partial_t \nabla w_t$  is equivalent to   the horizontal variation
   $- \nabla_{\wass}^U F(\rho_t).$
   By \eqref{eq:pma}, $(u_t)$ solves a PMA. We show in Lemma \ref{lem:dualPMA} that $(w_t)$ also solves a (different) PMA. This conjugacy is also evident from the Sinkhorn algorithm which comes with a pair of potentials at each step. The PMA \eqref{eq:pma} is the limit (in $\vep$) of one of these sequence of potentials. So, by symmetry, it is only natural that the other sequence of potentials has a PMA limit as well and the corresponding potentials in the two limiting PMAs are related by convex duality.

    


\begin{lmm}\label{lem:dualPMA}
Suppose \cref{asn:solcon} holds. Then the process $\left( w_t=u_t^*,\; t\ge 0 \right)$  is also the solution of a PMA:
\begin{equation}\label{eq:dualPMA}
\frac{\partial w_t}{\partial t}(y)= g(y) - f(y^{w_t}) + \ldet \left( \frac{\partial y^{w_t}}{\partial y\hfill}\right),
\end{equation}
with the initial condition $w_0=u_0^*$. 
\end{lmm}        

The next result is a property of the velocity field $v_t$ in \eqref{eq:velocity}.

\begin{lmm}\label{lem:convexcall}
Suppose \cref{asn:solcon} holds. Then, for $\delta>0$ small enough, the function $y\mapsto y^{w_t}+\delta v_t(y^{w_t})$ is the gradient in $y$ of the convex function $y\mapsto w_t(y)-\delta (f(y^{w_t})-h_t(y^{w_t}))$.
\end{lmm}

\begin{proof}[Proof of \cref{prop:metderlin}]

We first prove \eqref{eq:linot}. As $w_t=u_t^*$ satisfies $(\nabla w_t)_{\#} e^{-g}=\rho_t=e^{-h_t}$, we have:
    \begin{align}\label{eq:simlin}
    \frac{1}{\delta^2}\dlt{2}{e^{-g}}{\rho_{t+\delta}}{\rho_t}&=\int \lVert \delta^{-1}(\nabla w_{t+\delta}(y)-\nabla w_t(y))\rVert^2 e^{-g(y)}\,dy.
    \end{align}
    By \cref{asn:solcon} (also see \cref{rem:dualasn}), we note that
    $$\sup_y \bigg|\delta^{-1}(\nabla w_{t+\delta}(y)-\nabla w_t(y))-\nabla \frac{\partial w_t}{\partial t\hfill}(y)\bigg|\le \sup_{y,s\in [t,t+\delta]}\bigg|\nabla \frac{\partial w_s}{\partial s\hfill}(y)-\nabla \frac{\partial w_t}{\partial t\hfill}(y)\bigg|=o(1).$$
    By combining the above display with \eqref{eq:dualPMA}, we get:
    \begin{align*}
        \sup_y \bigg|\delta^{-1}(\nabla w_{t+\delta}(y)-\nabla w_t(y))- \nabla\left(g(y)-f(y^{w_t})+\ldet\left(\frac{\partial y^{w_t}}{\partial y\hfill}\right)\right)\bigg|=o(1).
    \end{align*}
    Here $o(1)$ is with respect to $\delta\to 0$. Next we use the first equality in \eqref{eq:vtalt} to observe that 
    \begin{align}\label{eq:labelgrad}
        v_t(y^{w_t})=\frac{\partial}{\partial y}((h_t-f)(y^{w_t}))=\frac{\partial}{\partial y}\left(g(y)-f(y^{w_t})+\ldet\left(\frac{\partial y^{w_t}}{\partial y\hfill}\right)\right).
    \end{align}
    In the second equality, we have used \eqref{eq:com1}. By combining the two displays above, we get:
    \begin{align}\label{eq:vtcall2}
        \sup_y \bigg|\delta^{-1}(\nabla w_{t+\delta}(y)-\nabla w_t(y))- v_t(y^{w_t})\bigg|=o(1).
    \end{align}
    By combining the above display with \eqref{eq:simlin}, we get
    \begin{align*}
    \frac{1}{\delta^2}\dlt{2}{e^{-g}}{\rho_{t+\delta}}{\rho_t}&=\int \lVert v_t(y^{w_t})\rVert^2 e^{-g(y)}\,dy+o(1).
    \end{align*}
    A final change of variable with $x=y^{w_t}$ establishes \eqref{eq:linot}. 

\vspace{0.05in}

We next prove \eqref{eq:linotsec}. By \cref{lem:convexcall}, there exists $\delta>0$ small enough such that the function $y\mapsto y^{w_t}+\delta v_t(y^{w_t})$ is the gradient of a convex function. With such $\delta>0$, by McCann's Theorem (see \cite{McCann1995}), $y^{w_t}+\delta v_t(y^{w_t})$ is the optimal transport map from $e^{-g}$ to $(y^{w_t}+\delta v_t(y^{w_t}))_{\#} e^{-g}$. As $y^{w_{t+\delta}}$ is the optimal transport map from $e^{-g}$ to $\rho_{t+\delta}$. With these observations, we get:
\begin{align*}
    &\;\;\;\;\dlt{}{e^{-g}}{\rho_{t+\delta}}{(\nabla w_t+\delta v_t(\nabla w_t))_{\#} e^{-g}}\\ &=\left(\int \lVert \nabla w_{t+\delta}(y)-\nabla w_t(y)-\delta v_t(y^{w_t})\rVert^2 e^{-g(y)}\,dy\right)^{\frac{1}{2}}=o(\delta), 
\end{align*}
where the last equality follows from \eqref{eq:vtcall2}. This proves \eqref{eq:linotsec}.
\end{proof}

\begin{remark}[Mirror descent for fixed $\vep>0$]\label{rem:Flavian}
In \cite[Theorem 2]{leger2021gradient}, the author showed that for fixed $\vep>0$, the marginals of the Sinkhorn algorithm \eqref{eq:sinkupdt} can be viewed as iterations of an explicit mirror descent algorithm. In particular, one can (informally) write 
\begin{align*}
\rho_{k+1}^{\vep}&=\argmin_{\mu\in\mathcal{P}_2(\R^d)}\bigg[\KL{\rho_{k}^{\vep}}{e^{-f}}+\bigg\langle \mu-\rho_{k}^{\vep}, \frac{\partial \hfill}{\partial \mu}\KL{\mu}{e^{-f}}\big|_{\mu=\rho_{k}^{\vep}}\bigg\rangle\\ &\quad\quad + F^*(\mu)-F^*(\rho_{k,\vep})-\bigg\langle \mu-\rho_{k}^{\vep},\frac{\partial \hfill}{\partial \mu}F^*(\mu)\bigg|_{\mu=\rho_{k}^{\vep}}\bigg\rangle,  
\end{align*}
where $F^*$ is the convex conjugate of the function $F$ defined by 
$$F(\phi)=\langle \opV[\phi],e^{-g}\rangle,$$
for ``smooth" functions $\phi:\R^d\to\R^d$ and $\opV$ defined in \eqref{eq:basedef} later. 

While \cite{leger2021gradient} focuses on the $\vep>0$ case, we take the limiting perspective ($\vep\to 0$) which allows us to explicitly identify the mirror function (see \eqref{eq:mirror}) and velocity vector field of the limiting (Sinkhorn) PDE \eqref{eq:velocity}. To our understanding, these explicit quantities cannot be derived from the proof techniques used in \cite{leger2021gradient}.
\end{remark}

\subsection{Exponential convergence of the 
Sinkhorn flow}\label{sec:sinkflow}

As \eqref{eq:velocity} 
arises out of the mirror flow of the entropy functional $\KL{\rho}{e^{-f}}$, it is natural to ask the following: does $\rho_t\to e^{-f}$ as $t\to\infty$? If so, then in what sense and what is the speed of convergence? 

To address this, we will now establish the exponential convergence of the Sinkhorn flow to $e^{-f}$. To begin, let us define the log-Sobolev inequality.

\begin{defn}\label{def:isid}
    We say that a probability measure $\xi\in \ptac$ satisfies a logarithmic Sobolev inequality (LSI) with constant $\lsi{\xi}>0$ if for all $\rho\in\ptac$, 
    \[
    \KL{\rho}{\xi}\le \frac{1}{2\lsi{\xi}}I(\rho|\xi),
    \]
    where $I(\rho|\xi)$ is the \emph{relative Fisher information} defined by
    \[
    I(\rho|\xi):=\int \left\lVert \nabla\log\frac{d\rho}{d\xi}\right\rVert^2\,d\rho.
    \]
\end{defn}

\begin{thm}\label{lem:expcon}
    Suppose there exists $\lsi{f}>0$ such that $e^{-f}$ satisfies LSI with constant $\lsi{f}$. Also assume that there exists a positive continuous function $h(\cdot)$ on $[0,\infty)$ such that
    \begin{equation}\label{eq:uniflb}
    \inf_{x\in\R^d}\lmn\left(\frac{\partial x\hfill}{\partial x^{u_t}}\right)\ge h(t).
    \end{equation}
    Then we have
    $$\KL{\rho_t}{e^{-f}}\le \KL{\rho_0}{e^{-f}}\exp(-2\lsi{f}H(t)),$$
    and 
    $$\wass_2(\rho_t,e^{-f})\le \sqrt{\frac{2\KL{\rho_0}{e^{-f}}}{\lsi{f}}}\exp(-\lsi{f} H(t)),$$
    where $H(t):=\int_0^t h(s)\,ds$.
\end{thm}




\begin{proof}[Proof of \cref{lem:expcon}]
    For convenience, we define $m:=\KL{\rho_0}{e^{-f}}$. By using the standard terminology of analysis in the $2$-Wasserstein space (see \cite[Chapters 9 and 10]{ambrosio2005gradient}), we note that the function $\rho\mapsto \KL{\rho}{e^{-f}}$ admits a subdifferential $\nabla \log\rho_t+\nabla f$.  
 This implies
    \begin{align*}
        \frac{d}{dt}\KL{\rho_t}{e^{-f}}&=\int \iprod{\nabla \log{\rho_t(x)}+\nabla f(x),v_t(x)}\,d\rho_t(x)\\ &=-\int \left(\nabla \log{\rho_t(x)}+\nabla f(x)\right)^{\top}\frac{\partial x\hfill}{\partial x^{u_t}\hfill}\left(\nabla \log{\rho_t(x)}+\nabla f(x)\right)\,d\rho_t(x).
    \end{align*}
    In the last display, we have used the representation of $v_t$ from \eqref{eq:velocity} coupled with the chain rule as illustrated in \eqref{eq:notcal}. Next, by using \eqref{eq:uniflb}, we get:
    \begin{align*}
        &\;\;\;\;-\int \left(\nabla \log{\rho_t(x)}+\nabla f(x)\right)^{\top}\frac{\partial x\hfill}{\partial x^{u_t}\hfill}\left(\nabla \log{\rho_t(x)}+\nabla f(x)\right)\,d\rho_t(x)\\ &\le -h(t)\int \lVert \nabla \log{\rho_t(x)}+\nabla f(x)\rVert^2\,d\rho_t(x)\\ &=-h(t) I(\rho_t|e^{-f})\le -2\lsi{f}h(t)\KL{\rho_t}{e^{-f}}
    \end{align*}
    In the final display here we have used the logarithmic Sobolev inequality (LSI) assumption, see \cref{def:isid}. 
    Next by invoking Gronwall's inequality (see \cite[Theorem II]{walter2012differential}), we get:
    $$\KL{\rho_t}{e^{-f}}\le \KL{\rho_0}{e^{-f}}\exp(-2\lsi{f}\int_0^t h(s)\,ds).$$
    Next we use the HWI inequality (see \cite[Theorem 1]{Otto2000}) to get:
    $$\wass_2^2(\rho_t,e^{-f})\le \frac{2}{\lsi{f}}\KL{\rho_t}{e^{-f}}\le \frac{2}{\lsi{f}}\KL{\rho_0}{e^{-f}}\exp(-2\lsi{f}\int_0^t h(s)\,ds).$$
    This completes the proof.
\end{proof}

The LSI assumption (see~\cref{def:isid}) is standard in the literature on rates of convergence of flows and plays a pivotal role in establishing information geometric inequalities; see \cite{Anton2001,Otto2000} and the references therein. In particular, the LSI condition can be verified in many popular examples. We cite two of them below.

\begin{enumerate}
    \item If $\inf_{x\in\R^d}\lmn(\nabla^2 f(x))\ge c$ for some constant $c>0$, then $e^{-f}$ satisfies LSI with constant $c>0$ (see \cite{Bakry1985}).
    \item Suppose $e^{-\tilde{f}}$ satisfies LSI with constant $\tilde{c}$. Let $\bar{f}:=f-\tilde{f}$ and assume $\bar{f}\in L^{\infty}(\R^d)$. Then $e^{-f}$ satisfies LSI with constant $c:=\tilde{c}\exp(\inf \bar{f}-\sup \bar{f})$ (see \cite{Holley1987}).
\end{enumerate}
We refer the reader to \cite{Cattiaux2010,Chen2021,Wang2001} for other conditions under which  the LSI condition can be established.


\subsection{Other examples of Wasserstein mirror gradient flows}\label{sec:othexamp}

It is natural to be curious about Wasserstein mirror gradient flows with other choices of function $F$. We give a few examples below. The reader should be careful that the PDEs and flows that we calculate have not been shown to exist or be well-behaved. 
\medskip

\begin{ex}[Potential energy]
Consider the function $F(\rho)=\frac{1}{2}\int \norm{x}^2 \rho(dx)$. It is strictly geodesically convex and has a unique minimizer at $\delta_0$. 

Given $e^{-g}$ and the mirror function \eqref{eq:mirror}, let us compute the time evolution of the mirror gradient flow. The evolution of the mirror potential is given by 
\[
\frac{\partial u_t}{\partial t}(x)= \frac{\delta F}{\delta \rho_t}(x)= \frac{1}{2}\norm{x}^2.  
\]
Thus $u_t= u_0 + \frac{t}{2} \norm{x}^2$, $t\ge 0$. A simple but instructive special case is when $u_0(x)=\frac{1}{2} \norm{x}^2$, i.e. $\rho_0=e^{-g}$, itself. Then $\nabla u_t= (1+t)\mathbf{id}$. Thus $\rho_t$ is the law of $Y/(1+t)$, where $Y \sim e^{-g}$. Clearly $\lim_{t\rightarrow \infty}\rho_t=\delta_0$ in $\wass_2$.
\end{ex}

\medskip

\begin{ex}[Entropy]
Let $F(\rho)= \int \rho(x) \log \rho(x) dx$ denote the entropy function. The function is defined to be $+\infty$ when the measure is not absolutely continuous. From \cite[Lemma 10.4.1]{ambrosio2005gradient}, $\frac{\delta F}{\delta \rho}=\log \rho + 1$. Thus, given $e^{-g}$, the time evolution of the mirror potential is given by the PDE
\begin{equation}\label{eq:entropymirror}
\frac{\partial u_t}{\partial t}(x)=  \log \rho_t(x) +1, \quad \text{i.e.},\quad \nabla u_t(x)= \nabla u_0(x) + \int_0^t \nabla \log \rho_s ds.
\end{equation}
The flow on the other hand is given by the continuity equation $\dot{\rho}_t + \nabla \cdot (v_t \rho_t)=0$ where
\[
v_t(x) = - \frac{\partial\hfill}{\partial x^{u_t}} \log \rho_t(x).  
\]
A closed form solution is available for the Gaussian family. Let $\rho_0, e^{-g}$ be both standard normal. Thus $\nabla u_0(x)= x$. Then $\rho_t=N(0, (1+t)^2I)$ is a solution of the mirror gradient flow. This can be easily verified from the fact that 
\[
\nabla \log \rho_t(x)= -\frac{x}{(1+t)^2}, \quad \nabla u_t(x)= \frac{x}{(1+t)},\quad \nabla \frac{\partial u_t}{\partial t}(x)= -\frac{x}{(1+t)^2}.
\]
This system is a solution to \eqref{eq:entropymirror}. 

It is, of course, well-known that the Wasserstein gradient flow of entropy is the heat flow which admits the solution $N(0, (1+t)I)$ at time $t$ when started with standard normal. Thus, in this case, the mirror gradient flow ``converges'' faster than the usual gradient flow, although in both cases the solution diffuses as $t\rightarrow \infty$.  
\end{ex}

\medskip

\begin{ex}[R\'{e}nyi entropy]
In \cite{Otto_2001}, Otto identified the solution of the porous medium equation as the Wasserstein gradient flow of the following functional that is related to the R\'{e}nyi entropy:
\[
F(\rho)= \frac{1}{m-1}\int \rho^m(x)dx,  
\]
where $m \ge 1 - \frac{1}{d}$, $m> \frac{d}{d+2}$ and $m\neq 1$. The function is defined to be $+\infty$ if $\rho$ is not absolutely continuous.

To find its mirror gradient flow, fix $e^{-g}$ to generate the mirror potential. From \cite[Lemma 10.4.1]{ambrosio2005gradient}, $\frac{\delta F}{\delta \rho}=\frac{m}{m-1} \rho^{m-1}$. Then, the evolution of the mirror potential is given by the PDE 
\[
\frac{\partial u_t}{\partial t}(x)= \frac{m}{m-1} \rho_t^{m-1}, \quad \text{where}\quad \rho_t= (\nabla u^*_t)_{\#} e^{-g},
\]
and the continuity equation of the mirror gradient flow is given by the velocity
\[
v_t(x)=-m\rho_t^{m-2}\nabla_{x^{u_t}} \rho_t(x).
\]
It is not immediate how these new flows compare with the traditional ones. 
\end{ex}

\section{The Sinkhorn Diffusion}\label{sec:diffmirr}
 It is a natural question whether there exists a diffusion such that the family of marginal distributions at each time point gives the Sinkhorn flow \eqref{eq:velocity}.  A classic example of this correspondence is  
 the Langevin diffusion (see \cite{lemons1997paul}) whose time marginals satisfy the Fokker-Planck equation. Such stochastic processes are useful in many applications including optimization (see \cite{chizat2022trajectory,durmus2019analysis,roberts1996exponential}) and sampling (see \cite{sohl2015deep,song2020score}). In this section we construct such a stochastic process inspired from a natural Markov chain embedded in the Sinkhorn algorithm, see \cref{prop:mchn} later in the paper. 


\begin{defn}[Sinkhorn diffusion] The Sinkhorn diffusion is   the solution to the following stochastic differential equation (SDE):
\begin{equation}\label{eq:diffSDE}
    dX_t=\left(-\frac{\partial f\hfill}{\partial \xsut}(X_t)-\frac{\partial g\hfill}{\partial\xsut}\left(X_t^{u_t}\right)+\frac{\partial h_t\hfill}{\partial \xsut}(X_t)\right)\,dt+\sqrt{2\frac{\partial X_t\hfill}{\partial X_t^{u_t}}}dB_t,
\end{equation}

where
\begin{enumerate}[(i)]

\item $X_0$ is distributed according to an initial density $\rho_0$. At each subsequent time $t$, $X_t$ admits a density $\rho_t=e^{-h_t}$. 

\item $u_t$ is a convex function whose gradient is the Brenier map transporting $\rho_t$ to $e^{-g}$. That is, $(\nabla u_t)_{\#} \rho_t=e^{-g}$. At each time $t$, this leads to a mirror coordinate system $x \mapsto x^{u_t}$. 

\item $\frac{\partial f\hfill}{\partial \xsut}(x)$ refers to the derivative of $x \mapsto f(x)$ with respect to the dual variable $x^{u_t}$. Same for $\frac{\partial h_t\hfill}{\partial \xsut}(x)$.

\item $\frac{\partial g\hfill}{\partial \xsut}(x^{u_t})$ is the gradient of the map $y \mapsto g(y)$ evaluated at $y=x^{u_t}$. 

\item $(B_t,\; t\geq 0)$ is a standard $d$-dimensional  Brownian motion and the diffusion matrix 
$\displaystyle 2\frac{\partial X_t\hfill}{\partial X_t^{u_t}}$ 
at time $t$ is  
\[
2\frac{\partial x\hfill}{\partial x^{u_t}}=2\left( \nabla^2 u_t(x) \right)^{-1},
\]
evaluated at $X_t=x$. 
\end{enumerate}
\end{defn}

Sinkhorn diffusion is an example of Mckean-Vlasov family of diffusions \cite{mckean1966class}.  The study of such systems originated from the probabilistic study of the Boltzmann and Vlasov equations due to~\cite{dobrushin79,kac1956foundations,mckean75,tanaka78} and many others. For modern surveys, see ~\cite{ChaintronDiez,Jabin14,SznitmanSF,villani12notes}.  



We will show that a weak solution of the SDE exists and is unique under suitable assumptions. Towards this, suppose that solution of the PMA \eqref{eq:pma} exists which satisfies Assumption \ref{asn:solcon} with the initial condition $u_0$.
In fact, Assumption \ref{asn:solcon} gives sufficient regularity to the mirroring map $\nabla u_t$, which gives the mirror map $x\mapsto x^{u_t}$. 
The Sinkhorn diffusion always has a dual diffusion process, say $Y$,  given via this mirror map, i.e., $Y_t=X_t^{u_t}$.   As we will show in \cref{thm:existpropX} below $Y$ satisfies the following SDE: 
\begin{equation}\label{eq:dualdiffSDE} 
\begin{split}
dY_t&= -\nabla h_t\left(Y^{w_t}_t\right)dt +\sqrt{2\frac{\partial Y_t\hfill}{\partial Y_t^{w_t}}}d B_t,
\end{split}
\end{equation}
where $\rho_t=e^{-h_t}= \left(\nabla w_t\right)_{\#} e^{-g}$ is the pushforward of $e^{-g}$ by the map $y \mapsto \nabla w_t(y)$. 
Here $\nabla h_t(Y_t^{w_t})$ is the gradient of $h_t$ with respect to its argument, evaluated at $Y_t^{w_t}$. 


\begin{thm}\label{thm:existprop}
Let
\[
b(t,y):=- \nabla h_t(y^{w_t}), \quad \sigma(t,y):=\sqrt{2 \frac{\partial y\hfill}{\partial y^{w_t}}}= \sqrt{2  \left(\nabla^{2}\left( w_t(y)\right)\right)^{-1}}.
\]

Additionally, suppose the standard global Lipschitz and linear growth conditions hold,  namely, for some $K>0$, $b, \sigma$ are uniformly $K$-Lipschitz functions on $\rr^d$ and,  
    \begin{equation}\label{eq:growth}
    \norm{b(t,y)}^2 + \norm{\sigma(t,y)}^2 \le K\left(1 + \norm{y}^2 \right)
    \end{equation}
for all $t\ge 0$. 
Then, if the initial distribution is square-integrable, the SDE \eqref{eq:dualdiffSDE} admits a unique strong solution such that every subsequent $Y_t$ is also square-integrable.  

The infinitesimal generator of the process at time $t$, acting on a $\diffcont^2$ function $\phi$, is given by 
\[
\mathcal{L}_t^Y \phi = e^{g} \div \left( e^{-g} \nabla_{y^{w_t}}\phi \right).
\]
Consequently $e^{-g}$ is a stationary distribution for this process.
\end{thm}

\begin{remark}\label{rem:suffcon}
    The condition \eqref{eq:growth} holds under \cref{asn:solcon} via elementary computations. In particular, this implies that under \cref{asn:solcon}, the conclusion in \cref{thm:existprop} holds.
\end{remark}

\begin{proof}[Proof of~\cref{thm:existprop}] The claim about existence, uniqueness (pathwise and in law) and square-integrability follow from \cite[Theorem 5.2.9]{karatzas1991brownian}.

It remains to compute the infinitesimal generator. Let $\varphi\in \diffcont^2$. Then, by It\^o's formula 
\[
\begin{split}
    \mathcal{L}^Y_t\varphi &= - \frac{\partial \varphi}{\partial y} \cdot \frac{\partial h_t\hfill}{\partial y^{w_t}}(y^{w_t}) + \sum_{i=1}^d \sum_{l=1}^d \frac{\partial y_l\hfill}{\partial y^{w_t}_i} \frac{\partial^2 \varphi}{\partial y_i \partial y_l}.
\end{split}
\]
By the formula for change of measures
\[
-h_t(y^{w_t})= -g(y) - \log \det \frac{\partial y\hfill}{\partial y^{w_t}}.
\]
Taking gradients with respect to $y^{w_t}$ on both sides and using formula \eqref{eq:tensorelpf} we get 
\[
\begin{split}
  \mathcal{L}^Y_t\varphi &=  -\frac{\partial \varphi}{\partial y} \cdot \frac{\partial g\hfill}{\partial y^{w_t}}(y) + \sum_{i=1}^d\sum_{l=1}^d \frac{\partial \varphi}{\partial y_l} \frac{\partial^2 y_l}{\partial y_i \partial y^{w_t}_i} + \sum_{i=1}^d \sum_{l=1}^d  \frac{\partial^2 \varphi}{\partial y_i \partial y_l}\frac{\partial y_l\hfill}{\partial y^{w_t}_i}\\
  &= -\frac{\partial \varphi}{\partial y} \cdot \frac{\partial g\hfill}{\partial y^{w_t}}(y) + \sum_{i=1}^d \frac{\partial}{\partial y_i}\left[ \sum_{l=1}^d \frac{\partial \varphi}{\partial y_l} \frac{\partial y_l\hfill}{\partial y^{w_t}_i} \right]\\
  &=-\frac{\partial \varphi}{\partial y} \cdot \frac{\partial g\hfill}{\partial y^{w_t}}(y) + \sum_{i=1}^d \frac{\partial^2 \varphi}{\partial y_i \partial y^{w_t}_i}= e^{g}\div\left( e^{-g} \nabla_{y^{w_t}} \varphi\right).
\end{split}
\]

That $e^{-g}$ is a stationary measure follows immediately, since
\[
\int \mathcal{L}^Y_{t} \varphi(y) e^{-g(y)}dy = \int \div\left( e^{-g} \nabla_{y^{w_t}} \varphi\right) dy=0.
\]
\end{proof}

\begin{remark}
    The above diffusion generator is a time-inhomogeneous analog of the one described by \cite[Section 2]{Kolesnikov2012HessianMC}. 
\end{remark}

\begin{thm}\label{thm:existpropX}
On any filtered probability space that supports a standard $d$-dimensional Brownian motion, let $Y$ be a solution of \eqref{eq:dualdiffSDE}. Define $X_t:=Y_t^{w_t}$. Then $(X_t,\; t\geq 0)$ is a strong solution of \eqref{eq:diffSDE}. Conversely, for any solution $X$ of \eqref{eq:diffSDE}, the transformed process $Y_t=X_t^{u_t}$ is a solution of \eqref{eq:dualdiffSDE}. Hence, under the assumptions of Theorem \ref{thm:existprop}, a strong solution of \eqref{eq:diffSDE} exists and is unique in law. The marginals of $X_t$ so constructed satisfy \eqref{eq:velocity}. 
\end{thm}

Of course, one may also impose global Lipschitzness and linear growth property on the drift and diffusion coefficients of $X$ to obtain a strong solution. The reason we used the $Y$ process is one, that it has fewer terms in the drift, and two, we only require the $Y$ process to run in stationarity to obtain a weak solution for $X$.

In order to prove \cref{thm:existpropX}, we need the well-known Ito's lemma \cite{karatzas1991brownian}, which we note down here for easy reference. 

\begin{lmm}\label{lem:ito}
Consider an SDE of the form
$$d X_t=b(t,X_t)\,dt+\sigma(t,X_t)\,dB_t.$$
Here $b:[0,\infty)\times \R^d\to \R^d$ and $\sigma:[0,\infty)\times \R^d \rightarrow \R^d \times \R^d$ are progressively measurable functions. Let $\phi:[0,\infty)\times \R^d\to\R$ be an element in  $\diffcont^{1,2}$. Define $S_t=\phi(t,X_t)$. Then 
\begin{align}\label{eq:ito}
d S_t= \frac{\partial}{\partial x}  \phi(t,X_t)^{\top}& \sigma(t,X_t)\,dB_t+\left[\frac{\partial}{\partial t}\phi(t,X_t) + \frac{\partial}{\partial x} \phi(t,X_t)^{\top} b(t,X_t)\right] dt \nonumber\\
&+\frac{1}{2}\trc\left(\sigma(t,X_t) \sigma(t,X_t)^{\top}\nabla^2_x \phi(t,X_t)\right) dt.
\end{align}
\end{lmm}

\begin{proof}[Proof of Theorem \ref{thm:existpropX}]
By \cref{thm:existprop}, there exists a strong Markov process $(Y_t,\ t\ge 0)$ which is a weak solution to the SDE in \eqref{eq:dualdiffSDE}. 
Consider $X_t=Y_t^{w_t}$. By \cref{asn:solcon} (iv), we can apply It\^{o}'s formula in \cref{lem:ito} with
$$b(t,y)=-\frac{\partial h_t \hfill}{\partial y^{w_t}}(y^{w_t}),\quad \sigma^2(t,y)=2\frac{\partial y}{\partial y^{w_t}},\quad \phi(t,y)=y^{w_t},$$ 
to get that $(X_t,\ t\ge 0)$ is Markov and:
\begin{align}\label{eq:dual2}
dX_t&= \frac{\partial Y_t^{w_t}}{\partial Y_t\hfill}  \sqrt{2\frac{\partial Y_t\hfill}{\partial Y_t^{w_t}}}\,dB_t+\left[\left(\frac{\partial}{\partial t}\nabla w_t\right)(Y_t) - \frac{\partial Y_t^{w_t}}{\partial Y_t}  \frac{\partial h_t\hfill}{\partial y^{w_t}}(Y_t^{w_t})\right] dt\nonumber \\ &+\frac{1}{2}\trc\left(2\frac{\partial Y_t\hfill}{\partial Y_t^{w_t}}\nabla^2_y (Y_t)^{w_t}_i\right)_{i\in [d]} dt.
\end{align}
By using \cref{lem:dualPMA} and the multivariate chain rule, the second term on the right hand side above reduces to
\begin{align}\label{eq:dual1}
&\;\;\;\;\;\left(\frac{\partial}{\partial t}\nabla w_t\right)(Y_t) - \frac{\partial Y_t^{w_t}}{\partial Y_t}\frac{\partial h_t\hfill}{\partial y^{w_t}}(Y_t^{w_t})\nonumber \\&=\frac{\partial g}{\partial y}(Y_t)-\frac{\partial f}{\partial y}(Y_t^{w_t})+\frac{\partial}{\partial y}\ldet \left(\frac{\partial Y_t^{w_t}}{\partial Y_t\hfill}\right)-\frac{\partial h_t}{\partial y\hfill}(Y_t^{w_t})\nonumber \\&=\frac{\partial g\hfill}{\partial x^{u_t}}(X_t^{u_t})-\frac{\partial f\hfill}{\partial x^{u_t}}(X_t)-\frac{\partial}{\partial x^{u_t}}\ldet\left(\frac{\partial X_t\hfill}{\partial X_t^{u_t}\hfill}\right)-\frac{\partial h_t\hfill}{\partial x^{u_t}}(X_t) \nonumber \\ &=-\frac{\partial f\hfill}{\partial x^{u_t}}(X_t).
\end{align}
In the third display above, we have used that $Y_t=X_t^{u_t}$. In the fourth display, we use \cref{lem:jacobian} with $\phi=u_t$, $a=h_t$ and $b=g$. Next let us simplify the third term on the right hand side of \eqref{eq:dual2}, for each $i\in [d]$.
\begin{align*}
\trc\left(\frac{\partial Y_t\hfill}{\partial Y_t^{w_t}}\nabla^2_y (Y_t)^{w_t}_i\right)&=\sum_{k} \frac{\partial^2}{\partial y^{w_t}_k \partial y_k}(Y_t)^{w_t}_i=\sum_{k}\frac{\partial^2}{\partial x_k\partial x^{u_t}_k}(X_t)_i.
\end{align*}
In the above display, we have used once again that $Y_t=X_t^{u_t}$. By combining the above observation with  \cref{lem:tensorel}, we then have:
\begin{align}\label{eq:dual3}
&\;\;\;\;\;\trc\left(\frac{\partial Y_t\hfill}{\partial Y_t^{w_t}}\nabla^2_y (Y_t)^{w_t}_i\right)\nonumber \\&=-\frac{\partial}{\partial x^{u_t}_i}\ldet\left(\frac{\partial X_t^{u_t}}{\partial X_t\hfill}\right)=\frac{\partial h_t\hfill}{\partial x^{u_t}_i}(X_t)-\frac{\partial g\hfill}{\partial x^{u_t}_i}(X_t^{u_t}).
\end{align}
By combining \eqref{eq:dual2}, \eqref{eq:dual1} and \eqref{eq:dual3}, we get that $(X_t,\ t\ge 0)$ is a strong Markov process. Note that $X$ is driven by the same Brownian motion that drives the strong solution of $Y$. Thus we have constructed a strong solution of \eqref{eq:diffSDE}.

\bigskip 

\noindent Note that there exists a unique strong Markov process which is a weak solution to \eqref{eq:dualdiffSDE} by \cref{thm:existprop}. In order to establish uniqueness in \cref{thm:existpropX}, it suffices to show that given any strong Markov process $(X_t,\ t\ge 0)$ which is a weak solution of \eqref{eq:diffSDE}, the process $(Y_t=X_t^{u_t},\ t\ge 0)$ is a weak solution to \eqref{eq:dualdiffSDE}. Once again, we use It\^{o}'s lemma \ref{lem:ito}, this time with 
$$b(t,x)=-\frac{\partial f\hfill}{\partial \xsut}(x)-\frac{\partial g\hfill}{\partial\xsut}\left(x^{u_t}\right)+\frac{\partial h_t\hfill}{\partial \xsut}(x), \quad \sigma^2(t,x)=2\frac{\partial x\hfill}{\partial x^{u_t}},\quad \phi(t,x)=x^{u_t}.$$
This gives
\begin{align}\label{eq:dual4}
dY_t= \frac{\partial X_t^{u_t}\hfill}{\partial X_t\hfill}  \sqrt{2\frac{\partial X_t\hfill}{\partial X_t^{u_t}}}\,dB_t&+\Bigg[\left(\frac{\partial}{\partial t}\nabla u_t\right)(X_t) -\frac{\partial X_t^{u_t}}{\partial X_t\hfill}\Bigg(\frac{\partial f\hfill}{\partial \xsut}(X_t)+\frac{\partial g\hfill}{\partial\xsut}\left(X_t^{u_t}\right)\nonumber \\ &-\frac{\partial h_t\hfill}{\partial \xsut}(X_t)\Bigg)\Bigg] dt+\frac{1}{2}\trc\left(2\frac{\partial X_t\hfill}{\partial X_t^{u_t}}\nabla^2_x (X_t)^{u_t}_i\right)_{i\in [d]} dt.
\end{align}
By using \eqref{eq:pma}, we get:
\begin{align*}
&\;\;\;\;\left(\frac{\partial}{\partial t}\nabla u_t\right)(X_t) -\frac{\partial X_t^{u_t}}{\partial X_t\hfill}\left(\frac{\partial f\hfill}{\partial \xsut}(X_t)+\frac{\partial g\hfill}{\partial\xsut}\left(X_t^{u_t}\right)-\frac{\partial h_t\hfill}{\partial \xsut}(X_t)\right)\\ &=-2\frac{\partial g}{\partial x}(X_t^{u_t})+\frac{\partial h_t}{\partial x}(X_t)+\frac{\partial}{\partial x}\ldet\left(\frac{\partial X_t^{u_t}}{\partial X_t\hfill}\right)=-\frac{\partial g}{\partial x}(X_t^{u_t}).
\end{align*}
Here the last equality follows by invoking \cref{lem:jacobian} with $\phi=u_t$, $a=h_t$ and $b=g$. Finally, fix $i\in [d]$ and note that by the same computation as in \eqref{eq:dual3}, we have:
\begin{align*}
\trc\left(\frac{\partial X_t\hfill}{\partial X_t^{u_t}}\nabla^2_x (X_t)^{u_t}_i\right)=\frac{\partial}{\partial x_i}\ldet\left(\frac{\partial X_t^{u_t}}{\partial X_t\hfill}\right)=\frac{\partial g\hfill}{\partial x_i}(X_t^{u_t})-\frac{\partial h_t}{\partial x_i}(X_t).
\end{align*}
Combining the two displays above with \eqref{eq:dual4} and using that $Y_t=X_t^{u_t}$, we get that $(Y_t,\ t\ge 0)$ is a strong Markov process which is a weak solution of \eqref{eq:dualdiffSDE}.

\noindent We will use Ito's rule to establish the flow of the marginals. Pick a smooth, compactly supported real-valued function $\phi(\cdot)$, then by invoking Ito's rule in~\eqref{eq:ito} with 
\begin{align*}
b(t,x)=-\frac{\partial f\hfill}{\partial x^{u_t}}(x)-\frac{\partial g\hfill}{\partial x^{u_t}}(x^{u_t})+\frac{\partial h_t\hfill}{\partial x^{u_t}}(x),\quad \quad \sigma^2(t,x)=2\frac{\partial x\hfill}{\partial x^{u_t}},
\end{align*}
the expectation of the generator is given by:
\begin{align*}
&=\E[\mathcal{L}(\phi)(X_t)]\\ &=\int \left\langle\frac{\partial\phi}{\partial x}(x),b(t,x)\right\rangle \exp(-h_t(x))\,dx + \sum_{i,j}\int \frac{\partial}{\partial x_i}\left(\frac{\partial \phi\hfill}{\partial x_j}(x)\right)\left(\frac{\partial x\hfill}{\partial x^{u_t}}\right)_{i,j}\exp(-h_t(x))\,dx\\ &=\int \left\langle\frac{\partial\phi}{\partial x}(x),b(t,x)\right\rangle \exp(-h_t(x))\,dx + \sum_{j}\int \frac{\partial}{\partial x^{u_t}_j}\left(\frac{\partial \phi\hfill}{\partial x_j}(x)\right)\exp(-h_t(x))\,dx\\ &=\int \left\langle\frac{\partial\phi}{\partial x}(x),b(t,x)\right\rangle \exp(-h_t(x))\,dx + \sum_{j}\int \frac{\partial\phi\hfill}{\partial x_{j}}(x)\frac{\partial}{\partial x^{u_t}_j}\exp(-h_t(x))\,dx\\ &=\int \left\langle\frac{\partial\phi}{\partial x}(x),b(t,x)+\frac{\partial}{\partial x^{u_t}}(g(x^{u_t}))\right\rangle \exp(-h_t(x))\,dx\\ &=\int \left\langle\frac{\partial\phi}{\partial x}(x),-\frac{\partial}{\partial x^{u_t}}(f-h_t)(x)\right\rangle \exp(-h_t(x))\,dx. 
\end{align*}
By the absolute continuity of $(\rho_t)$, there exists a velocity field $v_t(\cdot)\in \R^d$ such that the continuity equation 
\begin{equation}\label{eq:continuity}
\frac{\partial\rho_t}{\partial_t\hfill} +\div{(v_t \rho_t)}=0
\end{equation}
is satisfied in the sense that 
\begin{align*}
\E[\mathcal{L}(\phi)(X_t)]=\int \left\langle \frac{\partial\phi}{\partial x\hfill}(x),v_t(x)\right\rangle \exp(-h_t(x))\,dx.
\end{align*}
As the above displays hold for all smooth $\phi(\cdot,\cdot)$, by comparing them, we get:
$$v_t(x)=-\frac{\partial}{\partial x^{u_t}}(f-h_t)(x).$$

\noindent This completes the proof.
\end{proof} 
    




 
%

\section{Sinkhorn Markov chain and Limiting Dynamics}\label{sec:mcconst}

\noindent Fix $\epsilon>0$. Following \cite{berman2020}, we will track the evolution of the Sinkhorn algorithm  using the following  maps $\opV:\diffcont(\R^d)\to \diffcont(\R^d)$ and $\opU:\diffcont(\R^d)\to \diffcont(\R^d)$, where 
\begin{equation}\label{eq:basedef}
    \begin{split}
        \opV[u](y)&:= \vep \log\int \exp\left(\frac{1}{\vep}\langle x,y\rangle-\frac{1}{\vep}u(x)-f(x)\right)\,dx,\quad u \in \diffcont(\R^d)\\
        \opU[v](x)&:= \vep \log\int  \exp\left(\frac{1}{\vep}\langle x,y\rangle-\frac{1}{\vep}v(y)-g(y)\right)\,dy, \quad v \in \diffcont(\R^d).
    \end{split}
\end{equation}

Next, consider the new operator $\opS:\diffcont(\R^d)\to \diffcont(\R^d)$ defined by:
\begin{equation}\label{eq:Sep}
\opS[u]:=\opU\circ \opV[u].    
\end{equation}
In terms of the Sinkhorn algorithm $\opS$ tracks the potential after two successive steps.

The following proposition from \cite[eqn. (2.1.4)]{berman2020} shows an useful property of the increment $\opS[u]-u$ which will be useful in the sequel. 

\begin{prop}\label{cl:normalize}
For any $u\in \diffcont(\R^d)$, consider the nonnegative function \begin{equation}\label{eq:claim12}
\rv[u](x):=\exp\left(\frac{1}{\vep}\left(\opS[u](x)- u(x)\right)\right).
\end{equation}
Then, 
\[
\int \rv[u](x)\exp(-f(x))\,dx=1.
\]
That is, $\rv[u]\exp(-f)$ is a probability measure. 
\end{prop}

The operator $\opS$ yields a natural iteration on $\diffcont(\R^d)$. Starting with some $u_0\in \diffcont(\R^d)$, consider
\begin{equation}\label{eq:twostepit}
u^{\vep}_{k+1}:=\opS[u_k],\quad \mk{k+1}:=\rv[u_k^{\vep}]\exp(-f),\quad k\ge 0.
\end{equation}
The $u_k^{\vep}$s are usually called \emph{Sinkhorn potentials}. By~\cref{cl:normalize}, we have
\begin{equation}\label{eq:discdiff}
\frac{u^{\vep}_{k+1}-u^{\vep}_k}{\vep}-f=\log{\mk{k+1}},\quad \mbox{where}\, \int_x \mk{k+1}(x)\,dx=1.
\end{equation}
Consequently, the average increment satisfies
\[
\int_x (u^{\vep}_{k+1}(x)-u^{\vep}_k(x))\,e^{-f(x)}\,dx=-\vep \KL{e^{-f}}{{\color{purple}\mk{k+1}}}\leq 0,
\]
where $\mI$ denotes the appropriate Kullback-Leibler divergence.

Note how the LHS of \eqref{eq:discdiff} looks like a discrete time derivative if iterations are indexed in units of $\epsilon$. That is, we replace $k$ by $k \vep$ and $(k+1)$ by $(k+1)\vep$. This observation will be useful later when we take scaling limit by sending $k\vep \rightarrow t>0$. 
\par 


Define the following sequence of probability densities on $\R^d\times\R^d$ for $k\ge 0$:
\begin{align}\label{eq:coupling}
\gvp_{k+1}(x,y):=\exp\left(\frac{1}{\vep}\langle x,y\rangle-\frac{1}{\vep}u^{\vep}_k(x)-\frac{1}{\vep}\opV[u^{\vep}_{k}](y)-f(x)-g(y)\right).
\end{align}
From the definitions of $\opV$, $\opU$ and $\opS$ in \eqref{eq:basedef} and \eqref{eq:claim12}, it is easy to check that
$$\int_{\R^d\times\R^d}\gvp_{k+1}(x,y)\,dx\,dy=1,\quad \forall \ k\ge 0,$$
and further
\begin{equation}\label{eq:curtaincall}
p_X \gvp_{k+1}=\mk{k},\quad p_Y \gvp_{k+1}=\exp(-g),\quad \quad \forall\ k\ge 0.
\end{equation}
Therefore $\gamma_{k}^{\vep}$ is the Schr\"{o}dinger Bridge (see \cite{Schrodinger1932}) coupling between $\mk{k}$ and $e^{-g}$.  
{\color{black}As the $Y$-marginals of all $\gvp_{k+1}$s remain stationary at $\exp(-g)$, one can construct a natural Markov chain using $\gvp_{k+1}$s. This is elucidated in the following definition.

\begin{defn}[Sinkhorn Markov chain]\label{prop:mchn}
Let $\gamma_0^{\vep}$ be some arbitrary joint distribution where $p_Y \gamma_0^\vep =e^{-g}$. For $k\ge 1$, consider the family of joint distributions $\gamma_k^\vep$ from \eqref{eq:coupling}. Then, the transition probabilities for the Sinkhorn Markov chain can be defined inductively as follows. For any $k\ge 0$, suppose $(X_k^{\vep},Y_k^{\vep})=(x,y)$. Then sample $X_{k+1}^{\vep}$ from the conditional distribution of $X|Y=y$ under $\gvp_{k+1}$. Now  suppose $X_{k+1}^\vep=x'$. Then, sample $Y^{\vep}_{k+1}$ similarly from the other conditional, i.e., $Y|X=x'$ under $\gvp_{k+1}$. Then $(X_k^{\vep},\ k\ge 0)$ forms a time-inhomogeneous Markov chain which we will call the \emph{Sinkhorn Markov chain}.
\end{defn}

The transition kernel for the Markov chain in \cref{prop:mchn} can be written out in terms of the conditional distributions under $(\gvp_{k+1},\ k\ge 0)$. In particular, note that 
\begin{equation}\label{eq:con1}
p_{Y|X}\gvp_{k+1}(y|x)=\exp\left(\frac{1}{\vep}\langle x,y\rangle-\frac{1}{\vep}\opV[u_k^{\vep}](y)-\frac{1}{\vep}\opS[u_k^{\vep}](x)-g(y)\right),
\end{equation}
and 
\begin{equation}\label{eq:con2}
p_{X|Y}\gvp_{k+1}(x|y)=\exp\left(\frac{1}{\vep}\langle x,y\rangle-\frac{1}{\vep}\opV[u_k^{\vep}](y)-\frac{1}{\vep}u_k^{\vep}(x)-f(x)\right).
\end{equation}
Then the transition kernel can be written as
\begin{align}\label{eq:tranker}
    \Pr(X^{\vep}_{k+1}\in \,dx|X^{\vep}_k=z)&=\int_{\R^d} p_{Y|X}\gvp_{k}(y|z)p_{X|Y}\gvp_{k+1}(x|y)\,dy.
\end{align}

The following is easy to check. We provide a brief proof below.

\begin{prop}\label{prop:Sinmar}
    The marginals of the \emph{Sinkhorn Markov chain} are distributed according to $\mk{k}$ for $k\ge 1$.
\end{prop}

\begin{proof}
    By \eqref{eq:tranker}, for $k\ge 1$, we get:
    \begin{align*}
        \mathbb{P}(X_{k}^{\vep}\in \,dx)&=\int \mathbb{P}(X_{k}^{\vep}\in \, dx| X_{k-1}^{\vep}=z)p_{X}\gvp_{k-1}(z)\,dz\\ &=\int \int p_{Y|X}\gvp_{k-1}(y|z)p_{X|Y}\gvp_{k}(x|y) p_{X}\gvp_{k-1}(z)\,dy\,dz\\ &=\int p_{X|Y}\gvp_{k}(x|y)\left(\int \gvp_{k-1}(z,y)\,dz\right)\,dy\\ &=\int \exp\left(\frac{1}{\vep}\langle x,y\rangle - \frac{1}{\vep}\opV[u_{k-1}^{\vep}](y)-\frac{1}{\vep}u_{k-1}^{\vep}(x)-f(x)-g(y)\right)\,dy\\ &=\exp\left(\frac{1}{\vep}(u_{k}^{\vep}(x)-u_{k-1}^{\vep}(x))-f(x)\right)=\mk{k}(x).
    \end{align*}
    For the fourth equality, we use \eqref{eq:con2}. The fifth equality uses \eqref{eq:curtaincall}. The rest follows using elementary manipulations of conditional probabilities.

\end{proof}

It is natural to ask if the Sinkhorn Markov chain in \cref{prop:mchn} has a diffusive limit as $\vep\rightarrow 0+$. In fact, the drift and the diffusion coefficients in the SDE \eqref{eq:diffSDE} correspond to the leading order terms of the conditional expectation and the conditional variance of the one-step increment of this Markov chain. Hence, it can be conjectured that the Sinkhorn diffusion is indeed the limit of the Sinkhorn Markov chain, under suitably strong assumptions. However, the proof of this convergence appears challenging and is not taken up here. The theorem below proves the weaker statement of one-dimensional marginal convergence, i.e., the convergence of $\rho_k^{\vep}$s according to \cref{prop:Sinmar}.

\begin{thm}\label{thm:convergence}
    Suppose that Assumption \ref{asn:solcon} holds. Fix $T>0$. For $t\in [0,T]$, recall from \cref{thm:existlin} that $\rho_{t}=\exp(-h_t)=(\nabla w_t)_{\#} e^{-g}$ satisfies the Sinkhorn flow in \eqref{eq:velocity}. Then the following holds:
    
    \begin{equation}\label{eq:step1show}
    \sup_{k\ge 1\ : \ k\vep\le T}\wass_2^2(\rho_{k\vep},\mk{k})\le C_T \vep,
    \end{equation}
    where $C_T$ is a constant depending on $T$ and the constants implicit in Assumption \ref{asn:solcon}. This implies, in particular,  
    $$ \mk{\lfloor T/\vep \rfloor}\to \rho_T, \qquad \mbox{as}\ \vep\to 0$$
    for every fixed $T>0$, in the topology of weak convergence.
\end{thm}

\begin{proof}
{\color{black}    
We set up some notation first. Let $(u_{k\vep})_{k\ge 0}$ and $(w_{k\vep}=u^*_{k\vep})_{k\ge 0}$ denote sequences corresponding to the PMA process \eqref{eq:pma} and the dual PMA process \eqref{eq:dualPMA}, restricted to time points $t=0,\vep,2\vep,\ldots $. Throughout this proof, we will write $C>0$ for a generic constant depending on $f$, $g$, $T$, $d$, but crucially not on $\vep$. Note that this constant may change from one line to another.

To establish \eqref{eq:step1show}, we also define
\[
\tilde{\xi}_{k\vep}:=\left( \nabla w_{k\vep},\mathrm{id}\right)_{\#} e^{-g},
\]
which is the law of $(Y^{w_t},Y)$ where $Y \sim e^{-g}$. Therefore $p_X \tilde{\xi}_{k\vep}=\rho_{k\vep}$. Finally, recall the definition of $\gamma^{\vep}_{k+1}(x,y)$ from \eqref{eq:coupling}.  

\vspace{0.1in}

\emph{Proof of \eqref{eq:step1show}.} In this proof, we will establish the following result:

\begin{align}\label{eq:jointcon}
\wass_2^2(\gamma_{k+1}^{\vep},\tilde{\xi}_{k\vep})\le C\vep.
\end{align}
As the $2$-Wasserstein distance between the marginals is smaller, by \eqref{eq:jointcon}, we have $\wass_2^2(p_X \gamma_{k+1}^{\vep}, \ p_X \tilde{\xi}_{k\vep})\le C\vep$. Further, As $p_X \gvp_{k+1}=\rho_k^{\vep}$ (by \eqref{eq:curtaincall}) and $p_X \tilde{\xi}_{k\vep}=\rho_{k\vep}$ (as argued above), the conclusion in \eqref{eq:step1show} will follow. Therefore, it only remains to show \eqref{eq:jointcon}.

\vspace{0.05in}

\emph{Proof of \eqref{eq:jointcon}.} We will break this proof down into $3$ steps. 

\noindent \emph{Step (a).} We will show that 
\begin{align}\label{eq:stepa}
\wass_2^2(\gamma_{k+1}^{\vep},\tilde{\xi}_{k\vep})\le \E_{\gamma_{k+1}^{\vep}} \lVert X-Y^{w_{k\vep}}\rVert^2.
\end{align}

\noindent \emph{Step (b).} The right hand side of \eqref{eq:stepa} is technically challenging to work with as it involves $\gvp_{k+1}$, which in turn involves the Sinkhorn potentials $u_k^{\vep}$'s which are not smooth as $\vep\to 0$. To circumvent this, we define a surrogate probability measure replacing $u_k^{\vep}$'s in the right hand side of \eqref{eq:stepa} with $u_{k\vep}$'s from the PMA. To wit, define 
$$
\xi^{\vep}_{k+1}(x,y):=\exp\left(\frac{1}{\vep}\langle x,y\rangle - \frac{1}{\vep}u_{k\vep}(x)-\frac{1}{\vep}\opV[u_{k\vep}](y)-f(x)-g(y)\right),
$$
We will show in step (b) that: 
\begin{align}\label{eq:stepb}
\E_{\gamma_{k+1}^{\vep}} \lVert X-Y^{w_{k\vep}}\rVert^2\le C\E_{\xi_{k+1}^{\vep}} \lVert X-Y^{w_{k\vep}}\rVert^2.
\end{align}

\noindent \emph{Step (c).} In the final step, we will leverage the smoothness of $(u_{k\vep})$'s from \cref{asn:solcon} to show that 

\begin{align}\label{eq:stepc}
\sup_{k:\ k\vep\le T} \E_{\xi_{k+1}^{\vep}} \lVert X- Y^{w_{k\vep}}\rVert^2\le C\vep.
\end{align}

Clearly, combining \eqref{eq:stepa}, \eqref{eq:stepb} and \eqref{eq:stepc} establishes \eqref{eq:jointcon}, thereby completing the proof. 

\vspace{0.05in}

\emph{Proof of step (a).} Recall that $p_Y \gvp_{k+1}=e^{-g}$ (from \eqref{eq:curtaincall}) and $p_Y \xi_{k+1}^{\vep}=e^{-g}$ (by definition). With this in mind, construct a coupling $\tilde{\pi}_{k+1}^{\vep}$ between $\gamma_{k+1}^{\vep}$ and $\tilde{\xi}_{k\vep}$ as follows: sample $(X,Y)\sim \gamma_{k+1}^{\vep}$ and let $\tilde{\pi}_{k+1}^{\vep}$ be the law of $(X,Y,Y^{w_{k\vep}},Y)$. By definition of $2$-Wasserstein distance (see \eqref{eq:2wass}), we then have:
\begin{align}\label{eq:jointcon1}
    \wass_2^2(\gamma_{k+1}^{\vep},\tilde{\xi}_{k\vep})\le \E_{\tilde{\pi}_{k+1}^{\vep}} \lVert X-Y^{w_{k\vep}}\rVert^2=\E_{\gamma_{k+1}^{\vep}} \lVert X-Y^{w_{k\vep}}\rVert^2.
\end{align}
This establishes step (a).

\vspace{0.05in} 

\emph{Proof of step (b).} For proving step (b), we need the following  preparatory lemma. 

\begin{lmm}\label{lem:erboundmain}
Suppose \cref{asn:solcon} holds. Define $$a_k^{\vep}(x):=\frac{1}{\vep}(u_k^{\vep}-u_{k\vep})(x),\quad b_k^{\vep}(y):=\frac{1}{\vep}(\opV[u_k^{\vep}]-\opV[u_{k\vep}])(y).$$
Then there exists $C>0$ such that 
\begin{align*}
    \sup_{k:\ k\vep\le T} \left(\lVert a_k^{\vep}\rVert_{\infty}+\lVert b_k^{\vep}\rVert_{\infty}\right)\le C.
\end{align*}
\end{lmm}

From the above definitions of $a_k^{\vep}$ and $b_k^{\vep}$, the following relation is immediate:
\begin{align}\label{eq:funrel}
\gamma^{\vep}_{k+1}(x,y)= \xi^{\vep}_{k+1}(x,y) \exp\left(-a_k^\vep(x) - b_k^\vep(y)\right).
\end{align}
An application of \cref{lem:erboundmain} then yields
$$\gamma_{k+1}^{\vep}\le C\xi_{k+1}^{\vep}.$$
This readily implies \eqref{eq:stepb}. We remind the reader here that the constant $C$ is changing from line to line. Therefore, $C$ in the above display is not the same as that in \cref{lem:erboundmain}.

\vspace{0.05in}

\emph{Proof of step (c).} To establish step (c), we need the following lemma. 
\begin{lmm}\label{lem:pmaboundmain}
Suppose \cref{asn:solcon} holds and $\vep \in (0,1)$.  
Then there exists $C>0$ such that 
\begin{align*}
    \sup_{\substack{k:\ k\vep\le T,\\ y\in \R^d}} \left| \opV[u_{k\vep}](y)-w_{k\vep}(y)-\frac{\vep d}{2}\log{(2\pi\vep)}+\vep f(y^{w_{k\vep}})-\frac{\vep}{2}\ldet\left(\frac{\partial y^{w_{k\vep}}}{\partial y\hfill}\right) \right|\le C\vep^2.
\end{align*}
\end{lmm}

The following function (often referred to as the \emph{Bregman divergence function}) will help us simplify notation in the proof of step (c).
\begin{equation}\label{eq:bregdiv}
\mathcal{D}[u_{k\vep}](x|y):=u_{k\vep}(x)+w_{k\vep}(y)-\langle x,y\rangle\ge 0.
\end{equation}

Now the left hand side of \eqref{eq:stepc} (barring the supremum) simplifies as 

\begin{align*}
&\;\;\;\;\E_{\xi_{k+1}^{\vep}} \lVert X-Y^{w_{k\vep}}\rVert^2\nonumber \\ &=\int \int \lVert x-y^{w_{k\vep}}\rVert^2 \exp\left(\frac{1}{\vep}\langle x,y\rangle - \frac{1}{\vep} u_{k\vep}(x)-\frac{1}{\vep}\opV[u_{k\vep}](y)-f(x)-g(y)\right)\,dx\,dy\nonumber \\ &\le C \int \int \frac{1}{(2\pi\vep)^{d/2}}\sqrt{\mathrm{det}\left(\frac{\partial y\hfill}{\partial y^{w_{k\vep}}}\right)}\lVert x-y^{w_{k\vep}}\rVert^2\times \\ \;\;\;\;&\exp\bigg(-\frac{1}{\vep}\mathcal{D}[u_{k\vep}](x|y)+f(y^{w_{k\vep}})-f(x)-g(y)\bigg)\,dx\,dy
\end{align*}
where the last inequality follows from \cref{lem:pmaboundmain}. 

\noindent By a Taylor expansion of $u_{k\vep}(x)$ around the point $y^{w_{k\vep}}$, coupled with \cref{asn:solcon}, (i), we get:
\begin{equation}\label{eq:estimpf1}
\mathcal{D}[u_{k\vep}](x|y)\ge \frac{A_T}{2}\lVert x-y^{w_{k\vep}}\rVert^2.
\end{equation}

Combining the two displays above with \cref{asn:solcon}, part (ii),  it follows that: 

\begin{align*}
&\;\;\;\;\;\E_{\xi_{k+1}^{\vep}} \lVert X-Y^{w_{k\vep}}\rVert^2\\ &\le \frac{C B_T^{\frac{d}{2}}}{(2\pi\vep)^{\frac{d}{2}}} \int \int \lVert x-y^{w_{k\vep}}\rVert^2 \exp\left(-\frac{A_T}{2\vep}\lVert x-y^{w_{k\vep}}\rVert^2+f(y^{w_{k\vep}})-f(x)-g(y)\right)\,dx\,dy.
\end{align*}

Next, we focus on the inner integral above (the one with respect to $x$). We drop the $e^{-g(y)}$ term and adjust some constants, all of which are free of $x$ and reproduce the rest below:

\begin{align*}
    &\;\;\;\;\frac{1}{(2\pi\vep A_T)^{\frac{d}{2}}} \int \lVert x-y^{w_{k\vep}}\rVert^2 \exp\left(-\frac{A_T}{2\vep}\lVert x-y^{w_{k\vep}}\rVert^2+f(y^{w_{k\vep}})-f(x)\right)\,dx\\ &\le \frac{1}{(2\pi\vep A_T)^{\frac{d}{2}}} \int \lVert x-y^{w_{k\vep}}\rVert^2 \exp\left(-\frac{A_T}{2\vep}\lVert x-y^{w_{k\vep}}\rVert^2+\lVert x-y^{w_{k\vep}}\rVert_1 \lVert \nabla f\rVert_{\infty}\right)\,dx\\ &=\E\left[\lVert Z_{\vep}-y^{w_{k\vep}}\rVert^2 \exp\left(\lVert Z_{\vep}-y^{w_{k\vep}}\rVert_1 \lVert \nabla f\rVert_{\infty}\right)\right],
\end{align*}
where $Z_{\vep,y}\sim N(y^{w_{k\vep}},\vep A_T^{-1} \mathrm{I}_d)$. 

As $\lVert Z_{\vep,y}-y^{w_{k\vep}}\rVert$ is distributed according to a $\sqrt{\vep A_T^{-1}}\chi_d$ random variable. By a standard $\chi_d$ tail bound (see example \cite[Lemma 1]{laurent2000adaptive}), it follows that:
$$\E\left[\lVert Z_{\vep}-y^{w_{k\vep}}\rVert^2 \exp\left(\lVert Z_{\vep}-y^{w_{k\vep}}\rVert_1 \lVert \nabla f\rVert_{\infty}\right)\right]\le C\vep.$$
By combining the above observations, we then get:
$$\E_{\xi_{k+1}^{\vep}} \lVert X-Y^{w_{k\vep}}\rVert^2\le C\vep \int \exp(-g(y))\,dy=C\vep.$$
As the above bound is uniform over $k$ such that $k\vep\le T$, the conclusion in \eqref{eq:stepc} follows.
}
\end{proof}

\noindent In other words, \cref{thm:convergence} shows that the marginals from the Sinkhorn algorithm \eqref{eq:sinkupdt} converge to the marginals of the limiting dynamics governed by the Sinkhorn PDE (see \eqref{eq:velocity}). In light of \cref{sec:mirror}, the above result shows the connection between the Sinkhorn algorithm and the Wasserstein mirror flow of entropy (in continuum) with respect to the mirror $U(\cdot)$ which maps $\rho\mapsto \frac{1}{2}\wass_2^2\left(\rho, e^{-g}\right)$ (see \eqref{eq:mirror}).

\begin{remark}
    Recall that $\gamma_{k+1}^{\vep}$ is the Schr\"{o}dinger Bridge coupling between $\mk{k}$ and $e^{-g}$. From the proof of \cref{thm:convergence} (see in particular \eqref{eq:jointcon}), we actually obtain a quantitative convergence bound on this coupling too. In particular, it holds that $\wass_2^2(\gamma_{k+1}^{\vep},(\nabla w_{k\vep},\mathrm{id})_{\#} e^{-g})\le \vep$ uniformly over $1\le k\le \lfloor T/\vep\rfloor$.
\end{remark}

\section{Proof of technical results}\label{sec:pfres}
In this section we will prove the auxiliary results from the main paper whose proofs we deferred, except for Lemmas \ref{lem:erboundmain} and \ref{lem:pmaboundmain} (which will be proved in the next section). 
\begin{proof}[Proof of~\cref{lem:tensorel}]
By~\eqref{eq:conjrel}, the following matrix relationship holds:
\begin{align*}
\left(\secphx\right)\left(\secph\right)=I_d.
\end{align*}
Therefore, by expanding the matrix product, for any $i,j\in [d]^2$, we get:
\begin{align}\label{eq:tensorel1}
\sum_{\ell} \secphxil \cdot \secphlj=\sum_{\ell} \left(\secphx\right)_{i,\ell}\left(\secph\right)_{\ell,j}=\delta_{i,j},
\end{align}
where $\delta_{i,j}=0$ if $i\neq j$ and $\delta_{i,i}=1$. 

\noindent Fix $k\in [d]$. By differentiating the LHS  of~\eqref{eq:tensorel1} with respect to $\xsph_k$, i.e., applying the operator $\frac{\partial}{\partial \xsph_{k}}$, we get:
\begin{align}\label{eq:tensorel2}
\frac{\partial}{\partial \xsph_{k}}\left(\sum_{\ell} \secphxil \cdot \secphlj\right)&=\sum_{\ell} \frac{\partial}{\partial \xsph_{k}} \left(\secphxil\right)\cdot \secphlj+\sum_{\ell} \secphxil \cdot \frac{\partial}{\partial \xsph_k} \left(\secphlj\right)\nonumber \\
&=\sum_{\ell,m} \frac{\partial}{\partial x_{m}}\left(\secphxil\right)\cdot\secphmk \cdot \secphlj+\sum_{\ell} \secphxil \cdot \frac{\partial}{\partial \xsph_k} \left(\secphlj\right)\nonumber \\&=\sum_{\ell,m} \frac{\partial}{\partial x_{\ell}}\left(\secphxim\right)\cdot\secphmk \cdot \secphlj+\sum_{\ell} \secphxil \cdot \frac{\partial}{\partial \xsph_k} \left(\secphjl\right).
\end{align}
Here the first equality uses the product rule of derivatives. The second equality uses the chain rule of derivatives on the first term (the second term is intact here). Finally the third display follows by noting that $\frac{\partial}{\partial x_{m}}\left(\secphxil\right)=\frac{\partial}{\partial x_{\ell}}\left(\secphxim\right)$ and $\frac{\partial x}{\partial x^{\phi}}$ is a symmetric matrix (both of which follow from the $\diffcont^2$ diffeomorphism assumption on $\nabla \phi$). Next note that the derivative of the right hand side of \eqref{eq:tensorel1} with respect to $x_k^{\phi}$ is $0$. With this observation, by  combining~\eqref{eq:tensorel1}~and~\eqref{eq:tensorel2}, we get:
\begin{align*}
\sum_{\ell} \secphxil \cdot \frac{\partial}{\partial \xsph_k} \left(\secphjl\right)=-\sum_{\ell,m} \frac{\partial}{\partial x_{\ell}}\left(\secphxim\right)\cdot\secphmk \cdot \secphlj.
\end{align*}
By choosing $k=i$, summing up over $i$, we get:
\begin{align}\label{eq:tenso1}
\sum_{\ell} \frac{\partial}{\partial x_{\ell}}\left(\secphjl\right)=-\sum_{i,\ell,m} \frac{\partial}{\partial x_{\ell}}\left(\secphxim\right)\cdot\secphmi \cdot \secphlj.
\end{align}

\noindent Note that the left hand side of \eqref{eq:tenso1} is same as the right hand side of \eqref{eq:tensorelpf}. Let us now show that the right hand side of \eqref{eq:tenso1} matches the left hand side of \eqref{eq:tensorelpf}.

\noindent To achieve this, note that for a symmetric positive definite matrix $A$, by \cite[Section A.4.1]{boyd2004convex}, we have $\nabla_A \ldet(A)= A^{-1}$. Here the gradient is computed entry-wise. By using the above observation with $A=\frac{\partial x^{\phi}}{\partial x\hfill}$ (so $A^{-1}=\frac{\partial x\hfill}{\partial x^{\phi}}$), we get
\begin{align*}
\frac{\partial}{\partial \xsph_j} \ldet \left(\secphx\right)&= \sum_{i,m} \frac{\partial}{\partial \xsph_j} \left( \secphxim \right) \secphmi &= \sum_{i,\ell,m} \frac{\partial}{\partial x_{\ell}}\left(\secphxim\right)\cdot\secphlj\cdot\secphmi .
\end{align*}
The last equality above follows using the chain rule of derivatives. Coupling the above observation with \eqref{eq:tenso1}  completes the proof.
\end{proof}

\begin{proof}[Proof of \cref{lem:dualPMA}]
As $w_t=u_t^*$ with $u_t$ being the solution from \eqref{eq:pma}, given $y\in\R^d$, the following holds:
$$w_t(y)+u_t(y^{w_t})-\langle y, y^{w_t}\rangle=0.$$
By taking partial derivatives with respect to $t$ on both sides above, we further have:
\begin{align}\label{eq:dualPM1}
    \frac{\partial w_t}{\partial t\hfill}(y)+\frac{\partial}{\partial t}(u_t(y^{w_t}))-\iprod{y,\frac{\partial y^{w_t}}{\partial t\hfill}}=0
\end{align}
The second term in the left hand side of \eqref{eq:dualPM1} needs further simplification. To wit, by using the chain rule of derivatives, we get:
$$\frac{\partial}{\partial t}(u_t(y^{w_t}))=\frac{\partial}{\partial t}u_t(y^{w_t})+\iprod{(y^{w_t})^{u_t},\frac{\partial y^{w_t}}{\partial t\hfill}}=\frac{\partial}{\partial t}u_t(y^{w_t})+\iprod{y,\frac{\partial y^{w_t}}{\partial t\hfill}}.$$
In the last equality above, we have used the observation that $\nabla u_t(\nabla w_t(y))=y$. By coupling the above observation with \eqref{eq:dualPM1} we get
$$\frac{\partial w_t}{\partial t\hfill}(y)=-\frac{\partial}{\partial t}u_t(y^{w_t})=g((y^{w_t})^{u_t})-f(y^{w_t})-\ldet\left(\frac{\partial x^{u_t}}{\partial x\hfill}\right)\bigg|_{x=y^{w_t}}.$$
In the last equality we have used the PMA \eqref{eq:pma}. The conclusion of the lemma now follows by combining the above display with the following observations:
$$(y^{w_t})^{u_t}=y,\quad\quad \ldet\left(\frac{\partial x^{u_t}}{\partial x\hfill}\right)\bigg|_{x=y^{w_t}}=-\ldet\left(\frac{\partial y^{w_t}}{\partial y\hfill}\right).$$
\end{proof}

\begin{proof}[Proof of  \cref{lem:convexcall}]
By \eqref{eq:labelgrad}, $\nabla w_t(y)+\delta v_t(y^{w_t})$ is the gradient of the following function with respect to $y$
$$\Lambda_t(y):=w_t(y)+\delta \left(g(y)-f(y^{w_t})+\ldet\left(\frac{\partial y^{w_t}}{\partial y\hfill}\right)\right).$$
It remains to prove that $y\mapsto \Lambda_t(y)$ is convex. Towards this direction, we first assume the following claim and complete the proof.
\begin{align}\label{eq:eigbound}
    \sup_y \lVert \nabla^2 \Lambda_t(y)\rVert_{\mathrm{op}}<\infty,
\end{align}
where $\lVert \cdot \rVert_{\mathrm{op}}$ denotes the $L^2$ operator norm of a matrix (see~\cref{tab:table3}). Recall from \cref{asn:solcon}, part (i) that $\inf_y \lmn\left(\frac{\partial y^{w_t}}{\partial y\hfill}\right)>0$. Therefore there exists $\delta>0$ such that $$\delta < \ \frac{\inf_y \lmn\left(\frac{\partial y^{w_t}}{\partial y\hfill}\right)}{2\sup_y \lVert \nabla^2 \Lambda_t(y)\rVert_{\mathrm{op}}}.$$ 
By using Weyl's inequality (see \cite[Theorem 3.3.16]{horn1994topics}), we then get:
$$\inf_y \lmn(\Lambda_t(y))\ge \inf_y \lmn\left(\frac{\partial y^{w_t}}{\partial y\hfill}\right)-\delta \sup_y\lVert \Lambda_t(y)\rVert_{\mathrm{op}}>\frac{1}{2}\inf_y \lmn\left(\frac{\partial y^{w_t}}{\partial y\hfill}\right)>0.$$

This establishes the convexity of $y\mapsto \Lambda_t(y)$ and  completes the proof of \cref{lem:convexcall}. Therefore, it only remains to prove \eqref{eq:eigbound}. 

\emph{Proof of \eqref{eq:eigbound}.} By \cref{asn:solcon}, parts (i) and (iii), $\sup_y \lVert w_t(y)\rVert_{\mathrm{op}}<\infty$ and $\sup_y \lVert \nabla^2 g(y)\rVert_{\mathrm{op}}<\infty$. Similarly for $i,j\in [d]^2$, we get 
$$\frac{\partial^2}{\partial y_i\partial y_j}(f(y^{w_t}))=\sum_{\ell}\frac{\partial^3}{\partial y_i \partial y_j \partial y_{\ell}} w_t(y) \frac{\partial}{\partial y_{\ell}} f(y^{w_t})+\sum_{\ell,m}\left(\frac{\partial y_{\ell}^{w_t}}{\partial y_j\hfill}\right)\left(\frac{\partial y_m^{w_t}}{\partial y_i\hfill}\right)\frac{\partial^2}{\partial y_m\partial y_{\ell}}f(y^{w_t}).$$
Here we have used a combination of the product rule and the chain rule. Similar calculations were done in the proofs of \cref{thm:existlin} in the main paper, and so we skip the details for brevity. By the uniform boundedness of the first two derivatives of $f$ and the first $3$ derivatives of $w_t$ from \cref{asn:solcon}, part (iii) (also see \cref{rem:dualasn}), we immediately get:
$$\max_{i,j\in [d]^2}\sup_{y}\bigg|\frac{\partial^2}{\partial y_i\partial y_j}(f(y^{w_t}))\bigg|<\infty.$$
A similar computation also yields
\begin{align*}
    \frac{\partial^2}{\partial y_i \partial y_j}\left(\ldet\left(\frac{\partial y^{w_t}}{\partial y\hfill}\right)\right)&=\sum_{k,\ell}\left(\frac{\partial y_k\hfill}{\partial y_{\ell}^{w_t}}\right)\frac{\partial^4}{\partial y_i\partial y_j\partial y_k\partial y_{\ell}} w_t(y)\\ &\;\;\;+\sum_{k,\ell,m}\frac{\partial^3}{\partial y_k\partial y_{\ell}\partial y_m}u_t(y^{w_t})\frac{\partial^3}{\partial y_k\partial y_{\ell}\partial y_j}w_t(y)\left(\frac{\partial y_m^{w_t}}{\partial y_i}\right).
\end{align*}
By the uniform boundedness of the first two derivatives of $f$ and the first $4$ derivatives of $w_t$ from \cref{asn:solcon}, part (iii) (also see \cref{rem:dualasn}), we immediately get:
$$\max_{i,j\in [d]^2}\sup_y \Bigg|\frac{\partial^2}{\partial y_i \partial y_j}\left(\ldet\left(\frac{\partial y^{w_t}}{\partial y\hfill}\right)\right)\Bigg|<\infty.$$
Combining these observations, \eqref{eq:eigbound} follows.
\end{proof}


\section{Proof of Lemmas \ref{lem:erboundmain} and \ref{lem:pmaboundmain}}\label{sec:mainresultlems}
Recall the definitions of $\opU$ and $\opV$ from \eqref{eq:basedef}. 
\cref{lem:pmaboundmain} is a Laplace approximation applied to the integral operator $\opV$. A similar estimate can be found in \cite[Lemma 4.2]{berman2020} under different assumptions. On the other hand, \cref{lem:erboundmain} is a triangular approximation argument which is motivated from the proof of \cite[Lemma 4.4]{berman2020}. We will actually prove Lemmas \ref{lem:erboundmain} and \ref{lem:pmaboundmain} in the reverse order. We need to set up some notation first. For $k\ge 0$, consider the Bregman divergence from \eqref{eq:bregdiv} for the convex function $u_{k\vep}(\cdot)$ and $x,y\in\R^d$, 
\[
\mcD[u_{k\vep}](x|y) = u_{k\vep}(x) + u_{k\vep}^*(y) - \langle x,y\rangle.
\]
Note that $\mcD[u_{k\vep}^*](y|x)=\mcD[u_{k\vep}](x|y)$. Also given a sufficiently smooth function $h:\R^d\to\R$ and any $r\in \mathbb{N}$, define 
    \begin{align}\label{eq:taylornot}
        T[h:r](x|y):=\sum_{|\alpha|=r}\frac{D^{\alpha} h(y)}{\alpha!}(x-y)^{\alpha},
    \end{align}
    and 
    $$R[h:r](x|y):=\sum_{|\beta|=r}\frac{r}{\beta!}\left(\int_0^1 (1-t)^{r-1} D^{\beta} h(y+t(x-y))\,dt\right)(x-y)^{\beta}-T[h:r](x|y).$$
    Here, the $D^{\alpha}$ (or $D^{\beta}$) operators denote the standard multivariate differential operators given a nonnegative multi-index $\alpha=(\alpha_1,\ldots ,\alpha_d)$ (respectively $\beta=(\beta_1,\ldots ,\beta_d)$). Also $|\alpha|=\sum_{i=1}^d \alpha_i$, $|\beta|=\sum_{i=1}^d \beta_i$ and, as usual, $(x-y)^\alpha=\prod_{i=1}^d (x_i-y_i)^{\alpha_i}$ and so on.
    
    In other words, $T[h:r](x|y)$ denotes the $r$-th polynomial in the Taylor series expansion of $h$ at the point $x$ around the point $y$. On the other hand, $R[h:r](x|y)$ denotes the corresponding remainder term after a Taylor expansion of the function $h$ up to the $r$-th order term (at the point $x$  around the point $y$).

    Finally, we note the following canonical estimates:
    \begin{equation}\label{eq:gradbound3}
    |T[h:r](x|y)|\le  C_r\lVert \nabla^r h\rVert_{\infty}\lVert x-y\rVert^{r},
    \end{equation}
    and
    \begin{equation}\label{eq:gbd4}
    |R[h:r](x|y)|\le C_r \lVert x-y\rVert^r \sup_{\lVert z_1-z_2\lVert \le \lVert x-y\rVert}\lVert \nabla^r h(z_1)-\nabla^r h(z_2)\rVert_{\infty},
    \end{equation}
    where $C_r>1$ is a universal constant (depending on $d$) that is free of $h$, $x$, and $y$. Here $\lVert \nabla^r h\rVert_{\infty}$ denotes the maximum of the supremum norms of every component function in the $r$-th order multiderivative.

\begin{proof}[Proof of \cref{lem:pmaboundmain}]
Throughout the proof we always assume $k\vep\le T$. Further, $C$ will denote constants (possibly different in various steps) depending only on $d$ and all the other constants implicit in \cref{asn:solcon}. 

Let $x,y\in\R^d$. Recall from \eqref{eq:estimpf1} that 
$\mathcal{D}[u_{k\vep}](x|y)\ge \frac{A_T}{2}\lVert x-y^{w_{k\vep}}\rVert^2$, for some positive constant $A_T$. We are now in a position to simplify $\opV$. 
    
    \vspace{0.1in}
    
    By an algebraic identity using the definition of $\opV$ from \eqref{eq:basedef}, we observe that 
    \begin{align}\label{eq:simpl1}
       &\;\;\;\;\frac{\sqrt{\mathrm{det}(\nabla^2 u_{k\vep}(y^{w_{k\vep}}))}}{(2\pi\vep)^{\frac{d}{2}}}\exp\left(\frac{1}{\vep}\opV[u_{k\vep}](y)-\frac{1}{\vep}w_{k\vep}(y)+f(y^{w_{k\vep}})\right)\nonumber \\ &=\frac{\sqrt{\mathrm{det}(\nabla^2 u_{k\vep}(y^{w_{k\vep}}))}}{(2\pi\vep)^{\frac{d}{2}}}\int \exp\bigg(-\frac{1}{\vep}\mcD[u_{k\vep}](x|y)-f(x)+f(y^{w_{k\vep}})\bigg)\,dx.
    \end{align}
    Define 
    $$r_{\vep}:=\sqrt{-40d A_T^{-1} \vep\log{\vep}},\qquad \mbox{for}\ \vep\in (0,1/2).$$
    We now split the integral in \eqref{eq:simpl1} into two complementary domains: the first one is an integral over $B_{r_{\vep}}(y^{w_{k\vep}})$ and the second one is over $B^c_{r_{\vep}}(y^{w_{k\vep}})$. We will show that 
    \begin{equation}\label{eq:showsmall}
        \sup_y \frac{\sqrt{\mathrm{det}(\nabla^2 u_{k\vep}(y^{w_{k\vep}}))}}{(2\pi\vep)^{\frac{d}{2}}}\int\limits_{B^c_{r_{\vep}}(y^{w_{k\vep}})} \exp\bigg(-\frac{1}{\vep}\mcD[u_{k\vep}](x|y)-f(x)+f(y^{w_{k\vep}})\bigg)\,dx \le C \vep^{10},
    \end{equation}
    and 
    \begin{equation}\label{eq:showlarge}
        \sup_y \bigg|\frac{\sqrt{\mathrm{det}(\nabla^2 u_{k\vep}(y^{w_{k\vep}}))}}{(2\pi\vep)^{\frac{d}{2}}}\int\limits_{B_{r_{\vep}}(y^{w_{k\vep}})} \exp\bigg(-\frac{1}{\vep}\mcD[u_{k\vep}](x|y)-f(x)+f(y^{w_{k\vep}})\bigg)\,dx - 1\bigg|\le C \vep.
    \end{equation}
    Let us complete the proof by assuming \eqref{eq:showsmall} and \eqref{eq:showlarge} first. To wit, by combining \eqref{eq:showsmall} and \eqref{eq:showlarge}, with \eqref{eq:simpl1}, we get:
    $$\sup_y \bigg|\frac{\sqrt{\mathrm{det}(\nabla^2 u_{k\vep}(y^{w_{k\vep}}))}}{(2\pi\vep)^{\frac{d}{2}}}\exp\left(\frac{1}{\vep}\opV[u_{k\vep}](y)-\frac{1}{\vep}w_{k\vep}(y)+f(y^{w_{k\vep}})\right)-1\bigg|\le C \vep.$$
    The conclusion of \cref{lem:pmaboundmain} then follows by using the elementary inequality $|\log{(1+x)}|\le 2|x|$ for $|x|\le 1/2$. It only remains to show \eqref{eq:showsmall} and \eqref{eq:showlarge}.
    
    \vspace{0.1in}
    
    \emph{Proof of \eqref{eq:showsmall}.} We note the following sequence of displays with line-by-line explanations to follow:
    \begin{align}\label{eq:gradbound4}
        &\;\;\;\;\frac{\sqrt{\mathrm{det}(\nabla^2 u_{k\vep}(y^{w_{k\vep}}))}}{(2\pi\vep)^{\frac{d}{2}}}\int\limits_{B^c_{r_{\vep}}(y^{w_{k\vep}})} \exp\bigg(-\frac{1}{\vep}\mcD[u](x|y)-f(x)+f(y^{w_{k\vep}})\bigg)\,dx\nonumber \\ & \le \frac{C}{(2\pi\vep)^{\frac{d}{2}}}\int\limits_{B^c_{r_{\vep}}(y^{w_{k\vep}})} \exp\bigg(-\frac{A_T}{2\vep }\lVert x-y^{w_{k\vep}}\rVert^2-f(x)+f(y^{w_{k\vep}})\bigg)\,dx\nonumber \\ &= C\E\left[\exp\left(f(y^{w_{k\vep}})-f\left(y^{w_{k\vep}}+\sqrt{\vep A_T^{-1}} Z\right)\right)\mathbf{1}\left(\sqrt{\vep A_T^{-1}}Z\in B_{r_{\vep}}^c(0)\right)\right]\nonumber \\ &\le C  \E\left[\exp\left(\sqrt{\vep A_T^{-1}}\lVert \nabla f\rVert_{\infty}\sum_{i=1}^d |Z_i|\right)\mathbf{1}\left(\sqrt{\vep A_T^{-1}}Z\in B_{r_{\vep}}^c(0)\right)\right]
    \end{align}
    Above, in the first inequality, we have used  \eqref{eq:estimpf1} and \cref{asn:solcon}, part (i) to bound $\sqrt{\mathrm{det}(\nabla^2 u_{k\vep}(y^{w_{k\vep}}))}$ from above. In the following equality, we have adjusted some constants to rewrite the integral in terms of an expectation of a standard multivariate Gaussian random variable $Z$. The next inequality follows from the elementary observation that 
    \begin{align*}
    \left|f\left(y^{w_{k\vep}}+\sqrt{\vep A_T^{-1}}Z\right)-f(y^{w_{k\vep}})\right|&=\int_0^1 \left\langle \nabla f\left(y^{w_{k\vep}}+t\sqrt{\vep A_T^{-1}}Z\right),\sqrt{\vep A_T^{-1}}Z\right\rangle\,dt\\ &\le \sqrt{\vep A_T^{-1}}\lVert \nabla f\rVert_{\infty}\sum_{i=1}^d |Z_i|.
    \end{align*}
    Next, we apply the Cauchy-Schwartz inequality in \eqref{eq:gradbound4} to get:
    \begin{align*}
    &\;\;\;\;\frac{\sqrt{\mathrm{det}(\nabla^2 u_{k\vep}(y^{w_{k\vep}}))}}{(2\pi\vep)^{\frac{d}{2}}}\int\limits_{B^c_{r_{\vep}}(y^{w_{k\vep}})} \exp\bigg(-\frac{1}{\vep}\mcD[u](x|y)-f(x)+f(y^{w_{k\vep}})\bigg)\,dx\\ &\le
    C \sqrt{\E\exp\left(\sqrt{\vep A_T^{-1}}\lVert \nabla f\rVert_{\infty} \sum_{i=1}^d |Z_i|\right)}\sqrt{\Pr(\sqrt{\vep A_T^{-1}}\lVert Z\rVert\ge r_{\vep})}\\ &\le C \sqrt{\left(\E\exp\left(\sqrt{\vep A_T^{-1}}\lVert \nabla f\rVert_{\infty}  |Z_1|\right)\right)^d}\sqrt{d\Pr\left(|Z_1|\ge r_{\vep}/\sqrt{\vep d A_T^{-1}}\right)}\le C \vep^{10}.
    \end{align*}
    To understand the inequalities in the last line above, note that by using standard Gaussian tail bounds, we have 
    $$\E\exp\left(\sqrt{\vep A_T^{-1}}\lVert \nabla f\rVert_{\infty}  |Z_1|\right)\le \E\exp\left(\sqrt{A_T^{-1}}\lVert \nabla f\rVert_{\infty}|Z_1|\right)\le C,$$
    as $\vep\le 1$. Moreover, by the union bound,
     \begin{align*}
        \Pr&(\sqrt{\vep A_T^{-1}}\lVert Z\rVert\ge r_{\vep})\le \sum_{i=1}^d \Pr\left(|Z_i|\ge r_{\vep}/\sqrt{\vep d A_T^{-1}}\right)=d \Pr\left(|Z_1|\ge r_{\vep}/\sqrt{\vep d A_T^{-1}}\right)\\ &\le 2d \exp(-r_{\vep}^2/(2\vep d A_T^{-1}))=2d \exp(20 \log{\vep})=2d \vep^{20}.
    \end{align*}
\vspace{0.1in}

   \emph{Proof of \eqref{eq:showlarge}.} 
    Set $$\tilde{Z}^{(\vep)}_{w_{k\vep},y}\sim y^{w_{k\vep}}+\sqrt{\vep}Z_{w_{k\vep},y}, \quad \mbox{where} \quad Z_{w_{k\vep},y}\sim N(0,\nabla^2 w_{k\vep}(y)).$$
    By a third order Taylor series expansion, for any $k\ge 0$:
    \begin{small}
    \begin{align}\label{eq:taylor1}
        &\mcD[u_{k\vep}](x|y)=\frac{1}{2}(x-y^{w_{k\vep}})^{\top}\nabla^2 u_{k\vep}(y^{w_{k\vep}})(x-y^{w_{k\vep}}) + T[u_{k\vep}:3](x|y^{w_{k\vep}}) + R[u_{k\vep}:3](x|y^{w_{k\vep}}).
    \end{align}
    \end{small}
    Also, by a first order Taylor expansion to the function $f$:
    \begin{align}\label{eq:taylor2}
       f(x)&=f(y^{w_{k\vep}})+T[f:1](x|y^{w_{k\vep}})+R[f:1](x|y^{w_{k\vep}}).
    \end{align}
    By \eqref{eq:taylor1} and \eqref{eq:taylor2},
    \begin{align}\label{eq:largge}
        &\;\;\;\;\frac{\sqrt{\mathrm{det}(\nabla^2 u_{k\vep}(y^{w_{k\vep}}))}}{(2\pi\vep)^{\frac{d}{2}}}\int\limits_{B_{r_{\vep}}(y^{w_{k\vep}})} \exp\left(-\frac{1}{\vep}\mcD[u_{k\vep}](x|y)-f(x)+f(y^{w_{k\vep}})\right)\,dx \nonumber \\ &=\frac{\sqrt{\mathrm{det}(\nabla^2 u_{k\vep}(y^{w_{k\vep}}))}}{(2\pi\vep)^{\frac{d}{2}}}\int\limits_{B_{r_{\vep}}(y^{w_{k\vep}})} \exp\bigg(-\frac{1}{2\vep}(x-y^{w_{k\vep}})^{\top}\nabla^2 u_{k\vep}(y^{w_{k\vep}})(x-y^{w_{k\vep}})\bigg)\times \nonumber \\ 
        &  \exp\left(-\frac{1}{\vep}T[u_{k\vep}:3](x|y^{w_{k\vep}})- \frac{1}{\vep} R[u_{k\vep}:3](x|y^{w_{k\vep}})-T[f:1](x|y^{w_{k\vep}})-R[f:1](x|y^{w_{k\vep}})\right)\,dx.
    \end{align}

    Now in order to prove \eqref{eq:showlarge}, it suffices to show that the final equality above is $1+O(\vep)$. Let us sketch the rest of the argument first. Note that in the second line above, we have the density of a $N(y^{w_{k\vep}},\vep\nabla^2 w_{k\vep}(y))$ random variable, i.e., $\tilde{Z}^{(\vep)}_{w_{k\vep},y}$ declared above. This enables us to rewrite the above integral as a Gaussian integral as follows:

    \begin{align}\label{eq:large1}
    &\E\bigg[\exp\bigg(-\frac{1}{\vep}T[u_{k\vep}:3](\tilde{Z}^{(\vep)}_{w_{k\vep},y}|y^{w_{k\vep}})- \frac{1}{\vep} R[u_{k\vep}:3](\tilde{Z}^{(\vep)}_{w_{k\vep},y}|y^{w_{k\vep}})-T[f:1](\tilde{Z}^{(\vep)}_{w_{k\vep},y}|y^{w_{k\vep}})\nonumber \\ &-R[f:1](\tilde{Z}^{(\vep)}_{w_{k\vep},y}|y^{w_{k\vep}})\bigg)\mathbf{1}(\sqrt{\vep}Z_{w_{k\vep},y}\in B_{r_{\vep}}(0))\bigg].
    \end{align}

    Next we will use two simple approximations of the exponential function as follows: for $|z|\le M$, 
    \begin{align}\label{eq:taylorapp}
        \bigg|\exp(z)-1-z\bigg|\le \frac{z^2}{2}\exp(M),\qquad \mbox{and} \qquad \bigg|\exp(z)-1\bigg|\le |z|\exp(M).
    \end{align}
    The idea is to use \eqref{eq:taylorapp} to approximate the exponential terms in \eqref{eq:large1}. In particular, the terms in \eqref{eq:large1} featuring the Taylor polynomials in the exponent, namely $\exp\big(-\frac{1}{\vep}T[u_{k\vep}:3](\tilde{Z}^{(\vep)}_{w_{k\vep},y}|y^{w_{k\vep}})\big)$ and $\exp\big(-T[f:1](\tilde{Z}^{(\vep)}_{w_{k\vep},y}|y^{w_{k\vep}})\big)$, will be approximated using the first inequality in \eqref{eq:taylorapp}. The remainder terms $\exp\big(\frac{1}{\vep} R[u_{k\vep}:3](\tilde{Z}^{(\vep)}_{w_{k\vep},y}|y^{w_{k\vep}})\big)$ and $\exp\big(-R[f:1](\tilde{Z}^{(\vep)}_{w_{k\vep},y}|y^{w_{k\vep}})\big)$ on the other hand, will be approximated using the second inequality in \eqref{eq:taylorapp}. The role of $M$ in \eqref{eq:taylorapp} will be played by an appropriate function of $r_{\vep}$ thanks to the indicator term in \eqref{eq:large1}. Let us illustrate the above description concretely using one of the aforementioned terms. To wit, note that by \eqref{eq:gradbound3}, we have:
    \begin{align}\label{eq:large2}
    &\;\;\;\;\bigg|\frac{1}{\vep}T[u_{k\vep}:3](\tilde{Z}^{(\vep)}_{w_{k\vep},y}|y^{w_{k\vep}})\mathbf{1}(\sqrt{\vep}Z_{w_{k\vep},y}\in B_{r_{\vep}}(0))\bigg|\nonumber \\ &\le \frac{C}{\vep} \lVert \sqrt{\vep}Z_{w_{k\vep},y}\rVert^3 \mathbf{1}(\sqrt{\vep}Z_{w_{k\vep},y}\in B_{r_{\vep}}(0))\le \frac{C}{\vep} r_{\vep}^3 = C\sqrt{\vep}\left(\log{\left(\frac{1}{\vep}\right)}\right)^{3/2}.
    \end{align}
    Define
    \begin{align*}
    \vartheta^{(1)}_{\vep}(\tilde{Z}^{(\vep)}_{w_{k\vep},y}|y^{w_{k\vep}})&:=\exp\left(-\frac{1}{\vep}T[u_{k\vep}:3](\tilde{Z}^{(\vep)}_{w_{k\vep},y}|y^{w_{k\vep}})\right)-1+\frac{1}{\vep}T[u_{k\vep}:3](\tilde{Z}^{(\vep)}_{w_{k\vep},y}|y^{w_{k\vep}}).
    \end{align*}
    By invoking the first inequality in \eqref{eq:taylorapp} with $M=C\sqrt{\vep}\left(\log{\left(\frac{1}{\vep}\right)}\right)^{3/2}$ as in \eqref{eq:large2}, we get:
    \begin{align}\label{eq:large11}
    &\;\;\;\;\sup_y \big|\vartheta^{(1)}_{\vep}(\tilde{Z}^{(\vep)}_{w_{k\vep},y}|y^{w_{k\vep}})\big|\mathbf{1}(\sqrt{\vep}Z_{w_{k\vep},y}\in B_{r_{\vep}}(0))\nonumber \\ &\le \sup_y \left(\frac{1}{\vep}T[u_{k\vep}:3](\tilde{Z}^{(\vep)}_{w_{k\vep},y}|y^{w_{k\vep}})\right)^2 \mathbf{1}(\sqrt{\vep}Z_{w_{k\vep},y}\in B_{r_{\vep}}(0)) \exp\left(C\sqrt{\vep}\left(\log{\left(\frac{1}{\vep}\right)}\right)^{3/2}\right)\nonumber \\&\le C \vep \left(\log{\left(\frac{1}{\vep}\right)}\right)^3.
    \end{align}
    In the last inequality we have bounded the term $\exp(C\sqrt{\vep}(\log{(1/\vep)})^{3/2})$ by a constant using $\vep\in (0,1)$. In the same vein, 
    \begin{align}\label{eq:large12}
    &\;\;\;\;\sup_y \E\left[\big|\vartheta^{(1)}_{\vep}(\tilde{Z}^{(\vep)}_{w_{k\vep},y}|y^{w_{k\vep}})\big|\mathbf{1}(\sqrt{\vep}Z_{w_{k\vep},y}\in B_{r_{\vep}}(0))\right]\nonumber \\ &\le \sup_y \E\left[\left(\frac{1}{\vep}T[u_{k\vep}:3](\tilde{Z}^{(\vep)}_{w_{k\vep},y}|y^{w_{k\vep}})\right)^2 \mathbf{1}(\sqrt{\vep}Z_{w_{k\vep},y}\in B_{r_{\vep}}(0))\right] \exp\left(C\sqrt{\vep}\left(\log{\left(\frac{1}{\vep}\right)}\right)^{3/2}\right)\nonumber \\&\le  C \sup_y \E\left[\left(\frac{1}{\vep}T[u_{k\vep}:3](\tilde{Z}^{(\vep)}_{w_{k\vep},y}|y^{w_{k\vep}})\right)^2 \right]\le C \vep .
    \end{align}
    In the last display, we have additionally used standard Gaussian tail bounds.

    \vspace{0.1in}

    We can carry out the same line of argument for the other terms in the exponential as seen in \eqref{eq:large1}. We simply produce the corresponding definitions and bounds noting that they can be obtained similarly as the bounds for $\vartheta^{(1)}_{\vep}(\tilde{Z}^{(\vep)}_{w_{k\vep},y}|y^{w_{k\vep}})$ above.

    Define 
    \begin{align*}
    \vartheta^{(2)}_{\vep}(\tilde{Z}^{(\vep)}_{w_{k\vep},y}|y^{w_{k\vep}})&:=\exp\left(- \frac{1}{\vep} R[u_{k\vep}:3](\tilde{Z}^{(\vep)}_{w_{k\vep},y}|y^{w_{k\vep}})\right)-1.
    \end{align*}
    It satisfies 
    \begin{equation}\label{eq:large21}
        \sup_y \big|\vartheta^{(2)}_{\vep}(\tilde{Z}^{(\vep)}_{w_{k\vep},y}|y^{w_{k\vep}})\big|\mathbf{1}(\sqrt{\vep}Z_{w_{k\vep},y}\in B_{r_{\vep}}(0))\le C \vep \left(\log{\left(\frac{1}{\vep}\right)}\right)^2,
    \end{equation}
    and
    \begin{equation}\label{eq:large22}
        \sup_y \E\left[\big|\vartheta^{(2)}_{\vep}(\tilde{Z}^{(\vep)}_{w_{k\vep},y}|y^{w_{k\vep}})\big|\mathbf{1}(\sqrt{\vep}Z_{w_{k\vep},y}\in B_{r_{\vep}}(0))\right]\le C \vep.
    \end{equation}
    Define 
    \begin{align*}
    \vartheta^{(3)}_{\vep}(\tilde{Z}^{(\vep)}_{w_{k\vep},y}|y^{w_{k\vep}})&:=\exp\left(- T[f:1](\tilde{Z}^{(\vep)}_{w_{k\vep},y}|y^{w_{k\vep}})\right)-1+T[f:1](\tilde{Z}^{(\vep)}_{w_{k\vep},y}|y^{w_{k\vep}}).
    \end{align*}
    It satisfies 
    \begin{equation}\label{eq:large31}
        \sup_y \big|\vartheta^{(3)}_{\vep}(\tilde{Z}^{(\vep)}_{w_{k\vep},y}|y^{w_{k\vep}})\big|\mathbf{1}(\sqrt{\vep}Z_{w_{k\vep},y}\in B_{r_{\vep}}(0))\le C \vep \log{\left(\frac{1}{\vep}\right)},
    \end{equation}
    and
    \begin{equation}\label{eq:large32}
        \sup_y \E\left[\big|\vartheta^{(3)}_{\vep}(\tilde{Z}^{(\vep)}_{w_{k\vep},y}|y^{w_{k\vep}})\big|\mathbf{1}(\sqrt{\vep}Z_{w_{k\vep},y}\in B_{r_{\vep}}(0))\right]\le C \vep.
    \end{equation}
    Finally, define 
    \begin{align*}
    \vartheta^{(4)}_{\vep}(\tilde{Z}^{(\vep)}_{w_{k\vep},y}|y^{w_{k\vep}})&:=\exp\left(-  R[f:1](\tilde{Z}^{(\vep)}_{w_{k\vep},y}|y^{w_{k\vep}})\right)-1.
    \end{align*}
    It satisfies 
    \begin{equation}\label{eq:large41}
        \sup_y \big|\vartheta^{(4)}_{\vep}(\tilde{Z}^{(\vep)}_{w_{k\vep},y}|y^{w_{k\vep}})\big|\mathbf{1}(\sqrt{\vep}Z_{w_{k\vep},y}\in B_{r_{\vep}}(0))\le C \vep \log{\left(\frac{1}{\vep}\right)},
    \end{equation}
    and
    \begin{equation}\label{eq:large42}
        \sup_y \E\left[\big|\vartheta^{(4)}_{\vep}(\tilde{Z}^{(\vep)}_{w_{k\vep},y}|y^{w_{k\vep}})\big|\mathbf{1}(\sqrt{\vep}Z_{w_{k\vep},y}\in B_{r_{\vep}}(0))\right]\le C \vep.
    \end{equation}
    It is worth noting that the bounds \eqref{eq:large31} and \eqref{eq:large32} (which correspond to approximating terms involving Taylor polynomials) will require \eqref{eq:gradbound3} in the same way as the bounds \eqref{eq:large11} and \eqref{eq:large12} (which also correspond to approximating terms involving Taylor polynomials). On the other hand, the same step will require \eqref{eq:gbd4} while obtaining the bounds \eqref{eq:large21}, \eqref{eq:large22}, \eqref{eq:large41}, and \eqref{eq:large42} (which correspond to approximating remainder terms after suitable Taylor series approximations).

    We will now use the above bounds in conjunction with \eqref{eq:largge} and \eqref{eq:large1}. First we use the definitions of $\vartheta^{(i)}_{\vep}(\tilde{Z}^{(\vep)}_{w_{k\vep},y}|y^{w_{k\vep}})$, $i=1,2,3,4$, in \eqref{eq:largge} and \eqref{eq:large1} to note that:

    \begin{align}\label{eq:gaussbreak}
        &\;\;\;\;\frac{\sqrt{\mathrm{det}(\nabla^2 u_{k\vep}(y^{w_{k\vep}}))}}{(2\pi\vep)^{\frac{d}{2}}}\int\limits_{B_{r_{\vep}}(y^{w_{k\vep}})} \exp\left(-\frac{1}{\vep}\mcD[u_{k\vep}](x|y)-f(x)+f(y^{w_{k\vep}})\right)\,dx\nonumber \\ 
        &=\E\bigg[\left(1-\frac{1}{\vep}T[u_{k\vep}:3](\tilde{Z}^{(\vep)}_{w_{k\vep},y}|y^{w_{k\vep}})+\vartheta^{(1)}_{\vep}(\tilde{Z}^{(\vep)}_{w_{k\vep},y}|y^{w_{k\vep}})\right)\left(1+\vartheta^{(2)}_{\vep}(\tilde{Z}^{(\vep)}_{w_{k\vep},y}|y^{w_{k\vep}})\right)\nonumber \\ &\qquad \left(1-T[f:1](\tilde{Z}^{(\vep)}_{w_{k\vep},y}|y^{w_{k\vep}})+\vartheta^{(3)}_{\vep}(\tilde{Z}^{(\vep)}_{w_{k\vep},y}|y^{w_{k\vep}})\right)\left(1+\vartheta^{(4)}_{\vep}(\tilde{Z}^{(\vep)}_{w_{k\vep},y}|y^{w_{k\vep}})\right)\nonumber \\ &\qquad \qquad \mathbf{1}(\sqrt{\vep}Z_{w_{k\vep},y}\in B_{r_{\vep}}(0))\bigg].
    \end{align}

We will now expand out all the brackets above. Let us first try to isolate the terms which are $o(\vep)$. To wit, note that if a term has a product of at least two of $\vartheta_{\vep}^{(i)}(\tilde{Z}^{(\vep)}_{w_{k\vep},y}|y^{w_{k\vep}})$ and $\vartheta^{(j)}_{\vep}(\tilde{Z}^{(\vep)}_{w_{k\vep},y}|y^{w_{k\vep}})$, then it is $o(\vep)$. 
This is because by \eqref{eq:large11}, \eqref{eq:large21}, \eqref{eq:large31}, and \eqref{eq:large41}, for $i\neq j$, we have \begin{align*}&\;\;\;\;\sup_y \big|\vartheta_{\vep}^{(i)}(\tilde{Z}^{(\vep)}_{w_{k\vep},y}|y^{w_{k\vep}})\vartheta^{(j)}_{\vep}(\tilde{Z}^{(\vep)}_{w_{k\vep},y}|y^{w_{k\vep}})\mathbf{1}(\sqrt{\vep}Z_{w_{k\vep},y}\in B_{r_{\vep}}(0))\big|\\ &\le C\vep^2\left(\log{\left(\frac{1}{\vep}\right)}\right)^{6}=o(\vep).
\end{align*}
Same goes for terms having product of some $\vartheta_{\vep}^{(i)}(\tilde{Z}^{(\vep)}_{w_{k\vep},y}|y^{w_{k\vep}})$ with either of $\frac{1}{\vep}T[u_{k\vep}:3](\tilde{Z}^{(\vep)}_{w_{k\vep},y}|y^{w_{k\vep}})$ or $T[f:1](\tilde{Z}^{(\vep)}_{w_{k\vep},y}|y^{w_{k\vep}})$. For example, by \eqref{eq:large2} and \eqref{eq:large21}, we have:
\begin{align*}
&\;\;\;\;\sup_y \bigg|\vartheta_{\vep}^{(i)}(\tilde{Z}^{(\vep)}_{w_{k\vep},y}|y^{w_{k\vep}})\frac{1}{\vep}T[u_{k\vep}:3](\tilde{Z}^{(\vep)}_{w_{k\vep},y}|y^{w_{k\vep}})\bigg|\mathbf{1}(\sqrt{\vep}Z_{w_{k\vep},y}\in B_{r_{\vep}}(0))\\ &\le C\vep^{3/2}\left(\log{\left(\frac{1}{\vep}\right)}\right)^{7/2}=o(\vep).
\end{align*}
Therefore, we can easily isolate the terms that potentially contribute $O(\vep)$ or higher. By doing this in \eqref{eq:gaussbreak}, we get that:
    
    \begin{align*}
        &\;\;\;\;\frac{\sqrt{\mathrm{det}(\nabla^2 u_{k\vep}(y^{w_{k\vep}}))}}{(2\pi\vep)^{\frac{d}{2}}}\int\limits_{B_{r_{\vep}}(y^{w_{k\vep}})} \exp\left(-\frac{1}{\vep}\mcD[u_{k\vep}](x|y)-f(x)+f(y^{w_{k\vep}})\right)\,dx\nonumber \\ 
        &=1-\E\bigg[\bigg(\frac{1}{\vep}T[u_{k\vep}:3](\tilde{Z}^{(\vep)}_{w_{k\vep},y}|y^{w_{k\vep}})+\frac{1}{\vep}T[u_{k\vep}:3](\tilde{Z}^{(\vep)}_{w_{k\vep},y}|y^{w_{k\vep}})T[f:1](\tilde{Z}^{(\vep)}_{w_{k\vep},y}|y^{w_{k\vep}})\\ &-T[f:1](\tilde{Z}^{(\vep)}_{w_{k\vep},y}|y^{w_{k\vep}})+\sum_{i=1}^4 \vartheta_{\vep}^{(i)}(\tilde{Z}^{(\vep)}_{w_{k\vep},y}|y^{w_{k\vep}})\bigg)\mathbf{1}(\sqrt{\vep}Z_{w_{k\vep},y}\in B_{r_{\vep}}(0))\bigg]+o(\vep).
    \end{align*}
    Here the $o(\vep)$ term, as argued above, is uniform in $y$. Next, we use the symmetry of Gaussians to note that 
    $$\E\left[\frac{1}{\vep}T[u_{k\vep}:3](\tilde{Z}^{(\vep)}_{w_{k\vep},y}|y^{w_{k\vep}})\mathbf{1}(\sqrt{\vep}Z_{w_{k\vep},y}\in B_{r_{\vep}}(0))\right]=0,$$
    and 
    $$\E\left[T[f:1](\tilde{Z}^{(\vep)}_{w_{k\vep},y}|y^{w_{k\vep}})\mathbf{1}(\sqrt{\vep}Z_{w_{k\vep},y}\in B_{r_{\vep}}(0))\right]=0.$$
    Therefore, the second and the fourth term above are both $0$. For the third term, note that by using \cref{asn:solcon}, part (iii), we get:
    \begin{align*}
        &\;\;\;\sup_y\E\left[\frac{1}{\vep}T[u_{k\vep}:3](\tilde{Z}^{(\vep)}_{w_{k\vep},y}|y^{w_{k\vep}})T[f:1](\tilde{Z}^{(\vep)}_{w_{k\vep},y}|y^{w_{k\vep}})\mathbf{1}(\sqrt{\vep}Z_{w_{k\vep},y}\in B_{r_{\vep}}(0))\right]\\ &\le \sup_y C\frac{1}{\vep}\E\lVert \tilde{Z}^{(\vep)}_{w_{k\vep},y}-y^{w_{k\vep}}\rVert^4\le C \vep.
    \end{align*}
    Finally, by \eqref{eq:large12}, \eqref{eq:large22}, \eqref{eq:large32}, and \eqref{eq:large42}, we get:
    $$\sup_y \sum_{i=1}^4 \E\left[\bigg|\vartheta_{\vep}^{(i)}(\tilde{Z}^{(\vep)}_{w_{k\vep},y}|y^{w_{k\vep}})\mathbf{1}(\sqrt{\vep}Z_{w_{k\vep},y}\in B_{r_{\vep}}(0))\bigg|\right]\le C\vep.$$
    Therefore, all the requisite terms are $O(\vep)$ or of a lower order. This establishes \eqref{eq:showlarge}, thereby completing the proof of \cref{lem:pmaboundmain}.
\end{proof}

In order to prove \cref{lem:erboundmain}, we consider the following proposition which will be used multiple times in the sequel. 

\begin{prop}\label{prop:contra}
Given $\phi_1, \phi_2\in \diffcont(\R^d)$, we have
$$\lVert\opV[\phi_1]-\opV[\phi_2]\rVert_{\infty}\le \lVert \phi_1-\phi_2\rVert_{\infty},\qquad \lVert \opU[\phi_1]-\opU[\phi_2]\rVert_{\infty}\le \lVert \phi_1-\phi_2\rVert_{\infty}.$$
\end{prop}

\begin{proof}
    By the variational representation of KL divergence, we have:
\begin{equation}\label{eq:dualRE}
\opV[\phi_1](y) = \vep\sup_{\nu:\ \KL{\nu}{e^{-f}}<\infty}\left[ \int \left(\frac{1}{\vep}\langle x,y\rangle - \frac{1}{\vep}\phi_1(x)\right) d\nu(x) - \KL{\nu}{e^{-f}}\right].
\end{equation}
Clearly the same representation as in \eqref{eq:dualRE} also holds with $\phi_1$ replacing $\phi_2$. Now, for any probability measure $\nu$ satisfying $\KL{\nu}{e^{-f}}<\infty$, we have:
\begin{align*}
\bigg| &\;\;\;\; \left[\int \left(\frac{1}{\vep}\langle x,y\rangle - \frac{1}{\vep}\phi_1(x)\right) d\nu(x) - \KL{\nu}{\mu}\right] \\ & -  \int \left[\int \left(\frac{1}{\vep}\langle x,y\rangle - \frac{1}{\vep}\phi_1(x)\right) d\nu(x) - \KL{\nu}{\mu}\right]\bigg| \le \norm{\phi_1 - \phi_2}_\infty. 
\end{align*}
Invoking \eqref{eq:dualRE} completes the proof for $\opV$s. The same strategy also works for $\opU$'s.
\end{proof}

We are now in position to prove the key lemma.

\vspace{0.1in} 

\emph{Proof of \cref{lem:erboundmain}.} By the uniform Lipschitz property of $\opV$s as established in \cref{prop:contra}, we have: 
$$\lVert b_k^{\vep}\rVert_{\infty}=\frac{1}{\vep}\lVert \opV[u_k^{\vep}]-\opV[u_{k\vep}]\rVert_{\infty}\le \frac{1}{\vep}\lVert u_k^{\vep}-u_{k\vep}\rVert_{\infty}=\lVert a_k^{\vep}\rVert_{\infty}.$$
Therefore, it suffices to only work with $a_k^{\vep}$s as defined in \cref{lem:erboundmain}. 
Recall that $u_{k\vep}$'s are the solutions of the PMA \eqref{eq:pma} restricted to time points $t=\vep,2\vep,\ldots$. Consider the following surrogate sequence of functions 
$\bar{u}_{k+1}^{\vep}:=\opS[u_{k\vep}],$
and define,
\begin{equation}\label{eq:approx}
R_T(\vep):=\sup_{k\vep\le T}\lVert \bar{u}_{k+1}^{\vep}-u_{(k+1)\vep}\rVert_{\infty}.
\end{equation}

With the above notation, the proof of \cref{lem:erboundmain} will proceed in two steps. Once again throughout this proof, we will use $C$ to denote a generic constant (which might change from one line to the next) free of $\vep$. 

\emph{Step (a).} We will show that 
$$R_T(\vep)\le C \vep^2.$$

\emph{Step (b).} We then show that 
$$\lVert a_k^{\vep}\rVert_{\infty}\le (k/\vep) R_T(\vep),$$
for all $k$ such that $k\vep\le T$.

By combining steps (a) and (b), we get:
$$\sup_{k:\ k\vep\le T} \lVert a_k^{\vep}\rVert_{\infty}\le \frac{T}{\vep^2}R_T(\vep)\le CT.$$

This proves \cref{lem:erboundmain}. We now focus our attention on proving steps (a) and (b). 

\emph{Proof of step (a).} Recall the time derivatives of the $u_{k\vep}$s as in the PMA~\eqref{eq:pma}. By \cref{asn:solcon}, part (ii), we have:
\begin{align}\label{eq:pmareg}
    &\;\;\;\;\sup_{k:\ k\vep\le T}\sup_x \Bigg|u_{(k+1)\vep}(x)-u_{k\vep}(x)-\vep\left(f(x)-g(x^{u_{k\vep}})+\ldet\left(\frac{\partial x^{u_{k\vep}}}{\partial x\hfill}\right)\right)\Bigg|\nonumber \\ &=\sup_{k:\ k\vep\le T}\sup_x \Bigg|u_{(k+1)\vep}(x)-u_{k\vep}(x)-\vep\frac{\partial}{\partial t}u_t(x)\big|_{t=k\vep}\Bigg|\le C\vep.
\end{align}


Next, we will show the following:
\begin{equation}\label{eq:stepa1}
    \sup_{k: \ k\vep\le T}\sup_x \bigg|\exp\left(\frac{\bar{u}_{k+1}^{\vep}(x)-u_{k\vep}(x)}{\vep}-\left(f(x)-g(x^{u_{k\vep}})+\ldet\left(\frac{\partial x^{u_{k\vep}}}{\partial x\hfill}\right)\right)\right)-1\bigg|\le C\vep,
\end{equation}
By taking logarithms in \eqref{eq:stepa1} and combining it with \eqref{eq:pmareg}, we easily get Step (a). Therefore, it only remains to prove \eqref{eq:stepa1}.

\emph{Proof of \eqref{eq:stepa1}.} First let us define 
$$\vartheta_{k\vep}(y):=\vep^{-1}\left(\opV[u_{k\vep}](y)-w_{k\vep}(y)-\frac{\vep d}{2}\log{(2\pi\vep)}+\vep f(y^{w_{k\vep}})-\frac{\vep}{2}\ldet\left(\frac{\partial y^{w_{k\vep}}}{\partial y\hfill}\right)\right).$$
The following estimate is an immediate consequence of \cref{lem:pmaboundmain}.
\begin{align}\label{eq:estcon}
    \sup_y \bigg|\exp\left(\frac{\vartheta_{k\vep}(y)}{\vep}\right)-1\bigg|\le C\vep.
\end{align}
We also note the following algebraic identity:
\begin{align*}
    &\;\;\;\;\exp\left(\frac{\bar{u}_{k+1}^{\vep}(x)-u_{k\vep}(x)}{\vep}-\left(f(x)-g(x^{u_{k\vep}})+\ldet\left(\frac{\partial x^{u_{k\vep}}}{\partial x\hfill}\right)\right)\right)\\ &=\frac{\sqrt{\mathrm{det}\left(\frac{\partial x\hfill}{\partial x^{u_{k\vep}}\hfill}\right)}}{(2\pi\vep)^{\frac{d}{2}}}\int \exp\Bigg(\frac{1}{\vep}\langle x,y\rangle-\frac{1}{\vep}w_{k\vep}(y)-\frac{1}{\vep}u_{k\vep}(x)+f(y^{w_{k\vep}})-f(x)-g(y)+g(x^{u_{k\vep}})\\ &\qquad -\frac{1}{2}\ldet\left(\frac{\partial y^{w_{k\vep}}}{\partial y\hfill}\right) + \frac{1}{2}\ldet\left(\frac{\partial x\hfill}{\partial x^{u_{k\vep}}\hfill}\right)\Bigg)\exp\left(-\frac{\vartheta_{k\vep}(y)}{\vep}\right)\,dy
\end{align*}
By combining the above display with \eqref{eq:estcon}, we have:
\begin{align*}
    &\;\;\;\;\sup_{k:\ k\vep\le T}\sup_x \Bigg|\exp\left(\frac{\bar{u}_{k+1}^{\vep}(x)-u_{k\vep}(x)}{\vep}-\left(f(x)-g(x^{u_{k\vep}})+\ldet\left(\frac{\partial x^{u_{k\vep}}}{\partial x\hfill}\right)\right)\right)\\ &\qquad\qquad - \frac{\sqrt{\mathrm{det}\left(\frac{\partial x\hfill}{\partial x^{u_{k\vep}}}\right)}}{(2\pi\vep)^{\frac{d}{2}}}\int \exp\Bigg(\frac{1}{\vep}\langle x,y\rangle-\frac{1}{\vep}w_{k\vep}(y)-\frac{1}{\vep}u_{k\vep}(x)+f(y^{w_{k\vep}})-f(x)\\ &\qquad\qquad -g(y)+g(x^{u_{k\vep}})-\frac{1}{2}\ldet\left(\frac{\partial y^{w_{k\vep}}}{\partial y\hfill}\right) + \frac{1}{2}\ldet\left(\frac{\partial x\hfill}{\partial x^{u_{k\vep}}\hfill}\right)\Bigg)\,dy\Bigg|\\ &\le (C\vep) \sup_{k:\ k\vep\le T}\sup_x \frac{\sqrt{\mathrm{det}\left(\frac{\partial x\hfill}{\partial x^{u_{k\vep}}}\right)}}{(2\pi\vep)^{\frac{d}{2}}}\int \exp\Bigg(\frac{1}{\vep}\langle x,y\rangle-\frac{1}{\vep}w_{k\vep}(y)-\frac{1}{\vep}u_{k\vep}(x)+f(y^{w_{k\vep}})-f(x)\\ &\qquad -g(y)+g(x^{u_{k\vep}})-\frac{1}{2}\ldet\left(\frac{\partial y^{w_{k\vep}}}{\partial y\hfill}\right) + \frac{1}{2}\ldet\left(\frac{\partial x\hfill}{\partial x^{u_{k\vep}}\hfill}\right)\Bigg)\,dy.
\end{align*}
Therefore in order to prove \eqref{eq:stepa1}, it suffices to prove: 
\begin{align*}
    &\;\;\;\;\sup_{k:\ k\vep\le T}\sup_x\bigg|\frac{\sqrt{\mathrm{det}\left(\frac{\partial x\hfill}{\partial x^{u_{k\vep}}}\right)}}{(2\pi\vep)^{\frac{d}{2}}}\int \exp\Bigg(\frac{1}{\vep}\langle x,y\rangle-\frac{1}{\vep}w_{k\vep}(y)-\frac{1}{\vep}u_{k\vep}(x)+f(y^{w_{k\vep}})-f(x)\nonumber\\ &\qquad -g(y)+g(x^{u_{k\vep}})-\frac{1}{2}\ldet\left(\frac{\partial y^{w_{k\vep}}}{\partial y\hfill}\right) + \frac{1}{2}\ldet\left(\frac{\partial x\hfill}{\partial x^{u_{k\vep}}\hfill}\right)\Bigg)\,dy-1\bigg|\le C\vep.
\end{align*}
Once again by taking logarithms on both sides above, it suffices to show:
\begin{align}\label{eq:stepa2}
&\sup_{k:\ k\vep\le T} \sup_x \bigg|\log\int \exp\left(\frac{1}{\vep}\langle x,y\rangle-\frac{1}{\vep}w_{k\vep}(y)+f(y^{w_{k\vep}})-g(y)-\frac{1}{2}\ldet\left(\frac{\partial y^{w_{k\vep}}}{\partial y\hfill}\right)\right)\,dy\nonumber \\ &\qquad - \frac{\vep d}{2}\log{(2\pi\vep)} - f(x) + g(x^{u_{k\vep}})-\ldet\left(\frac{\partial x^{u_{k\vep}}}{\partial x\hfill}\right)\bigg|\le C\vep. 
\end{align}
Now $$\widetilde{\opV}[w_{k\vep}](x):=\vep\log\int \exp\left(\frac{1}{\vep}\langle x,y\rangle-\frac{1}{\vep}w_{k\vep}(y)+f(y^{w_{k\vep}})-g(y)-\frac{1}{2}\ldet\left(\frac{\partial y^{w_{k\vep}}}{\partial y\hfill}\right)\right)\,dy$$
has the same form as $\opV[u_{k\vep}]$ with $u_{k\vep}$ replaced by $w_{k\vep}$ and $f$ replaced by the function $\tilde{f}$ given by 
$$\tilde{f}(y)=-f(y^{w_{k\vep}})+g(y)+\frac{1}{2}\ldet\left(\frac{\partial y^{w_{k\vep}}}{\partial y\hfill}\right).$$
By \cref{asn:solcon}, part (iii), $\tilde{f}$ satisfies the same assumptions as $f$ required in \cref{lem:pmaboundmain}. 
Therefore, reworking the same proof as \cref{lem:pmaboundmain} with $\tilde{f}$, we get:
$$\sup_{k:\ k\vep\le T}\sup_x \bigg| \widetilde{\opV}[w_{k\vep}](x)-\frac{\vep d}{2}\log{(2\pi\vep)}+\vep \tilde{f}(x^{u_{k\vep}})-\frac{\vep}{2}\ldet\left(\frac{\partial x^{u_{k\vep}}}{\partial x\hfill}\right)\bigg|\le C\vep^2.$$
As $\tilde{f}(x^{u_{k\vep}})=-f(x)+g(x^{u_{k\vep}})-\frac{1}{2}\ldet\left(\frac{\partial x^{u_{k\vep}}}{\partial x\hfill}\right)$, the above display establishes \eqref{eq:stepa2}. This establishes \eqref{eq:stepa1} and consequently, also establishes Step (a).

\vspace{0.1in}

\emph{Proof of Step (b).} To establish step (b), the crucial tool will be \cref{prop:contra}. To wit, note that for all $k$ such that $k\vep\le T$, we have
\begin{align*}
    \lVert a_k^{\vep}\rVert_{\infty}&\le \frac{1}{\vep}\lVert u_{k}^{\vep}-\bar{u}_{k}^{\vep}\rVert_{\infty}+\frac{1}{\vep} \lVert \bar{u}_{k}^{\vep}-u_{k\vep}\rVert_{\infty}\\ &\le \frac{1}{\vep}\lVert \opS[u_{k-1}^{\vep}]-\opS[u_{(k-1)\vep}]\rVert_{\infty}+\frac{1}{\vep}R_t(\vep)\\ &\le \lVert a_{k-1}^{\vep}\rVert_{\infty}+\frac{1}{\vep}R_t(\vep).
\end{align*}
The last inequality follows by noting that $\opS$ is $1$-Lipschitz in the uniform norm which in turn follows from the fact that $\opS=\opU\circ\opV$ and both $\opU$, $\opV$ are $1$-Lipschitz in the uniform norm by \cref{prop:contra}. Now, in order to prove the conclusion in Step (b), we will use the above inequality recursively, i.e., 
$$\lVert a_k^{\vep}\rVert_{\infty}\le \lVert a_{k-1}^{\vep}\rVert_{\infty}+\frac{1}{\vep}R_t(\vep)\le \lVert a_{k-2}^{\vep}\rVert_{\infty}+\frac{2}{\vep}R_t(\vep)\le \ldots \le \lVert a_0^{\vep}\rVert_{\infty}+\frac{k}{\vep}R_T(\vep).$$
Now, by definition, $a_0^{\vep}=\vep^{-1}(u_0^{\vep}-u_0)=0$ as we use the same initializer for the Sinkhorn algorithm \eqref{eq:sinkupdt} and the PMA \eqref{eq:pma}. This implies, using the above display coupled with step (a), 
$$\sup_{k:\ k\vep\le T}\lVert a_k^{\vep}\rVert_{\infty}\le \sup_{k:\ k\vep\le T} \frac{k}{\vep}R_T(\vep)\le \sup_{k:\ k\vep\le T} \frac{k}{\vep}\cdot C\vep^2\le CT.$$
This establishes step (b).

\begin{longtable}{|p{2cm}|p{9.8cm}|}
\caption{Notation chart}
\label{tab:table3}\\
\hline
\textbf{Notation} & \textbf{Meaning} \\
\hline
$\mathbb{N}$ & Set of natural numbers \\
\hline
$[n]$, $n\in\mathbb{N}$ & The set $\{1,2,\ldots ,n\}$ \\
\hline
$\lmn(A)$ & The minimum eigenvalue of a square matrix $A$ \\
\hline 
$\lmx(A)$ & The maximum eigenvalue of a square matrix $A$ \\
\hline 
$\trc(A)$ & The trace of a square matrix $A$ \\
\hline 
$\lVert A\rVert_{\hs}$ & The Frobenius norm of a square matrix $A$ \\ 
\hline 
$\lVert A\rVert_{\mathrm{op}}$ & The $L^2$ operator norm of a square matrix $A$\\
\hline
$\sqrt{A}$ & Cholesky square root of a symmetric and positive definite matrix $A$, i.e., $\sqrt{A}=S$ if and only if $A=S^2$
\\ 
\hline 
$A\preceq B$ & $B-A$ is non-negative definite \\ 
\hline
$I_d$ & The $d\times d$ identity matrix. We will drop the $d$ from the subscript when the dimension is obvious\\ 
\hline 
$\mathrm{Id}$ & The identity function on $\R^d$\\ \hline 
$|x|$ & The Euclidean norm of a vector $x$ \\ 
\hline 
$T_{\#}\mu$ & Given a probability measures $\mu$ on $\R^d$ and a function $T:\R^d \rightarrow \R^d$, this is the push-forward of $\mu$ by $T$, i.e., the probability distribution of $T(X)$ where $X\sim\mu$ \\ \hline 
$\px \gamma$ & The $X$ marginal density of the joint density $\gamma$ on $\R^d\times \R^d$\\ \hline 
$\py \gamma$ & The $Y$ marginal density of the joint density $\gamma$ on $\R^d\times \R^d$
\\ \hline 
$p_{X|Y}\gamma(\cdot|\cdot)$ & The conditional density of $X$ given $Y$ under the joint density $\gamma$\\ \hline 
$p_{Y|X}\gamma(\cdot|\cdot)$ & The conditional density of $Y$ given $X$ under the joint density $\gamma$\\ \hline 
$\diffcont(\R^d)$ & The space of continuous functions on $\R^d$\\ \hline 
$\diffcont^k(\R^d)$ & The space of $k$ times continuously differentiable functions on $\R^d$ \\ \hline 
$\diffcont^{k,\ell}(\R^d)$ & The space of functions on $[0,\infty)\times \R^d$ (time $\times$ space) with uniformly continuous mixed derivatives up to order $k$ in time and $\ell$ in space \\ \hline 
$B_r(x)$ & The Euclidean ball centered at $x$ with radius $r$ \\ \hline
$\div$ & The Divergence operator\\ \hline 
$\bmd$ & The Laplacian operator\\ \hline
$\nabla$ or $\frac{\partial}{\partial x}$ & The gradient or partial derivative with respect to the usual coordinate chart $x$\\ \hline 
$\nabla^r$ & The $r$-th order multi-derivative with respect to the coordinate chart $x$\\ \hline 
$\int$ & The integral notation without any domain specified will always imply that the integral is over $\R^d$\\ \hline
$\nabla^{-2}$ & The inverse of the Hessian matrix with respect to the coordinate chart $x$\\ \hline 
$\frac{\delta}{\delta \rho}$ & The first variation of a function with respect to a probability measure $\rho$\\ \hline
$\lVert \cdot\rVert_{L^2(\rho)}$ & The $L^2$ norm of a function computed with respect to a probability measure $\rho$\\ \hline
\end{longtable}

\bibliographystyle{amsalpha}
\bibliography{references}
\end{document}